\documentclass[11pt]{article} 
\usepackage{enumerate}

\usepackage[%dvips,
            CJKbookmarks=true,
            bookmarksnumbered=true,
            bookmarksopen=true,
                       bookmarks=true,
            colorlinks=true,
            citecolor=red,
            linkcolor=blue,
            anchorcolor=red,
            urlcolor=blue
            ]{hyperref}

\usepackage{url, color, fullpage,xcolor}
\usepackage{amsmath,amssymb,amsfonts,amsthm,bbm,graphicx,subfigure,xspace}
\usepackage{tcolorbox}
\definecolor{myblue}{HTML}{0000FF}
\definecolor{myred}{HTML}{FF0000}
\definecolor{myorange}{HTML}{FFA500}
\definecolor{myyellow}{HTML}{FFFF00}
\definecolor{mycyan}{HTML}{00FFFF}
\definecolor{mygreen}{HTML}{008000}
\definecolor{mybrown}{HTML}{A52A2A}
\renewcommand{\Pr}{\Prob}

\newcommand{\mul}[1]{N_{#1}}

\newcommand{\prev}[1]{{\varPhi_{#1}}}

\newcommand{\poi}{\mathrm{poi}}
\newcommand{\indic}{\mathbbm{1}}
\newcommand{\Var}{\mathrm{Var}}
\newcommand{\Expect}{\mathbb{E}}

\newcommand{\coeffs}[1]{h^{#1}}
\newcommand{\coeffET}{\coeffs{\rm ET}}

\newcommand{\supp}{S}
\newcommand{\suppest}{\hat{\supp}}
\newcommand{\ed}{\stackrel{\mathrm{def}}{=}}

\newcommand{\EE}{\mathbb{E}}

\newcommand{\ZZ}{\mathbb{Z}}
\newcommand{\Zplus}{\ZZ_{+}}

\newcommand{\cO}{O}

\newcommand{\cX}{\mathcal{X}}
\newcommand{\cE}{\mathcal{E}}

\newcommand{\iid}{\emph{i.i.d.\xspace}}
\newcommand{\ie}{\emph{i.e., }}
\newcommand{\eg}{\emph{e.g., }}

\newtheorem{Theorem}{Theorem}
\newtheorem{Corollary}{Corollary}
\newtheorem{Lemma}{Lemma}
\newtheorem{remark}{Remark}

\newcommand{\ignore}[1]{}%

\newcommand{\upto}{,\ldots,}

\newcommand{\sets}[1]{\{#1\}}
\newcommand{\Paren}[1]{\left(#1\right)}

\newcommand{\Sold}{S_{\text{old}}}
\newcommand{\Snew}{S_{\text{new}}}

\newcommand{\Th}{^\mathrm{th}}
\newcommand{\integers}{\mathbb{Z}}
\newcommand{\naturals}{\mathbb{N}}

\newcommand{\Prob}{\mathbb{P}}
\newcommand{\prob}[1]{\mathbb{P}\left(#1\right)}

\newcommand{\Binom}{\mathrm{Bin}}
\newcommand{\hyper}{\mathrm{Hyp}}

\newcommand{\pth}[1]{\left( #1 \right)}
\newcommand{\qth}[1]{\left[ #1 \right]}
\newcommand{\sth}[1]{\left\{ #1 \right\}}

\usepackage{prettyref}
\newrefformat{eq}{(\ref{#1})}
\newrefformat{thm}{Theorem~\ref{#1}}
\newrefformat{th}{Theorem~\ref{#1}}
\newrefformat{chap}{Chapter~\ref{#1}}
\newrefformat{sec}{Section~\ref{#1}}
\newrefformat{algo}{Algorithm~\ref{#1}}
\newrefformat{fig}{Fig.~\ref{#1}}
\newrefformat{tab}{Table~\ref{#1}}
\newrefformat{rmk}{Remark~\ref{#1}}
\newrefformat{clm}{Claim~\ref{#1}}
\newrefformat{def}{Definition~\ref{#1}}
\newrefformat{cor}{Corollary~\ref{#1}}
\newrefformat{lmm}{Lemma~\ref{#1}}
\newrefformat{prop}{Proposition~\ref{#1}}
\newrefformat{pr}{Proposition~\ref{#1}}
\newrefformat{app}{Appendix~\ref{#1}}
\newrefformat{apx}{Appendix~\ref{#1}}
\newrefformat{ex}{Example~\ref{#1}}
\newrefformat{exer}{Exercise~\ref{#1}}
\newrefformat{soln}{Solution~\ref{#1}}

\newcommand{\psum}{p_{_S}}

\newcommand{\dTV}{d_{\rm TV}}

\newcommand{\UE}{U^{\scriptscriptstyle\rm E}}

\newcommand{\xn}{x^n}
\newcommand{\Xn}{X^n}

\newcommand{\XNPOM}{X_{N+1}^{N+M}}
\newcommand{\Xnpom}{X_{n+1}^{n+m}}

\newcommand{\Us}[1]{U^{\scriptscriptstyle\rm #1}}
\newcommand{\Usss}[3]{U^{\scriptscriptstyle\rm #1}(#2,#3)}

\newcommand{\UXnm}{\Usss{}{\Xn}{\Xnpom}} % \Unnp
\newcommand{\UXNM}{\Usss{}{X^N}{\XNPOM}} % \Unn

\newcommand{\UEXnm}{\Usss{E}{\Xn}{t}} % \UEnnp %% changed m to t

\newcommand{\UEXNm}{\Usss{E}{X^N}{t}} % \UEnnp  %5 changed m to t
\newcommand{\UGT}{\Us{GT}} % \UGTnnp
\newcommand{\UGTXnt}{\Usss{GT}{\Xn}{t}}
\newcommand{\aUXnm}{U} % \Unnp
\newcommand{\aUEXnm}{\Us{E}} % \UEnnp

\newcommand{\aUGT}{\Us{GT}} % \UGTnnp
\newcommand{\aUGTXnm}{\aUGT} % \UGTnnp
\newcommand{\aUGTxnm}{\aUGT} % \UGTnnp

\newcommand{\LEnt}{\cE_{n,t}(U^{\scriptscriptstyle\rm E})}
\newcommand{\LELnt}{\cE_{n,t}(U^{\scriptscriptstyle\rm L})}
\newcommand{\LElnt}{\cE_{n,t}(U^{\scriptscriptstyle \ell})}
\newcommand{\LEGTnt}{\cE_{n,t}(U^{\scriptscriptstyle \rm GT})}

\newcommand{\U}{U}
\newcommand{\Ut}{U}
\newcommand{\Up}{U}
\newcommand{\Uh}{U^{\scriptscriptstyle\rm h}}
\newcommand{\Uht}{U^{\scriptscriptstyle\rm h}}

\newcommand{\Ugm}{\hat{U}}

\newcommand{\binlmt}{k}
\newcommand{\mm}{m}

\newcommand{\LJnm}{L^{\scriptscriptstyle\rm J}_{n}(m)}

\newcommand{\MJnd}{M^{\scriptscriptstyle\rm J}_n(\delta)}

\newcommand{\UET}{U^{\scriptscriptstyle\rm ET}}

\newcommand{\UJp}{U^{\scriptscriptstyle\rm J}}

\newcommand{\Urand}{U^{\scriptscriptstyle\rm \rand}}

\newcommand{\hrand}{h^{\scriptscriptstyle\rm \rand}}

\newcommand{\rand}{L}
\newcommand{\fixed}{\ell}
\newcommand{\diff}{\text{d}}

\newcommand{\hFGT}{h^{\scriptscriptstyle\rm GT}}

\newcommand{\DLt}{\xi_{\rand}(t)}
\ignore{

}

\newcommand{\UL}{U^{\scriptscriptstyle\rm L}}
\newcommand{\hypN}{R}
\newcommand{\hypn}{r}

\newcommand{\ntotal}{n_{\text{total}}}
\newcommand{\hL}{h^{\scriptscriptstyle\rm L}}

\newcommand{\arxiv}[2]{#1}
\title{Estimating the number of unseen species: \\
A bird in the hand is worth $\log n $ in the bush}
\author{
\begin{tabular}[t]{c@{\extracolsep{5em}}c@{\extracolsep{5em}}c} 
 Alon Orlitsky & Ananda Theertha Suresh & Yihong Wu\\
UCSD & UCSD & UIUC \\
\small\texttt{alon@ucsd.edu} & \small\texttt{asuresh@ucsd.edu} 
& \small\texttt{yihongwu@illinois.edu} \\
\end{tabular}
\vspace{2ex}
}
\begin{document}
\maketitle
\begin{abstract}
Estimating the number of unseen species is an important problem in
many scientific endeavors.  Its most popular formulation, introduced
by Fisher, uses $n$ samples to predict the number $U$ of hitherto
unseen species that would be observed if $t\cdot n$ new samples were
collected. Of considerable interest is the largest
ratio $t$ between
the number of
new and existing samples
for which $U$ can be accurately predicted.

In seminal works, Good and Toulmin constructed an intriguing estimator that
predicts $U$ for all $t\le 1$, thereby showing that the
number of species can be estimated for a population twice
as large as that observed. Subsequently Efron and Thisted obtained a modified estimator
that empirically predicts $U$ even for some $t>1$,
but without provable guarantees.

We derive a class of estimators that \emph{provably} predict
$U$ not just for constant $t>1$, but all the way up to 
$t$ proportional to $\log n$. 
This shows that the number of species can be estimated for
a population $\log n$ times larger than that observed, a
factor that grows arbitrarily large as $n$ increases.
We also show that this range 
is the best possible and that the
estimators' mean-square error is 
optimal up to constants for any $t$.
Our approach yields the first provable guarantee for the Efron-Thisted estimator and, in addition, a variant which achieves stronger theoretical and
experimental performance
than existing methodologies
on a variety of synthetic and real datasets.

The estimators we derive are simple linear estimators that are 
computable in time proportional to $n$. 
The performance guarantees hold uniformly for all distributions,
and apply to all four standard sampling models
commonly used across various scientific disciplines:
multinomial, Poisson, hypergeometric, and Bernoulli product.
\end{abstract}

\newpage
\tableofcontents

%\section{Introduction}
%\label{sec:introduction}
%\subsection{Background}
%Population estimation is an important problem in numerous
%scientific disciplines, including ecology~\cite{C84, CL92, BF93, CCG12}
%bacteriology~\cite{KPD99,PBG01}, databases~\cite{HNSS95}, and
%linguistics~\cite{ET76, TE87}.
\section{Introduction}
\label{sec:intro}
Species estimation is an important problem in numerous scientific
disciplines.  Initially used to estimate ecological
diversity~\cite{C84,CL92,BF93,CCG12}, it was subsequently applied to
assess vocabulary size~\cite{ET76,TE87}, 
database attribute variation~\cite{HNSS95},
and 
password innovation~\cite{FH07}.
Recently it has found a number of
%genomic
bio-science applications including estimation of 
bacterial and microbial diversity~\cite{KPD99,PBG01,HHJTB01,GCB07},
immune 
%T-cell 
receptor diversity~\cite{R09},
and unseen genetic variations~\cite{ICL09}.

All approaches to the problem incorporate a statistical model, 
with the most popular being the \emph{extrapolation model}
introduced by Fisher, Corbet, and Williams~\cite{FCW43} in 1943.
%It considers $n$ samples $X^n=X_1\upto X_n$ generated by
%an unknown \iid distribution $p$, and assuming that $p$
%also generates $m$ additional samples $X_{n+1}^{n+m}=X_{n+1}\upto X_{n+m}$,
%it calls for estimating
%absent
It assumes that $n$ independent samples $X^n\ed X_1\upto X_n$ were collected from
an unknown distribution $p$, and calls for estimating
\[
U
%\aUXnm 
\ed
\UXnm
\ed
\left|\sets{X_{n+1}^{n+m}}\backslash\sets{\Xn}\right|,
\]
the number of hitherto unseen symbols that would be observed 
if $m$ additional samples $X_{n+1}^{n+m} \ed X_{n+1}\upto X_{n+m}$,
were collected from the same distribution.
%the number of hitherto unseen symbols that will be observed among
%the new samples.
%For notational simplicity, let  $\aUXnm \ed \UXnm$.

In 1956,
% I.J. Good and G.H. Toulmin~\cite{GT56} predicted $\UXnm$ by
Good and Toulmin~\cite{GT56} predicted $\aUXnm$ by
a fascinating estimator
%$\UGTnnp$
that has since intrigued statisticians and a
broad range of scientists alike~\cite{K86}.
%For example, around the year 2000, Stanford University's
%Statistics Department published a brochure where 
%several of its preeminent faculty described what
%attracted them to statistics~\cite{stanford}.
%preeminent statistician
%Following is a slightly abbreviated version of
%Bradley Efron credited the problem and its elegant
%solution with kindling his interest in statistics. 
For example, in the Stanford University Statistics Department
brochure~\cite{stanfordF}, published in the early 90's
and slightly abbreviated here,
%preeminent statistician
Bradley Efron credited the problem and its elegant solution with
kindling his interest in statistics. 
As we shall soon see, Efron, along with Ronald Thisted,
went on to make significant contributions to this problem. 
%\begin{figure*}[h]
%\centering
%\includegraphics[scale = 0.6]{References/Stanford-8.png}
%\caption{Reproduced from Stanford statistics brouchure \cite{stanford}}
%\end{figure*}

\begin{tcolorbox}
%From the time I was a little boy till my senior year in college I
%wanted to be a mathematician. Then I learned that I really wanted to
%be a 19th century mathematician, the kind who does a little theory, a
%lot of computation, and some consulting with real scientists. The
%field of statistics has allowed me to do all three things in whatever
%proportions I desired.
%Here is an example of the three faces of statistics, done in the early
%1940's.\\[1mm]
In the early 1940's, naturalist Corbet had spent two years trapping butterflies in Malaya. At the
end of that time he constructed a table (see below) to show how
many times he had trapped various butterfly species. For example,
118 species were so rare that Corbet had trapped only
one specimen of each, 74 species had been trapped twice each, etc.\\[2mm]
%\end{tcolorbox}
%\begin{table*}[h]
%\begin{center}
\centerline{
\begin{tabular}{|c|c|c|c|c|c|c|c|c|c|c|c|c|c|c|c|} \hline
Frequency & 1 & 2 & 3 & 4 & 5 & 6 & 7 & 8 & 9 & 10 & 11 & 12 & 13 & 14
& 15\\ \hline
Species  & 118& 74& 44& 24& 29& 22& 20& 19& 20& 15& 12& 14& 6& 12&
6\\ \hline
\end{tabular}
%\end{center}
}\\[1mm]
%%\caption{Corbet's data on how often butterfly species were trapped}
%\end{table*}
\indent Corbet returned to England with his table, and asked R.A. Fisher, the
greatest of all statisticians, how many new species he would see if he
returned to Malaya for another two years of trapping. This question
seems impossible to answer, since it refers to a column of Corbet's
table that doesn't exist, the ``0'' column. Fisher provided an
interesting answer that was later improved on [by Good and Toulmin].
The number of new species you can expect to see in two years of
additional trapping is
\[
118-74+44-24+ \ldots -12+6 = 75.
\]
\end{tcolorbox}

This example evaluates the Good-Toulmin estimator for the special case
where the original and future samples are of equal size, namely $\mm=n$.
To describe the estimator's general form we need only 
a modicum of nomenclature.

The \emph{prevalence} $\prev i\ed\prev i(\Xn)$ 
of an integer $i\ge0$ in $\Xn$ is the number of
symbols appearing $i$ times in $\Xn$.
For example, for $X^7$=\emph{bananas}, $\prev1=2$ and
$\prev2=\prev3=1$, and in Corbet's table, $\prev1=118$ and $\prev2=74$.
Let $t \ed \frac{m}{n}$ be the ratio of the number of future
and past samples so that $m=tn$. %\input contributions.tex
%$\Xnpom=\Xnpotn$.
Good and Toulmin
%~\cite{GT56}
estimated $U$
%$\aUXnm\ed\UXntn$  
by the
surprisingly simple formula
\begin{equation}
\label{eq:GT}
\aUGTXnm
\ed 
\UGTXnt
\ed
-\sum_{i=1}^\infty \Paren{-t}^i\prev{i}.%(\xn).
\end{equation}
They showed that for all $t\leq 1$, $\aUGTxnm$ is nearly unbiased,
%For every distribution $p$ and sample sizes $n$, $\mm$, 
%taking expectation over the sequences $\Xnpm \sim p$,
%\nbr{this cannot be true for all $n,m$!}
%%\[
%%\EE_{\Xn\sim p}\UGTXnm
%%\approx
%%\EE_{\Xnpm\sim p}\UXnm.
%%\]
%\[
%\EE[\aUGTXnm]
%\approx
%\EE[ \aUXnm].
%\]
and that while $\aUXnm$ can be as high as $nt$,\footnote{For $a,b >0$, denote $a \lesssim b$ or $b \gtrsim a$ if $\frac{a}{b} \leq c $ 
 for some universal constant $c$. Denote $a \asymp b$ if both $a \lesssim b$ and $a \gtrsim b$.}
%\begin{equation}
%\label{eq:GTguarantee}
\[
\EE
%_{\Xnpm\sim p}
(\aUGTXnm - \aUXnm)^2 \lesssim nt^2,
\]
%\end{equation}
hence in expectation,
$\aUGTxnm$ approximates $\aUXnm$ to within just $\sqrt{n}t$.
\arxiv{Figure~\ref{fig:GT_basic}}{(Fig. 1)} shows that for the ubiquitous
Zipf distribution, $\aUGTxnm$ indeed approximates $\aUXnm$ well for
all $t<1$.
\begin{figure}[ht]
\centering
\arxiv{\subfigure[]{\includegraphics[scale=0.5]{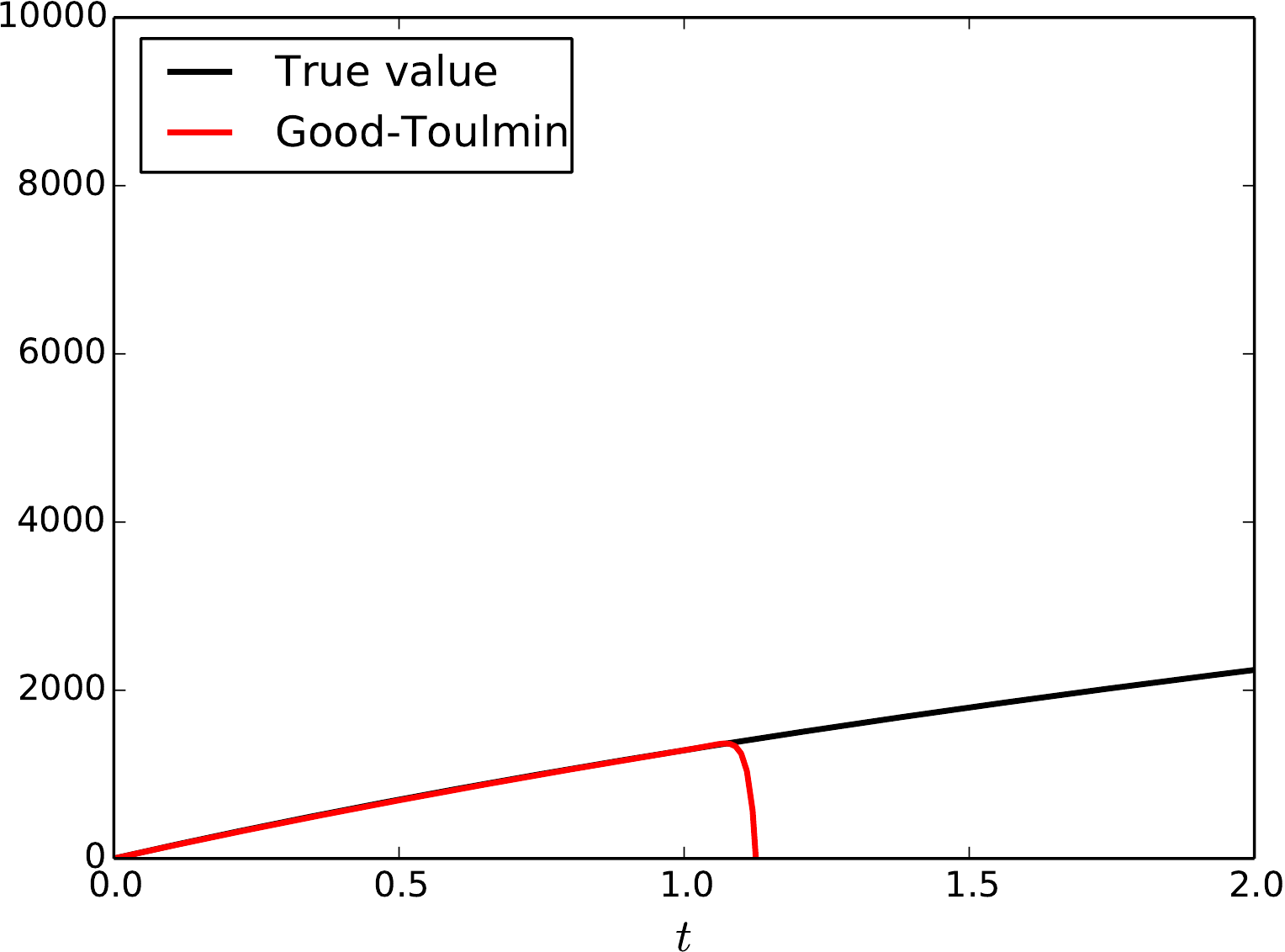}}}{\subfigure[]{\includegraphics[scale=0.5]{FGT_zipf-crop.pdf}}}
\arxiv{\subfigure[]{\includegraphics[scale=0.5]{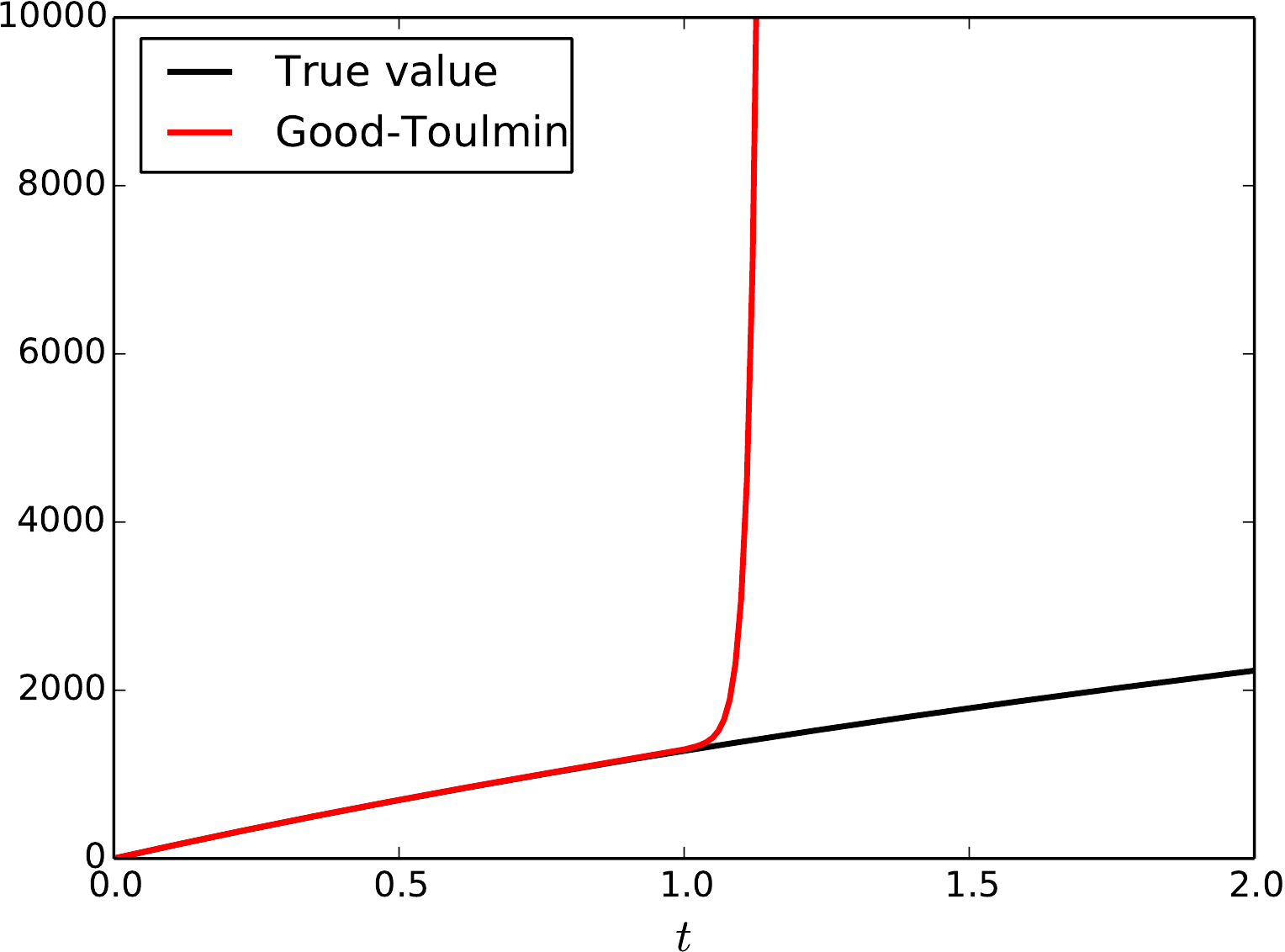}}}{\subfigure[]{\includegraphics[scale=0.5]{FGT_zipf_2-crop.pdf}}}
\caption{GT estimate as a function of $t$ for two realizations random samples of
size $n=5000$ generated by a Zipf distribution $p_i \propto 1/(i+10)$
for $1\le i\le 10000$.}
\label{fig:GT_basic}\end{figure}
Naturally, we would like to estimate $\aUXnm$ for as large a $t$ as
possible.
However, as $t>1$ increases, $\aUGTxnm$ grows as 
$(-t)^i\prev{i}$ for
the largest $i$ such that $\prev i>0$.  Hence whenever any symbol
appears more than once, $\aUGTxnm$ grows super-linearly in $t$, 
eventually far exceeding $\aUXnm$ that grows at most linearly in
$t$.
\arxiv{Figure~\ref{fig:GT_basic}}{(Fig. 1)} also shows that for the same Zipf distribution,
for $t>1$ indeed  $\aUGTxnm$ does not approximate $\aUXnm$ at all.

%will therefore have high variance, and will not be able to accurately approximate
%$\aUXnm$ that grows only linearly in $t$.

To predict $\aUXnm$ for $t>1$, Good and Toulmin~\cite{GT56}
suggested using the \emph{Euler transform}~\cite{AS64}
that converts an alternating series into another series with
the same sum, and heuristically often converges faster.
%a heuristic that re-parametrizes a series $\sum\prev{i}t^i$ by 
%$u=2t/(1+t)$
%Once transformed, the series can be evaluated for
%fewer terms, hopefully decreasing the estimate's variance and
%increasing its accuracy. 
%Following up on his early interest in the problem, Efron, along with
%Thisted~
%Efron-Thisted~\cite{ET76}, generalized $\UGT$ to other linear estimators.
%For a sequence $h\ed h_1,h_2,\ldots$, they defined
%\[
%\aUhxnm \ed \Uhxnm
%=
%\sum_{i=1}^n h_i\cdot \prev{i}.%(\xn).
%\]
Interestingly, Efron and Thisted~\cite{ET76} showed that when the
Euler transform of $\UGT$ is truncated after $k$ terms,
%is capped at some index $\binlmt\in\NN$
%and evaluated back in terms of $t^i$, it 
it can be expressed as another simple linear estimator, 
\begin{equation*}
\UET
\ed
\sum_{i=1}^n \coeffET_i\cdot \prev{i},%(\xn).
\label{eq:UET}
\end{equation*}
where
\begin{equation*}
\label{eq:coeffET}
\coeffET_i
\ed
-\Paren{-t}^{i}\cdot\prob{\Binom\Big(\binlmt, \frac{1}{1+t}\Big) \geq i},
\end{equation*}
and
\[
\prob{\Binom\Paren{\binlmt,\frac{1}{1+t}} \geq i} 
=
\begin{cases}
\sum_{j=i}^{\binlmt}
\binom{\binlmt}{j}\frac{t^{\binlmt-j}}{(1+t)^{\binlmt}}
& i \leq \binlmt,\\
0 & i > \binlmt,
\end{cases}
\]
%which leads to a linear estimator with the following coefficients obtained in~\cite{ET76}
%Motivated by the Euler's transform technique of \cite{GT56}, Efron
%and Thisted defined
is the \emph{binomial} tail probability that decays with $i$,
thereby moderating the rapid growth of $(-t)^i$.
%Note that as $i$ increases, the tail probability decays, moderating
%the rapid growth of $t^i$.

%Efron and Thisted~\cite{ET76} 
%\ignore{
%They also used these estimators to evaluate Shakespeare's vocabulary.
%The Bard wrote $n=884,647$ words, of which
%31,534 were distinct.  Using $\aUGTXnm$, they argued that had he
%continued to produce the same volume of work, namely
%another $884,647$ words, he would have used $11,430$ new words.
%For $t=10$, they 
%\ignore{used three estimators: $\UET$ described above, a
%\emph{linear-programming} estimator they derived in the same paper,
%and a Bayesian estimator presented in Fisher's original paper
%\cite{FCW43}.  All three methods suggested} 
%used $\UET$ to suggest that had Shakespeare
%written $10n=8,846,470$ more words, at least 35,000 of them would
%be new.
%}
%However, no theoretical guarantees have been shown for these results. 

Over the years, $\UET$ has been used by numerous researchers in a variety
of scenarios and a multitude of applications. 
Yet despite its wide-spread use and robust empirical results,
no provable guarantees have been established for its performance
%It was observed to perform well in practice.  For example for
%$m=1.5n$, ~\cite{ICL09} showed that .  However, no guarantees have
%been provided for its performance,
or that of any related estimator when $t>1$.  
%Consequently, beyond $t=1$, researchers using these estimators could 
%not be certain of their results.
The lack of theoretical understanding, has also precluded
%definitive
clear guidelines for choosing the parameter $k$ in $\UET$.

%\subsection{Revamp the Good-Toulmin estimator with smoothed cutoff}
%\subsection{New solution: Good-Toulmin estimator with smoothed cutoff}
%\subsection{Smoothed Good-Toulmin estimators}
	\label{sec:new}
%In this subsection 
%\section{Main results}
%\label{sec:results}
\section{Approach and results}
We construct a family of estimators that 
\emph{provably} predict $\aUXnm$ optimally not just for constant $t>1$,
but all the way up to $t \propto \log n$.
This shows that per each observed sample, we can infer properties
of $\log n$ yet unseen samples.
%, thereby showing that
%as the number of samples increases, $\aUXnm$ can be
%predicted for an unbounded
%We also show that this is the furthest one can predict no estimator
%can predict $\aUXnm$ beyond $t=c\log n$ for some constant $c$.
%highest limit of predictability.
The proof technique is general and provides a disciplined guideline
for choosing  the parameter $\binlmt$ for $\UET$ and, in addition, a modification 
%results an estimator 
that outperforms $\UET$.
% and can be used to 
%predict had signals up to $t \propto\log n$ and that

\subsection{Smoothed Good-Toulmin (SGT) estimator}
	\label{sec:SGT}

To obtain a new class of estimators, we too start with $\UGT$, but 
unlike $\UET$ that was derived from $\UGT$ via analytical
considerations aimed at improving the convergence rate,
%to improve on $\UGT$ and rigorously evaluate its performance,
we take a probabilistic view that controls the bias and
variance of $\UGT$ and balances the two to obtain a more
efficient estimator. 
%We begin however with a natural and robust measure for the
%performance of an estimator.

%Since $\EE(X^2)=(\EE X)^2+\Var(X)$ for any random variable $X$,
%we get that $\LEnt$ is determined by the estimator's \emph{bias}
%$\EE(\UE-\aUXnm)$ and \emph{variance} $\Var(\UE-\aUXnm)$.
%While the bias of $\UGT$ depends on the sampling model, 

Note that what renders $\UGT$ 
inaccurate when $t > 1$ is not its bias but mainly its high variance
due to the exponential growth of the coefficients $(-t)^i$ in \prettyref{eq:GT}; in fact $\UGT$
is the unique unbiased estimator for all $t$ and $n$ in the closely related Poisson
sampling model (see \prettyref{sec:def}). 
%A straightforward
%modification is to cut off the series \prettyref{eq:GT} at the
%$\ell\Th$ term and use the partial sum as an estimator, namely
%\begin{equation}
%\U^{\fixed}
%\ed
%\sum_{i=1}^\fixed
%-\Paren{-t}^i\prev{i}.%(\xn)
%	\label{eq:UGT-ell}
%\end{equation}
%However, it is not difficult to see this modified estimator still
%suffer from large bias when $m>n$. Consider the uniform distribution
%over $\frac{n}{\ell}$ elements so that with high probability most of
%the symbols appear roughly $\ell$ times in $n$ samples.  For any
%$t>1$, the last term in \prettyref{eq:UGT-ell} dominates, and the
%estimator incurs a bias proportional $n$.  Thus no matter how large
%$n$ is, the estimator would never be accurate (see
%\prettyref{sec:taylor} for a rigorous justification).  \ignore{ the
%  estimate value is $\UGT_\fixed \approx \frac{n
%    (-t)^{\ell+1}}{\ell^{3/2}}$, which can be a large positive or
%  negative value depending on $\ell$ being odd or even.  On the other
%  hand, $U$ always lies between $0$ and $nt$ and not approximated by
%  $\UGT_\fixed$ regardless of the choice of $\ell$}
%The large $(-t)^i$ coefficients of the $\UGT$ series in
%Equation~\prettyref{eq:GT}
%result in a large variance of $\UGT$ for $t>1$.
Therefore it is tempting to truncate the series \prettyref{eq:GT} at the
$\ell\Th$ term and use the partial sum as an estimator:
%consider the finite-sum estimator
\begin{equation}
\label{eq:UGT-ell}
\U^{\fixed}
\ed
-\sum_{i=1}^\fixed \Paren{-t}^i\prev{i}.%(\xn)
\end{equation}
However, for $t>1$, it can be shown that for certain distributions most of the symbols typically appear
$\fixed$ times and hence the last term in \prettyref{eq:UGT-ell} dominates, resulting in a large bias
and inaccurate estimates regardless of the choice of $\ell$ (see \prettyref{sec:taylor} for a rigorous justification).
%in distributions where most of the symbols appear
%$\fixed$ times, the last term dominates, resulting in a bias
%proportional to $n$, hence still high NMSE.
%$\LElnt$.

%However, it is not difficult to see this modified estimator still suffer from large bias when $m>n$. Consider the uniform distribution over $\frac{n}{\ell}$ elements so that with high probability most of the symbols appear roughly $\ell$ times in $n$ samples. 
%For any  $t>1$, the last term in \prettyref{eq:GT-ell} dominates, and the estimator incurs a bias proportional $n$.
%Thus no matter how large $n$ is, the estimator would never be accurate. % (see \prettyref{sec:taylor} for a rigorous justification).
%\ignore{
%the estimate value is $\UGT_\fixed \approx \frac{n (-t)^{\ell+1}}{\ell^{3/2}}$, which can be a large positive or negative value depending on $\ell$ being odd or even. 
%On the other hand, $U$ always lies between $0$ and $nt$ and not approximated by $\UGT_\fixed$ regardless of the choice of $\ell$}
%%First of all, in order to limit the blow-up of variance, $\ell$ has to be a constant. 

To resolve this problem, we truncate the Good-Toulmin estimator at a
random location, denoted by an independent random nonnegative integer $L$, and
average over the distribution of $\rand$, which yields the following estimator:
\begin{equation}
\UL = \EE_L \left[- \sum_{i=1}^\rand
  \Paren{-t}^i\prev{i}\ignore{(\xn)}\right].
	\label{eq:UL0}
\end{equation}
 The key insight is that since the bias of $U^\ell$
typically alternates signs as $\ell$ grows, averaging over
different cutoff locations takes advantage of the cancellation and
dramatically reduces the bias.
%This is motivated by the observation that the bias of $\UGT_\fixed$ typically changes signs alternatively as $\ell$ grows, which allows averaging to take advantage of the cancellation and reduces the bias dramatically. 
%Specifically, this method leads to the following estimator which can be expressed simply as a linear combination of prevalences:
Furthermore, the estimator \prettyref{eq:UL0} can be expressed simply as a linear combination of prevalences:
\begin{equation}
\UL =  \EE_L \left[-
  \sum_{i\geq 1} \Paren{-t}^i\prev{i} \indic_{i \leq L} \ignore{(\xn)}
  \right] = - \sum_{i \geq 1} \Paren{-t}^i \prob{L \ge i} \prev{i}.
	\label{eq:UL}
\end{equation}
We shall refer to estimators of the form \prettyref{eq:UL} \emph{Smoothed Good-Toulmin (SGT)} estimators and the distribution of $L$ the \emph{smoothing distribution}.

Choosing different smoothing distributions results a variety of
linear estimators, where the tail probability $\prob{L \ge i}$ compensates the exponential growth of $\Paren{-t}^i$ thereby stabilizing the variance.  Surprisingly, though the motivation and approach are quite different,
SGT estimators include $\UET$ in \prettyref{eq:UET} as a special case which corresponds to
the binomial smoothing $L\sim\Binom(k,  \frac{1}{1+t})$. 
%$L\sim\Binom(k,q)$ with $q = \frac{1}{1+t}$. 
This provides an intuitive probabilistic interpretation of $\UET$, which was
originally derived via Euler's transform and analytic considerations.
As we show in the next section, this interpretation leads to the first theoretical 
guarantee for $\UET$ as well as improved estimators that are provably optimal.
% principled way for choosing the
%parameter $k$ for $\UET$ and the first provable guarantee for
%its performance, shown in \prettyref{tab:values}.

%With this new interpretation, we are able to show that a class of
%estimators, including $\UET$, %with an appropriately chosen $k$,
%can
%accurately predict $\aUXnm$ not just for $m= O(n)$,
%% that is a constant times larger than $n$, 
% but even for $\mm$ up to $O(n\log n)$, which is in fact theoretically
% optimal. In other words, as the number of available samples grows,
% one can predict not just a constant times further into the future,
% but infinitely many times further.  We also show that a modification
% of $\UET$ provides even more accurate estimate (see
% Section~\ref{sec:results}).  Similar results are obtained for other
% commonly adopted samplings models, which we describe below.

%Observe however that the $\UGT$ series alternates signs,
%hence averaging partial sums can potentially reduce the bias significantly.
%%For a random variable $\rand$ ranging over the natural numbers, 
%For a random nonnegative integer $\rand$, consider
%\begin{equation*}
%\label{eq:UL}
%\UL
%=
%\sum_\ell \prob{\rand = \fixed} \cdot U^\fixed
%=
%-\EE_L \sum_{i=1}^\rand \Paren{-t}^i\prev{i}\ignore{(\xn)}.
%\end{equation*}

\subsection{Main results}
\label{sec:main}

%\emph{Mean-square error} is a common performance measure in science
%and engineering. 
Since $\aUXnm$ takes in values between $0$ and $nt$,
% ranges from 0 to $nt$, 
 we measure the performance of an estimator $\UE$ by the worst-case \emph{normalized mean-square error (NMSE)},
\[
\LEnt
\ed 
\max_p \EE_{p}\Paren{\frac{\UE - \aUXnm}{nt}}^2.
\]
Observe that this criterion conservatively evaluates the performance of the
estimator for the worst possible 
%discrete 
distribution.
The trivial estimator that always predicts $nt/2$ new elements has NMSE equal to $1/4$,
%$\LEnt=1/4$,
and we would like to construct estimators with vanishing NMSE,
%$\LEnt$, 
which can estimate $\aUXnm$ up to an error that diminishes with $n$,
regardless of the data-generating distribution; in particular, we are interested in the largest $t$ for which this is possible.

Relating the bias and variance of $\UL$ to the expectation of $t^L$
and another functional we obtain the following performance guarantee for SGT estimators with appropriately chosen smoothing distributions.
%A key contribution of this paper is a simple formula relating
%the expectation of $t^L$ and a related functional to the bias
%and variance of $\UL$.
%Using this relation, we analyze $\UL$.
%We analyze $\UL$ and show that 
%This is motivated by the observation that the bias of $\UGT_\fixed$
%typically changes signs alternatively as $\ell$ grows, which allows
%averaging to take advantage of the cancellation and reduces the bias
%dramatically.
\begin{Theorem}
\label{thm:res_main}
For Poisson or binomially distributed $L$ with the parameters given in
Table\arxiv{~\ref{tab:values}}{~1}, for all $t \geq 1$ and $n\in \naturals$, 
\[
\LELnt
\lesssim
\frac{1}{n^{1/t}}.
\]
\end{Theorem}
\begin{table}[ht]
\begin{center}
\begin{tabular}{|c|l|l|}
\hline
Smoothing distribution &\hfill Parameters\hfill${}$ &\hfill  $\LELnt\lesssim$\hfill${}$\\ \hline
Poisson $(r)$ & $r=\frac{1}{2t}\log_e \frac{n(t+1)^2}{t-1}$ & $n^{-1/t}$\\[0.4cm]  
Binomial $(\binlmt, q)$ & $\binlmt =  \left \lceil \frac{1}{2} \log_2  \frac{nt^2}{t-1} \right \rceil $, $q =\frac{1}{t+1} $ & $n^{-\log_2(1+1/t)}$  \\[0.4cm] 
Binomial $(\binlmt, q)$ & $\binlmt =  \left \lceil \frac{1}{2} \log_3  \frac{nt^2}{t-1} \right \rceil $, $q = \frac{2}{t+2}$ & $n^{-\log_3(1+2/t)}$ \\[0.4cm]   \hline
\end{tabular}
\caption{NMSE of SGT estimators for three smoothing distributions.
Since for any $t \geq 1$, $\log_3(1+2/t) \geq \log_2(1+1/t) \geq 1/t$,
binomial smoothing with $q = 2/(2+t)$ yields the best convergence rate.}
\label{tab:values}
\end{center}
\end{table}
%\nbr{We prove \prettyref{thm:res_main} for three common models used in
%  scientific communities: multinomial (the model described in Section~\ref{sec:intro}), Poisson, and Bernoulli-product
%  models. Furtermore, the hidden constant is small \eg the
%  hidden constant is $1$ for large values of $t$ under the Poisson
%  sampling model. A similar result also holds under the hypergeoemtric
%  model and is stated in \prettyref{cor:hyper}.}

%Building on the correspondence between the Efron-Thisted estimator and the binomial cutoff distribution described in \prettyref{sec:SGT},
\prettyref{thm:res_main} provides a principled way for choosing the
parameter $k$ for $\UET$ and the first provable guarantee for
its performance, shown in \prettyref{tab:values}.
Furthermore, the result shows that a modification of $\UET$
with $q= \frac{2}{t+2}$ enjoys even faster convergence rate
% stronger performance guarantees
and, as experimentally demonstrated in \prettyref{sec:experiments}, outperforms 
the original version of Efron-Thisted as well as other state-of-the-art estimators.

Furthermore, %\nbr{under multinomial and Poisson sampling} 
SGT estimators are essentially optimal as witnessed by the following matching minimax lower bound.
\begin{Theorem}
\label{thm:res_lb}
%\nbr{Under the multinomial and Poisson sampling models,} 
There exist universal constant $c,c'$ such that for any $t\geq c$, any $n\in \naturals$,  and any estimator $\UE$
\[
\LEnt \gtrsim \frac{1}{n^{c'/t}}.
\]
\end{Theorem}

Theorems \ref{thm:res_main} and \ref{thm:res_lb}
determine the limit of predictability up to a
constant multiple.
%the \emph{predition horizon} to obtain a NMSE 
\begin{Corollary}
For any $\delta > 0 $, 
\[
\lim_{n \to \infty} \frac{\max\sets{t: \LEnt<\delta\textrm{\,\, for some $\UE$}}}{\log n}
\asymp %\eqsim
\frac{1}{\log \frac{1}{\delta}}.% \cdot  (1+o_n(1)).
\]
\end{Corollary}
The rest of the paper is organized as follows: In
Section~\ref{sec:def}, we describe the four statistical models
commonly used across various scientific disciplines, namely, the
multinomial, Poisson, hypergeometric, and Bernoulli product models.
%The performance guarantees of the SGT estimators above hold for all
%four models; the constant factors involved are small and their precise 
%value depends on the model. 
Among the four models Poisson is the simplest to analyze and hence in
Sections~\ref{sec:prelim} and~\ref{sec:upper}, we first prove \prettyref{thm:res_main}
for the Poisson model and in Section~\ref{sec:extensions} we prove
similar results for the other three statistical models. In
Section~\ref{sec:lower}, we prove the lower bound for the multinomial and Poisson
models. Finally, in Section~\ref{sec:experiments} we demonstrate the
efficiency and practicality of our estimators on a variety of
synthetic and data sets.
\section{Statistical models}
\label{sec:def}
%To formally describe our results, we need some terminology.
The extrapolation paradigm has been applied to several statistical
models.
In all of them, an initial sample of size related to $n$
is collected, resulting in a set $\Sold$ of observed elements.
We consider collecting a new sample of size related to $\mm$,
that would result in a yet unknown set $\Snew$ of observed elements,
and we would like to estimate
\[
|\Snew\backslash\Sold|,
\]
the number of unseen symbols that will appear in the new sample.
For example, for the observed sample \texttt{bananas} and future
sample \texttt{sonatas}, 
%``b,a,n,a,n,a,s'' and future
%sample ``s,o,n,a,t,a,s'', 
$\Sold=\sets{\tt a,b,n,s}$,
$\Snew=\sets{\tt a,n,o,s,t}$, and $|\Snew\backslash\Sold|=|\sets{\tt o,t}|=2$.

Four statistical models have been commonly used in the
literature (cf.~survey \cite{BF93} and \cite{CCG12}), and our
results apply to all of them. The first three statistical models are
also referred as the~\emph{abundance models} and the last one is often
referred to as the~\emph{incidence model} in ecology~\cite{CCG12}.
\begin{description}
\item[Multinomial:]
%\noindent\textbf{Multinomial:}
This is Good and Toulmin's original model where the samples are independently and identically distributed (i.i.d.),
and the initial and new samples consist of exactly $n$ and $m$
elements respectively. Formally, 
$X^{n+m}=X_1\upto X_{n+m}$ are generated independently according
to an unknown discrete distribution of finite or even infinite support,
$\Sold=\sets{\Xn}$, and $\Snew=\sets{X_{n+1}^{n+m}}$.

\item[Hypergeometric:]
%\noindent\textbf{Hypergeometric:}
This model corresponds to a sampling-without-replacement variant of the multinomial model.
%selecting samples without replacement from
%a collection of symbols with possible repetitions, for example, an urn
%with six white and four black balls. 
Specifically, $X^{n+m}$ are drawn uniformly without replacement from an unknown
collection of symbols that may contain repetitions, for example, an urn
with some white and black balls. Again, $\Sold=\sets{\Xn}$
and $\Snew=\sets{X_{n+1}^{n+m}}$. 

\item[Poisson:]
%\noindent\textbf{Poisson:}
As in the multinomial model, the samples are also \iid, 
but the sample sizes, instead of
being fixed, are Poisson distributed. 
%This occurs for example when we
%take a Poisson number of samples, or when the independent samples
%streams follow Poisson arrival process. 
Formally, $N\sim\poi(n)$, $M\sim\poi(\mm)$, $X^{N+M}$ are generated
independently according to an unknown discrete distribution,
$\Sold=\sets{X^{N}}$, and $\Snew=\sets{X_{N+1}^{N+M}}$.

\item[Bernoulli-product:]
%\noindent\textbf{Bernoulli-product:} 
In this model we observe signals
from a collection of independent processes over subset of an unknown set
$\cX$. Every $x\in\cX$ is associated with an unknown probability $0\le
p_x\le 1$, where the probabilities do not necessarily sum to 1.  Each
sample $X_i$ is a \emph{subset} of $\cX$ where symbol $x\in\cX$
appears with probability $p_x$ and is absent with probability $1-p_x$,
independently of all other symbols. $\Sold=\cup_{i=1}^n X_{i}$ and
$\Snew=\cup_{i=n+1}^{n+m}X_i$.
\end{description}

%The remainder of the paper is devoded to theoretical analysis and experimental studies.
For theoretical analysis in Sections \ref{sec:prelim} and \ref{sec:upper} we use the Poisson sampling model as the leading example 
 due to its
simplicity. Later in \prettyref{sec:extensions}, we show that very similar
results continue to hold for the other three models.

We close this section by discussing two problems that are closely related to the 
extrapolation model, namely,
\emph{support size estimation} and \emph{missing mass estimation}, which correspond to $m=\infty$ and $m=1$ respectively. 
Indeed, the probability that the next sample is new is precisely the expected value of $U$ for $m=1$,
% with respect to the future samples is precisely 
which is the goal in the basic Good-Turing problem~\cite{G53,R68, MS00,OS15}.
On the other hand, any estimator $\aUEXnm$ for $U$ can be converted to a (not necessarily good) support size
estimator by adding the number of observed symbols.
%\[
%\suppest = \prev{+} + \aUEXnm,
%\]
%where $\sum^\infty_{i=1} \prev{i}$
Estimating the support size of an underlying distribution has been
studied by both ecologists~\cite{C84, CL92, BF93} and
theoreticians~\cite{RRSS09, VV11, VV13, WY14};
\ignore{ In ecology, several lower
  bounds on the underlying support size are known.  Most ecology
  papers study the support estimation in this context~\cite{C84, CL92,
    BF93}. While no upper bounds on the support can be
  rigorously proven, provable lower bounds can be obtained~\cite{}.}%
  however, to make the problem non-trivial, 
all statistical models impose a lower bound on the minimum non-zero
probability of each symbol, which is assumed to be known to the statistician.  We
discuss these estimators and their differences to our results in Section~\ref{sec:support}.
%    Finally, note that the expected value of
%$U(x^n,1)$ is simply the probability of the \emph{missing mass} in the
%basic Good-Turing problem~\cite{G53,R68, MS00}.
%For any positive integer $n$, let $[n]=\{1, \ldots, n\}$.
%For any set $T \subset [n]$, let $|T|$ denote its cardinality and $T^c$ denote its complement.

\section{Preliminaries and the Poisson model}
\label{sec:prelim}
Throughout the paper, 
we use standard asymptotic notation,
e.g., for any positive sequences $\{a_n\}$ and $\{b_n\}$, denote $a_n=\Theta(b_n)$ or $a_n  \asymp b_n$
if $1/c\le a_n/ b_n \le c$ for some universal constant $c>0$.
%Let $D(P\|Q)=\mathbb{E}_P[ \log \frac{d P}{d Q} ]$ denote the Kullback-Leibler (KL) divergence between 
%two distributions $P$ and $Q$.
Let $\indic_{A}$ denote the indicator random variable of an event $A$.
Let $\Binom(n,p)$ denote the binomial distribution with $n$ trials and success probability $p$ and let $\poi(\lambda)$ denote the Poisson distribution with mean $\lambda$.
%Let $\Bern(p)$ denote the Bernoulli distribution with mean $p$ and $d(p\|q) = D(\Bern(p)\|\Bern(q)) = p \log \frac{p}{q} + (1-p)\log \frac{1-p}{1-q}$.
All logarithms are with respect to the natural base unless otherwise specified. %test
%Let $\Phi(x)$ and $Q(x)$ denote the cumulative distribution function (CDF) and complementary CDF  of a standard normal distribution,
%respectively. We say a sequence of events $\calE_n$ holds with high probability, if $\prob{\calE_n} \to 1$ as $n \to \infty$

Let $p$ be a probability distribution over a discrete set $\cX$, namely $p_x \geq 0$ for all $x\in \cX$
 and $\sum_{x \in \cX} p_x = 1$.
Recall that the sample sizes are Poisson distributed: $N \sim \poi(n)$, $M \sim \poi(m)$, and  $ %\[ 
t = \frac{m}{n}$. We abbreviate the number of unseen symbols by
\[
\U \ed \UXNM,
\]
and we denote an estimator by $\UE\ed\UEXNm$.
%$ and $\Ug \ed \Ugnm$.% and $\UJ \ed \UJxnm$.

%Under Poisson sampling, $\U$ is the number of unseen
%symbols that we observe upon sampling $p$ another $\poi(n\cdot t)$ times.
Let $\mul{x}$ and $\mul{x}'$ denote the multiplicity of a symbol $x$
in the current samples and future samples, respectively. 
%Poisson sampling offers several advantages over the multinomial 
%and Bernoulli-product models.
Let $\lambda_x\ed np_x$. Then a symbol $x$ 
appears $\mul{x}\sim\poi(np_x)=\poi(\lambda_x)$ times, and for any $i\geq 0$,
\[
\EE[\indic_{\mul{x} = i}] = e^{-\lambda_x} \frac{\lambda_x^i}{i!}.
\]
Hence
\[
\EE[\prev{i}] = \EE\left[\sum_{x} \indic_{\mul{x} = i} \right] =
\sum_{x} e^{-\lambda_x} \frac{\lambda_x^i}{i!}.
\]
A helpful property of Poisson sampling is that the multiplicities of
different symbols are independent of each other.  
Therefore, for any function $f(x,i)$,
\[
\Var\left(\sum_x f(x,N_x)\right) = \sum_x \Var(f(x,N_x)).
\]
Many of our derivations rely on these three equations.
%\ignore{For example, since $\mul{x}'$ is a Poisson distributed random variable with mean
%$t\lambda_x$,
%\begin{equation}
%%\[
%\label{eq:Upoisson}
%\Ut= \sum_{x} \indic_{\mul{x} = 0} \cdot \EE[\indic_{\mul{x}' > 0}]
%= \sum_{x} \indic_{\mul{x} = 0} \cdot \left(1 - e^{-t\lambda_x}
%\right).
%%\]
%\end{equation}
%}
%We note that~\cite{ET76} studied species estimation under Poisson sampling model, however they estimated a slightly different quantity. They estimated
For example,
\begin{align*}
%\U' = 
\EE[\U] 
 = \sum_{x} \EE[\indic_{\mul{x} = 0}] \cdot \EE[\indic_{\mul{x}' > 0}]
= \sum_{x} e^{-\lambda_x}\cdot (1-e^{-t\lambda_x}),
\end{align*}
and 
\begin{align*}
\Var(\U)
&=
\Var \Paren{\sum_x \indic_{\mul{x} = 0} \cdot \indic_{\mul{x}' > 0}} 
%& 
= 
\sum_x \Var \Paren{\indic_{\mul{x} = 0} \cdot \indic_{\mul{x}' > 0}}  \\
&
\le
\sum_x \EE \left[\indic_{\mul{x} = 0}  \cdot  \indic_{\mul{x}' > 0}\right]
% &
=
\EE\left[ \Ut\right].
\end{align*}
Note that these equations imply that the standard deviation of $U$ is
at most $\sqrt{\EE[U]} \ll \EE[U]$, hence $U$ highly
concentrates around its expectation, and estimating $\U$ and $\EE[\U]$ are
essentially the same.

\subsection{The Good-Toulmin estimator}
Before proceeding with general estimators, we prove a few properties of $\UGT$.  Under the Poisson model, $\UGT$ is in fact the \emph{unique} unbiased estimator for $\U$.
\begin{Lemma}[\cite{ET76}]
For any distribution,
\[
\EE[\U] = \EE[\UGT].
\]
\end{Lemma}
\begin{proof}
\begin{align*}
\EE [\U]
&=
\EE \left[\sum_{x} \indic_{\mul{x} = 0} \cdot \indic_{\mul{x} > 0}\right] %\\
%&
=
\sum_{x} e^{-\lambda_x} \cdot  \left(1 - e^{-t\lambda_x} \right) \\
&=
-\sum_x e^{-\lambda_x} \cdot \sum^\infty_{i=1} \frac{(-t\lambda_x)^i}{i!} 
% \\ &
=
-\sum^\infty_{i=1} (-t)^i\cdot \sum_x e^{-\lambda_x} \frac{\lambda_x^i}{i!}  \\
%\\ &
& =
-\sum^\infty_{i=1} (-t)^i\cdot \EE[ \prev{i}] = \EE[\UGT]. \qedhere
\end{align*}
\end{proof}
Even though $\UGT$ is unbiased for all $t$, for $t>1$ it has
high variance and hence does not estimate $U$ well even for the simplest distributions.% In fact, the follow result shows that $\UGT$ cannot be accurate for $t > 1$.
\begin{Lemma}
\label{lem:FGT}
For any $t >1$,
\begin{align*}
\lim_{n \to \infty} \LEGTnt = \infty.
\end{align*}
%There exists a distribution such that for every $t >1$,
%\begin{align*}
%\lim_{n \to \infty} \frac{\EE[(\UGT-\U)^2] }{(nt)^2} = \infty.
%\end{align*}
\end{Lemma}
\begin{proof}
%From Equation~\eqref{eq:GTguarantee}, for any $m \leq n$,
%\begin{align*}
%\frac{\EE[(\aUxnm-\aUGTxnm)^2] }{m^2} \leq \frac{2}{m}.
%\end{align*}
%Next we prove a negative result: if $\lim_{n \to \infty} m/n >1$, then the normalized mean-square error of the GT estimator tends to infinity, thus proving the lemma.
%Let $t = m/n$. 
Let $p$ be the uniform distribution over two symbols $a$ and $b$, namely,
$p_a = p_b = 1/2$. First consider even $n$.
Since $(\UGT - U)^2$ is always nonnegative,
\[
\EE[(\UGT-\U)^2] \geq \Pr(\mul{a} =\mul{b} = n/2) (2(-t)^{n/2})^2 =
\left(e^{-n/2} \frac{(n/2)^{n/2}}{(n/2)!} \right)^2 4 t^n \geq \frac{4t^n}{e^2n},
\]
where we used the fact that $k! \leq (\frac{k}{e})^k \sqrt{k} e$.
Hence for $t > 1$,
\[
\lim_{n \to \infty} \frac{\EE[(\UGT-\U)^2] }{(nt)^2}
\geq \lim_{n \to \infty}  \frac{4t^n}{e^2n (nt)^2} = \infty.
\]
The case of odd $n$ can be shown similarly by considering the event $\mul{a} = \lfloor n/2\rfloor, \mul{b} = \lceil n/2\rceil$.
\ignore{
In this case, according to Eq.\prettyref{eq:Upoisson}
we have $\U = \prev{0} (1-e^{-nt/2})$ and hence
\begin{align*}
\UGT - \U
= -\sum^{\infty}_{i=1} \prev{i} (-t)^{i} - \prev{0} (1-e^{-nt/2})
 = \sum_{x \in \{a,b\}} \left(-\sum^{\infty}_{i=1} \indic_{\mul{x} = i} \cdot (-t)^{i} -  \indic_{\mul{x} = 0}\cdot (1-e^{-nt/2})\right).
\end{align*}
Since the GT estimator is unbiased, we have
\begin{align*}
\EE[(\UGT - \U )^2]
& = \Var(\UGT - \U) \\
& = \Var \left( \sum_{x \in \{a,b\}} -\left(\sum^{\infty}_{i=1} \indic_{\mul{x} = i} (-t)^{i} -  \indic_{\mul{x} = 0} (1-e^{-nt/2}) \right)\right) \\
& =  \sum_{x \in \{a,b\}} \Var \left(-\sum^{\infty}_{i=1} \indic_{\mul{x} = i} (-t)^{i} -  \indic_{\mul{x} = 0} (1-e^{-nt/2}) \right) \\
& \stackrel{(a)}{=} 2 \Var \left(-\sum^{\infty}_{i=1} \indic_{\mul{a} = i} (-t)^{i} -  \indic_{\mul{a} = 0} (1-e^{-nt/2}) \right) \\
& = 2 \EE \left[ \left(-\sum^{\infty}_{i=1} \indic_{\mul{a} = i} (-t)^{i} -  \indic_{\mul{a} = 0} (1-e^{-nt/2}) \right)^2\right] \\
& \stackrel{(b)}{=} 2 \EE \left[\sum^{\infty}_{i=1} \indic_{\mul{a} = i} (-t)^{2i} +  \indic_{\mul{a} = 0} (1-e^{-nt/2})^2 \right] \\ 
& = 2 \sum^\infty_{i=1} e^{-n/2} \frac{(n/2)^i t^{2i}}{i!} + 2e^{-n/2}(1-e^{-nt/2}) \\
& = 2 e^{\frac{n(t^2-1)}{2}} - 2e^{-n(t+1)/2},
\end{align*}
where (a) follows from symmetry as both symbols have identical probability, and
(b) follows from the fact that for any indicator random variable $\indic^2 =  \indic$. Hence for any $t > 1$
\[
\lim_{n \to \infty} \frac{\EE[(\UGT-\U)^2] }{(nt)^2} = \infty.
 \qedhere
\]
}%
\end{proof}
\subsection{General linear estimators} 
Following \cite{ET76}, we consider general linear estimators of the form
\begin{equation}
\Uht = \sum^\infty_{i=1} \prev{i} \cdot h_i,
	\label{eq:Uh}
\end{equation}
which can be identified with 
%Every linear estimator corresponds to 
a formal power series $h(y)=
\sum^\infty_{i=1}\frac{h_i y^i}{i!}$. For example, $\UGT$ in \prettyref{eq:GT}
corresponds to the function $h(y) = 1-e^{-yt}$.  
The next lemma bounds the bias and variance of any linear estimator $\Uh$ using
properties of the function $h$. In \prettyref{sec:smooth} we apply this result to 
the SGT estimator %\prettyref{eq:UL}
whose coefficients are of the specific form:
\begin{equation*}
h_i = -\Paren{-t}^{i}\cdot\prob{\rand \geq i}.
	\label{eq:hL}
\end{equation*}
%All of our results on the proposed
%estimator, including the following lemma holds for all distributions
%including the ones with infinite support. Hence the results are stated
%without any condition on the underlying distribution. 
Let $\prev{+} \ed \sum^\infty_{i=1} \prev{i}$ denote the number of observed symbols.
\begin{Lemma}
\label{lem:general_bounds}
%The bias of $\sum^\infty_{i=1} \prev{i} h_i$ is
The bias of $\Uht$ is
\begin{equation*}
\label{eq:Upoisson}
\EE[\Uht - \U] = \sum_{x} e^{-\lambda_x} \left( h(\lambda_x) - (1-e^{-t\lambda_x}) \right),
\end{equation*}
and the variance satisfies
\[
\Var(\Uht - \U) \leq \EE[\prev{+}]  \cdot \sup_{i \geq 1} h^2_i + \EE[U].
\]
\end{Lemma}
\begin{proof}
Note that
\begin{align}
\Uht - \Ut
& = \sum^\infty_{i=1} \prev{i} h_i -  \sum_x \indic_{\mul{x} = 0} \cdot \indic_{\mul{x}' > 0}\nonumber \\
& = \sum^\infty_{i=1}
  \sum_{x} \indic_{\mul{x}=i} \cdot h_i - \sum_x \indic_{\mul{x} = 0}
  \cdot  \indic_{\mul{x}' > 0}
\nonumber \\
& = \sum_x \left(\sum^\infty_{i=1} \indic_{\mul{x}=i} \cdot h_i -
  \indic_{\mul{x} = 0} \cdot  \indic_{\mul{x}' > 0} \right).
%\label{eq:without_expectation}
\nonumber
\end{align}
For every symbol $x$,
\begin{align}
\EE \left[\sum^\infty_{i=1} \indic_{\mul{x}=i} \cdot h_i -
  \indic_{\mul{x} = 0} \cdot  \indic_{\mul{x}' > 0} \right] 
& = 
 \sum^\infty_{i=1} e^{-\lambda_x}
\frac{\lambda^i_x}{i!} \cdot h_i - e^{-\lambda_x} \cdot (1 -
e^{-t\lambda_x})  \nonumber \\ 
& = e^{-\lambda_x} \left(
\sum^\infty_{i=1} \frac{\lambda^i_x h_i}{i!} - (1-e^{-t\lambda_x})
\right) \nonumber \\ 
& =  e^{-\lambda_x} \left( h(\lambda_x) -
(1-e^{-t\lambda_x}) \right),
\nonumber
%\label{eq:expectation}
\end{align}
from which \prettyref{eq:Upoisson} follows.
%and the result follows by summing over all symbols.
\ignore{
The bias of such an estimator is 
\begin{align*}
\EE[\Uht - \U] & = \EE\left[\sum^\infty_{i=1} \prev{i} h_i
  \right] - \EE\left[ \sum_x \indic_{\mul{x} = 0} \cdot \left(1 -
  e^{-t\lambda_x} \right)\right] \\ 
& = \EE\left[\sum^\infty_{i=1}
  \sum_{x} \indic_{\mul{x}=i} \cdot h_i - \sum_x \indic_{\mul{x} = 0}
  \cdot \left(1 - e^{-t\lambda_x} \right)\right] \\ 
& = \sum_{x}
\EE\left[\sum^\infty_{i=1} \indic_{\mul{x}=i} \cdot h_i -
  \indic_{\mul{x} = 0} \cdot \left(1 - e^{-t\lambda_x} \right)\right]
\\ & = \sum_{x} \left( \sum^\infty_{i=1} e^{-\lambda_x}
\frac{\lambda^i_x}{i!} \cdot h_i - e^{-\lambda_x} \cdot \left(1 -
e^{-t\lambda_x} \right)\right) \\ 
& = \sum_{x} e^{-\lambda_x} \left(
\sum^\infty_{i=1} \frac{\lambda^i_x h_i}{i!} - (1-e^{-t\lambda_x})
\right) \\ 
& = \sum_{x} e^{-\lambda_x} \left( h(\lambda_x) -
(1-e^{-t\lambda_x}) \right).
\end{align*}
}%
\ignore{
For every symbol $x$,
\begin{align*}
\Var \left(
\sum^\infty_{i=1} \indic_{\mul{x}=i} \cdot h_i - \indic_{\mul{x} = 0}
\cdot \left(1 - e^{-t\lambda_x} \right) \right) 
& \leq \EE
\left[ \left( \sum^\infty_{i=1} \indic_{\mul{x}=i} \cdot h_i -
  \indic_{\mul{x} = 0} \cdot \left(1 - e^{-t\lambda_x} \right)
  \right)^2 \right] \\ 
& \stackrel{(a)}{=}  \EE \left[ \sum^\infty_{i=1}
  \indic_{\mul{x}=i}h_i^2 \right] + \EE[\indic_{\mul{x} = 0}] \cdot
\left(1 - e^{-t\lambda_x} \right)^2 \\ 
& \leq \sum^\infty_{i=1}  \EE
[  \indic_{\mul{x}=i}]h_i^2 +
\EE[\indic_{\mul{x} = 0}] \cdot \left(1 - e^{-t\lambda_x} \right).
\end{align*}
For every $i \neq j$, $\EE[ \indic_{\mul{x}=i}  \indic_{\mul{x}=j}] = 0$ and hence $(a)$. 
}%
For the variance, observe that for every symbol $x$,
\begin{align*}
\Var \left(
\sum^\infty_{i=1} \indic_{\mul{x}=i} \cdot h_i - \indic_{\mul{x} = 0}
\cdot \indic_{\mul{x}'>0} \right) 
& \leq \EE
\left[ \left( \sum^\infty_{i=1} \indic_{\mul{x}=i} \cdot h_i -
  \indic_{\mul{x} = 0} \cdot \indic_{\mul{x}' >0}
  \right)^2 \right] \\ 
& \stackrel{(a)}{=}  \EE \left[ \sum^\infty_{i=1}
  \indic_{\mul{x}=i}h_i^2 \right] + \EE[\indic_{\mul{x} = 0}] \cdot
 \EE[\indic_{\mul{x}' >0}] \\ 
& = \sum^\infty_{i=1}  \EE
[  \indic_{\mul{x}=i}]\cdot h_i^2 +
\EE[\indic_{\mul{x} = 0}] \cdot  \EE[\indic_{\mul{x}' >0}],
\end{align*}
where $(a)$ follows as for every $i \neq j$, $\EE[ \indic_{\mul{x}=i}  \indic_{\mul{x}=j}] = 0$. 
% $  \EE
%[\indic_{\mul{x}=1} ]  = e^{-\lambda_x} \lambda_x = \EE[\indic_{\mul{x} = 0}] \lambda_x$ and hence $c$.
Since the variance of a sum of independent random variables is the sum of variances,
%The variance of the estimation error is 
\begin{align*}
\Var(\Uht - \U) 
& \leq  \sum_{x}   
\sum^\infty_{i=1}\EE[  \indic_{\mul{x}=i} ]h_i^2 +
\sum_x \EE[\indic_{\mul{x}=0}] \cdot  \EE[\indic_{\mul{x}' >0}]\\
& =
\sum^\infty_{i=1} \EE[\prev{i}] \cdot h_i^2 + \EE[\U] \\ 
%& =
%\sum^\infty_{i=1} \EE[\prev{i}] i \cdot \frac{h_i^2}{i} +
%\EE[\prev{1}] \cdot t \\ & \leq \left(\sum^\infty_{i=1} \EE[\prev{i}] i
%\right) \cdot \left(t + \sup_{j \geq 1} \frac{h^2_j }{j} \right)\\ 
& \leq
\EE[\prev{+}] \cdot \sup_{i \geq 1} h^2_i + \EE[U].
\qedhere
\end{align*}
\ignore{\begin{align*}
\Var(\Uht - \U) & = \Var \left( \sum^\infty_{i=1} \sum_{x}
\indic_{\mul{x}=i} \cdot h_i - \sum_x \indic_{\mul{x} = 0} \cdot
(1 - e^{-t\lambda_x})\right) \\
 & = \sum_x \Var \left(
\sum^\infty_{i=1} \indic_{\mul{x}=i} \cdot h_i - \indic_{\mul{x} = 0}
\cdot \left(1 - e^{-t\lambda_x} \right) \right) \\ 
& \leq \sum_x \EE
\left[ \left( \sum^\infty_{i=1} \indic_{\mul{x}=i} \cdot h_i -
  \indic_{\mul{x} = 0} \cdot \left(1 - e^{-t\lambda_x} \right)
  \right)^2 \right] \\ 
& = \sum_x \left( \EE \left[ \sum^\infty_{i=1}
  \indic_{\mul{x}=i} \right]h_i^2 + \EE[\indic_{\mul{x} = 0}] \cdot
\left(1 - e^{-t\lambda_x} \right)^2 \right)\\ 
& \leq \sum_x \left( \EE
\left[ \sum^\infty_{i=1} \indic_{\mul{x}=i} \right]h_i^2 +
\EE[\indic_{\mul{x} = 0}] \cdot t\lambda_x \right)\\ 
& =
\sum^\infty_{i=1} \EE[\prev{i}] \cdot h_i^2 + \EE[\prev{1}] t \\ 
& =
%\sum^\infty_{i=1} \EE[\prev{i}] i \cdot \frac{h_i^2}{i} +
%\EE[\prev{1}] t \\ & \leq \left(\sum^\infty_{i=1} \EE[\prev{i}] i
%\right) \cdot \left(t + \sup_{j \geq 1} \frac{h^2_j }{j} \right)\\ 
&\leq
\EE[\prev{+}] \cdot \left(t + \sup_{i \geq 1} \frac{h^2_i}{i} \right).
\qedhere 
\end{align*}
}
%where the last equality follows from the fact that $\sum^\infty_{i=1}
%\EE[\prev{i}] i$ is the total number of of expected symbols and hence is
%$n$.
\end{proof}
\prettyref{lem:general_bounds} enables us to reduce the estimation
problem to a task on approximating functions. Specifically, in view of \eqref{eq:Upoisson}, the goal is to approximate $1-e^{-yt}$ by a
function $h(y)$ whose derivatives at zero all have small magnitude.
%\ie $\sup_i |h_i|$ is small.

\subsection{Estimation via polynomial approximation and support size estimation}
\label{sec:support}
Approximation-theoretic techniques for estimating norms and other
properties such as support size and entropy have been successfully used in the
statistics literature. For example, estimating the $L_p$ norms in Gaussian models
\cite{LNS99,CL11} and  estimating entropy~\cite{WY14,JVHW15} and
support size~\cite{WY14b} of discrete distributions.
%have been considered. 
Among the aforementioned
problems, support size estimation is closest to ours. Hence, we now
discuss the difference between the approximation technique we use and the
those used for support size estimation.

The support size of a discrete distribution $p$ is 
\begin{equation}
S(p) = \sum_x \indic_{p_x > 0}.
	\label{eq:Sp}
\end{equation}
At the first glance, estimating $S(p)$ may
appear similar to species estimation problem as one can convert a
support size estimator $\suppest$ to $\hat U$ by
\[
\hat{U} = \suppest - \sum^\infty_{i=1} \prev{i}.
\]
%for the purpose of estimating $U_n(m)$
However, without any assumption on the distribution it is impossible to estimate the support size. 
For example, regardless how many samples are collected, there could be infinitely many symbols
with arbitrarily small probabilities that will never be observed.  A common assumption is therefore
that the minimum non-zero probability of the underlying distribution $p$, denoted by $p^+_{\min}$, 
is at least $1/k$, for some known $k$. Under this assumption~\cite{VV11} used a linear programming estimator similar to the one in \cite{ET76}, to estimate the support size within an additive error of $k \epsilon$ with constant probability using 
$\Omega(\frac{k}{\log k} \frac{1}{\epsilon^2})$ samples. Based on {best polynomial approximations} recently \cite{WY14b} showed that the minimax risk
of support size estimation satisfies
\[
\min_{\suppest} \max_{p: p^+_{\min} \geq 1/k} \EE_p
    [(\suppest-\supp(p))^2] = k^2 \exp \left(- \Theta\pth{\max \sth{
    \sqrt{\frac{k\log k}{n}}, \frac{k}{n}, 1}} \right)
\]
and that the optimal sample complexity of for estimating $S(p)$ within an additive error of $k \epsilon$ with constant probability is in fact $\Theta(\frac{k}{\log k} \log^2 \frac{1}{\epsilon})$.
%where $c$ and $c'$ are universal constants. 
Note that the assumption $p^+_{\min} \geq 1/k$ is crucial for this result to hold for otherwise estimation is impossible; in contrast, as we
show later, for species estimation no such assumptions are necessary. The intuition is that if there exist a large number of very improbable symbols, most likely they will not appear in the new samples anyway.

To estimate the support size, in view of \prettyref{eq:Sp} and the assumption $p^+_{\min} \geq 1/k$, the technique of \cite{WY14b} is to approximate the indicator function $y\mapsto\indic_{y \geq 1/k}$ in the
range $\{0\} \cup [1/k, \log k/n]$ using Chebyshev polynomials.  
%Graphically this can be seen as  follows. 
Since by assumption no $p_x$ lies in $(0,\frac{1}{k})$, the approximation error in this interval is irrelevant. For example, in \arxiv{Figure~\ref{fig:fsa}}{(\prettyref{fig:fsa})}, the red curve is a useful
  approximation for the support size, even though it behaves badly over
  $(0,1/k)$. To estimate the average number of unseen symbols $\U$, in view of \prettyref{eq:Upoisson}, we need to approximate $y\mapsto 1-e^{-yt}$ over the entire
  $[0,\infty)$ as in, \eg \arxiv{Figure~\ref{fig:fsb}}{(\prettyref{fig:fsb})}.
\begin{figure}[ht]
\centering
\subfigure[]{\label{fig:fsa}
\arxiv{\includegraphics[scale = 0.4]{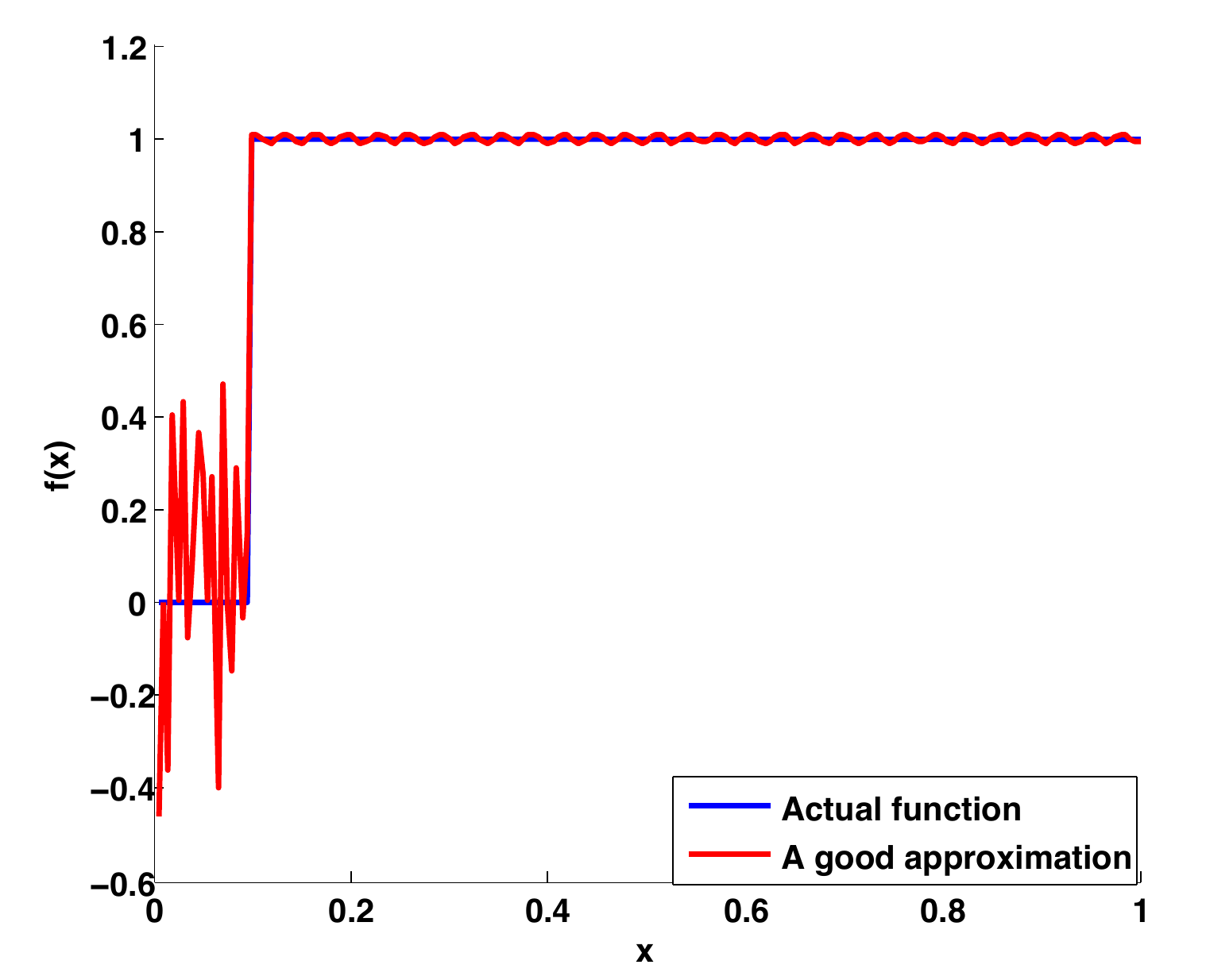}}{\includegraphics[scale = 0.4]{suppfigure.pdf}}}
\subfigure[]{\label{fig:fsb}
\arxiv{\includegraphics[scale = 0.4]{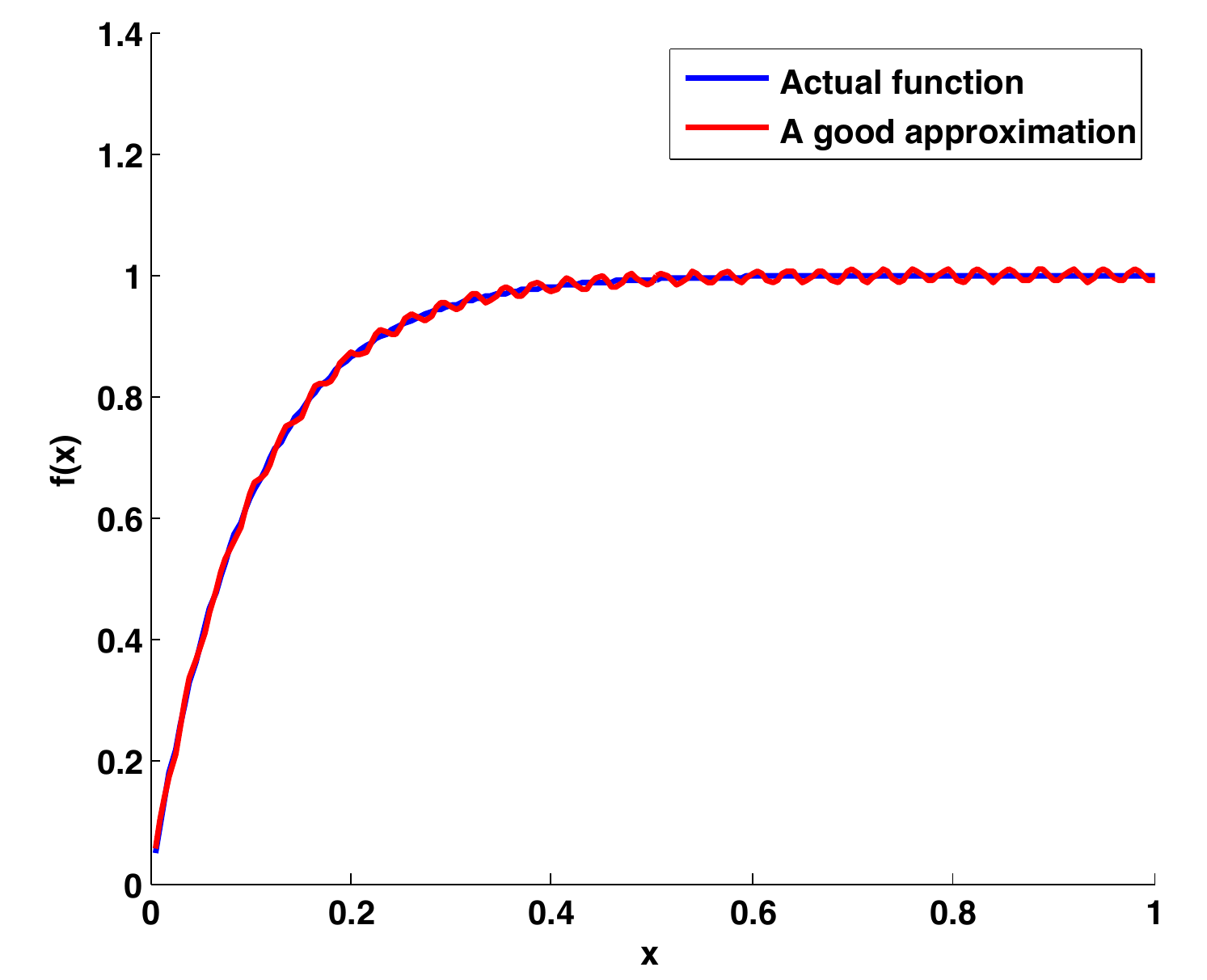}}{\includegraphics[scale = 0.4]{specifigure.pdf}}}
\caption{(a) a good approximation for support size; (b) a good
  approximation for species estimation.}
%\label{fig:fsa}
\end{figure}
Concurrent to this work,~\cite{VV15} proposed a linear programming algorithm to
estimate $U$.  However, their NMSE is $\cO(\frac{t}{\log n})$ compared
to the optimal result $\cO(n^{-1/t})$ in \prettyref{thm:res_main}, thus exponentially weaker for $t =
o(\log n)$. Furthermore, the computational cost far exceeds those of 
our linear estimators.% such as $\U$ and $\UET$. 
\section{Results for the Poisson model}
\label{sec:upper}
In this section, we provide the performance guarantee for SGT estimators under the Poisson sampling model.  We
first show that the truncated GT estimators incurs a high bias.  We
then introduce the class of smoothed GT estimators 
obtained by averaging several truncated GT estimators and bound their
mean squared error in Theorem~\ref{thm:general_bound} for an arbitrary smoothing distribution. We then apply
this result to obtain NMSE bounds for Poisson and Binomial smoothing
in Corollaries~\ref{cor:poi_tail} and~\ref{cor:bin_tail}
respectively, which imply the main result (Theorem~\ref{thm:res_main})
announced in \prettyref{sec:main} for the Poisson model.

\subsection{Why truncated Good-Toulmin does not work}
\label{sec:taylor}
Before we discuss the SGT estimator, we first show that the naive approach of truncating the GT estimator described in \prettyref{sec:SGT}
%\prettyref{eq:UGT-ell} 
leads to bad performance when $t>1$.
%outline a simple modification of the GT estimator and demonstrate why it has bad
%performance.
Recall from \prettyref{lem:general_bounds}
that designing a good linear estimator boils to approximating $1-e^{-yt}$ by an analytic function $h(y)=\sum_{i\geq 1} \frac{h_i}{i!} y^i$ such
that all its derivatives at zero are small, namely, $\sup_{i\geq 1} |h_i|$ is small. The GT estimator %\prettyref{eq:GT} 
corresponds to the perfect approximation
\[
h^{\scriptscriptstyle\rm GT}(y) = 1-e^{-yt};
\]
however, $\sup_{i\geq 1} |h_i| = \max(t,t^\infty)$, which is infinity if $t > 1$ and leads to large variance.
%Suppose we are interested in probabilities that are
%very small, then it is sufficient to approximate $1-e^{-yt}$ for small
%values $y \ll 1$. For these values, a natural approach would be to use
%the $\ell$-term Taylor series expansion of $1-e^{-yt}$ at $0$, namely,
To avoid this situation, a natural approach is to use use the $\ell$-term Taylor expansion of $1-e^{-yt}$ at $0$, namely,
\begin{equation}
h^\fixed(y) = - \sum^{\fixed}_{i=1} \frac{(-yt)^i}{i!},
	\label{eq:h-ell}
	\end{equation}
	which corresponds to the estimator $U^\ell$ defined in \prettyref{eq:UGT-ell}.
Then $\sup_{i \geq 1} |h_i| = t^{\fixed}$ and, by
Lemma~\ref{lem:general_bounds}, the variance is at most $n (t^{\fixed}
+ t)$. Hence if ${\fixed} \leq \log_t m$, the variance is at most
$n(m+t)$.  However, note that the $\fixed$-term Taylor approximation is a degree-$\ell$ polynomial which eventually diverges and 
deviates from $1-e^{-yt}$ as $y$ increases, thereby incurring a large bias. \arxiv{Figure~\ref{fig:taylor}}{(Fig. S2a)} illustrates this phenomenon by plotting the function $1-e^{-yt}$ and its Taylor expansion
with $5,10,$ and $20$ terms. 
\begin{figure}[ht]
\centering
\subfigure[]{\label{fig:taylor}
\arxiv{\includegraphics[scale=0.4]{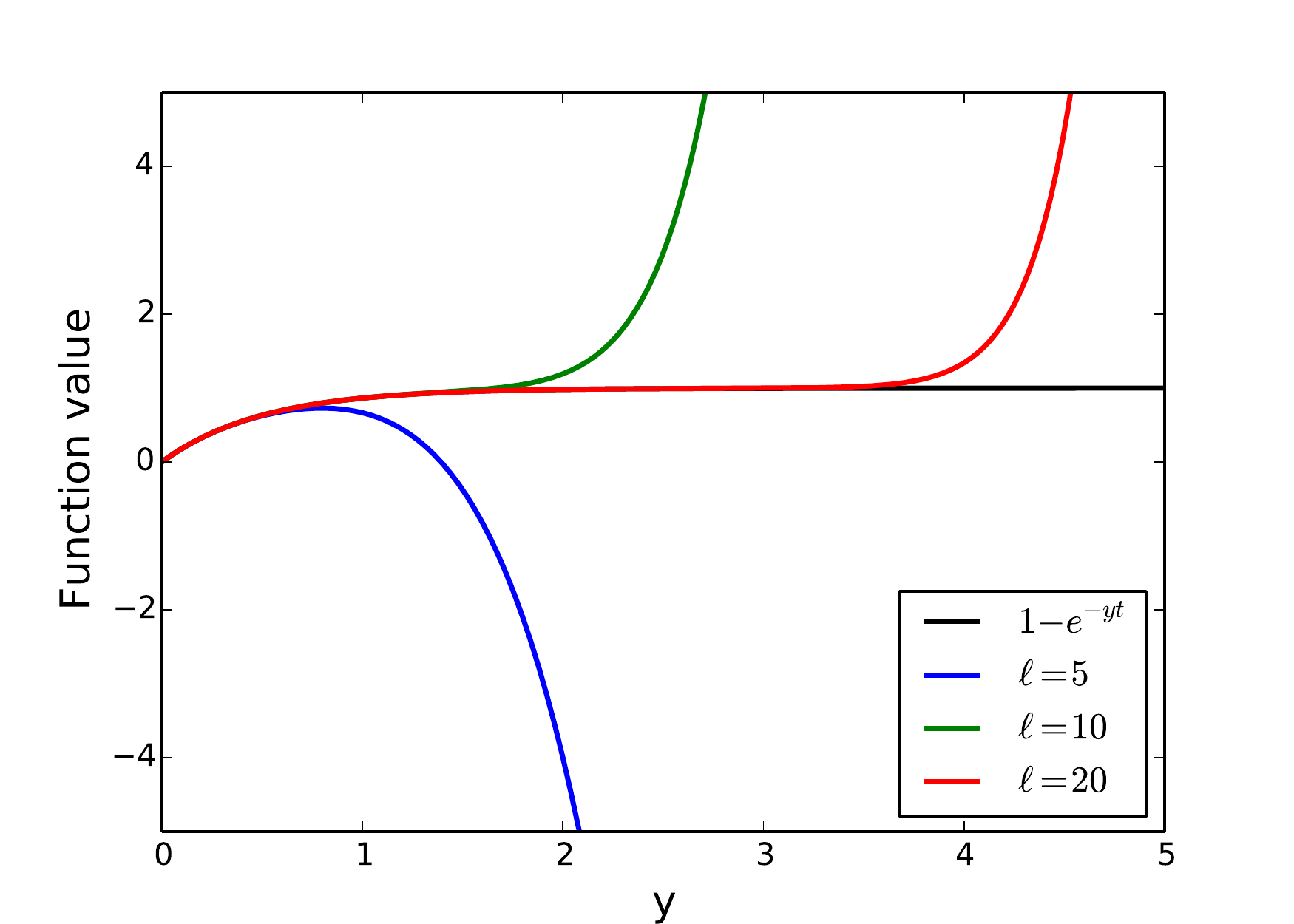}}{\includegraphics[scale=0.4]{Taylor.pdf}}}
\subfigure[.]{\label{fig:average_taylor}
\arxiv{\includegraphics[scale=0.4]{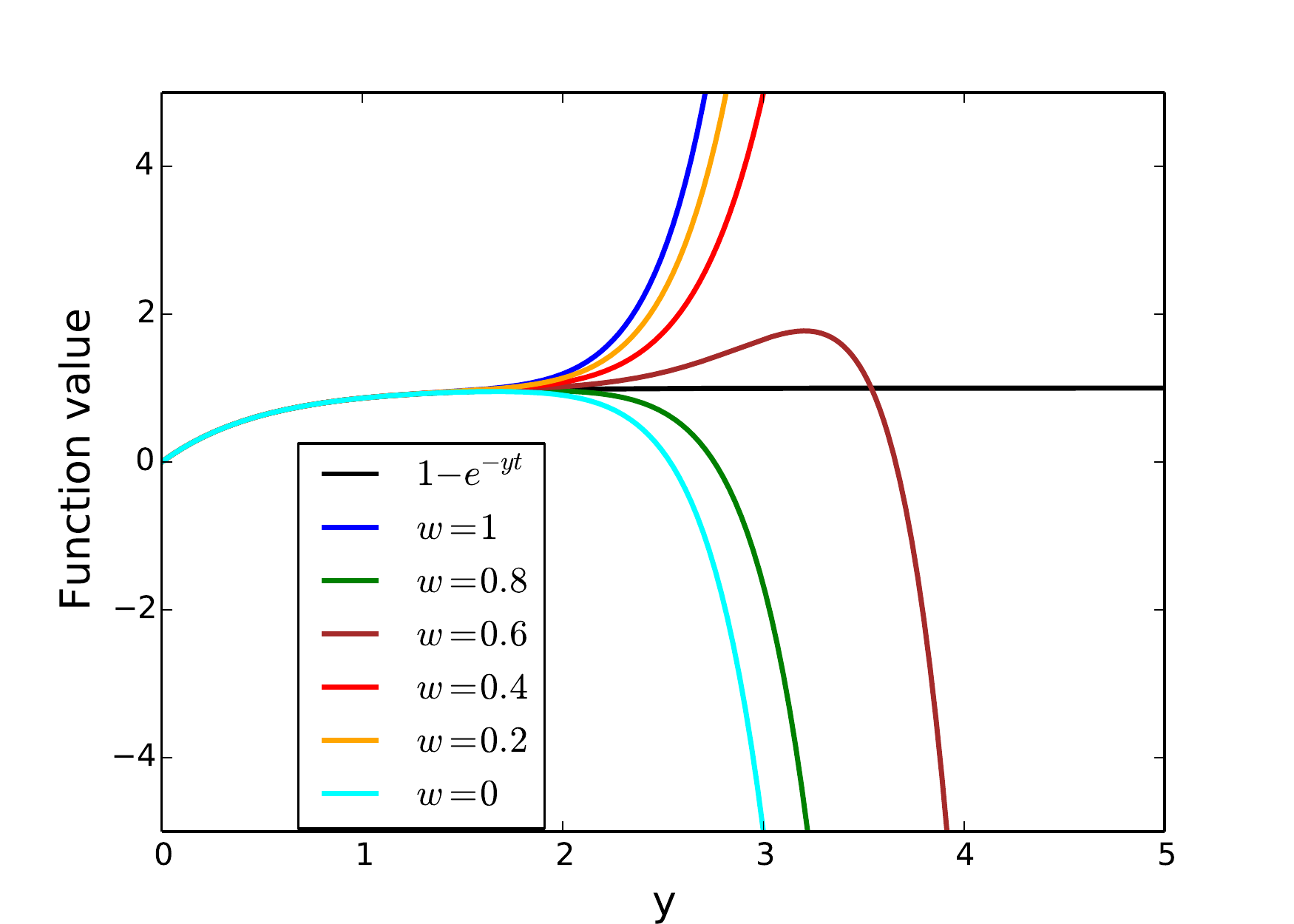}}{\includegraphics[scale=0.4]{Average_Taylor.pdf}}}
\caption{(a) Taylor approximation for $t = 2$, (b) Averages of $10$ and $11$ term Taylor approximation $t = 2$. }
\label{fig:taylor_all}
\end{figure}
%Notice that as $y$ increases the approximation deviates from $1-e^{-yt}$ for each approximation. 
Indeed, the next result (proved in \prettyref{app:ell}) rigorously shows that the NMSE of truncated GT estimator never vanishes:
\begin{Lemma}
\label{lmm:UGT-ell}	
There exist a constant $c > 0$ such that for any $\ell \geq 0$, any $t>1$ and any $n\in\naturals$,
	\[
	\LElnt \geq \frac{c(t-1)^5}{ t^4}.
	\]
\end{Lemma}
% Hence, we wish to construct
%a function such that the following two properties hold.
%\begin{itemize}
%\item 
%$h(y)$ around $0$ should behave $yt$.
%\item 
%$$ for large values of $y$ should be bounded. 
%\end{itemize}

%\subsection{Good-Toulmin estimator with smoothed smoothing}
\subsection{Smoothing by random truncation}
\label{sec:smooth}
As we saw in the previous section, the $\fixed$-term Taylor 
approximation, where all the coefficients after the $\ell\Th$ term
are set to zero results in large bias. Instead, one can choose a weighted average of several Taylor
series approximations, whose biases cancel each other leading to significant bias reduction. %This is the main idea of smoothed smoothing discussed in \prettyref{sec:new}. 
For example, in \arxiv{Figure~\ref{fig:average_taylor}}{(Fig. S2b)},
%Figure~\ref{fig:average_taylor},
 we plot
\[
w  h^{10} + (1-w) h^{11}
\]
for various values of $w\in[0,1]$. Notice that the weight $w=0.6$ leads to better approximation of $1-e^{-yt}$ than both $h^{10}$ and $h^{11}$.

A natural generalization of the above argument entails taking the weighted average of various Taylor approximations
with respect to a given probability distribution over $\Zplus
\ed \{0,1,2,\ldots\}$. For a $\Zplus$-valued random variable $\rand$, consider the power series
%approximations of the form,  
\[
\hrand(y) = \sum^{\infty}_{\fixed=0} \Pr(\rand = \fixed) \cdot h^{\fixed}(y),
\]
where $h^\ell$ is defined in	\prettyref{eq:h-ell}.
Rearranging terms, we have
\[
\hrand(y) = \sum^\infty_{\fixed=0} \Pr(\rand =\fixed) \sum^{\fixed}_{i=1} \frac{-(-yt)^i}{i!}
= - \sum^\infty_{i=1}  \frac{(-yt)^i}{i!} \Pr(\rand \geq i).
\]
Thus, the linear estimator with coefficients
\ignore{
A natural way to reduce the bias
would be to use a $\rand$-term Taylor series approximation, where
$\rand$ is distributed according to some known distribution over $\Zplus
\ed \{0,1,2,\ldots\}$.  Let $\U'$ be the estimator whose coefficients
are given by the formal power series
\[
h' = -\sum^{\rand}_{i=1} \frac{(-yt)^i}{i!}.
\]
However, by Jensen's inequality,
\[
\EE[(\U' - \U)^2] \geq \EE[(\EE_\rand[\U'] - \U)^2].
\]
Thus, instead of using $\U'$, we can use $\EE_\rand[\U']$ which has a
smaller mean squared loss. Thus, we propose to use estimators of the
form $\Urand \ed \EE_\rand[\U']$, where $\rand$ is distributed
according to some known distribution.
\[
\hrand = \sum^\infty_{\fixed=0} \Pr(\rand =\fixed) \sum^{\fixed}_{i=1} \frac{-(-yt)^i}{i!}
= - \sum^\infty_{i=1}  \frac{(-yt)^i}{i!} \Pr(\rand \geq i).
\]
Thus, randomizing the cutoff point $L$ amounts to attenuating the coefficients of the GT estimator by the tail probability of $L$, resulting in a linear estimator with the following coefficients:
}%
\begin{equation}
	\hrand_i = -(-t)^i \prob{\rand \geq i},
	\label{eq:hrandi}
\end{equation}
is precisely the SGT estimator $\UL$ defined in \prettyref{eq:UL}. Special cases of smoothing distributions include:
\begin{itemize}
\item $L=\infty$: This corresponds to the original Good-Toulmin estimator \prettyref{eq:GT} without smoothing;
	\item $L=\ell$ deterministically: This leads to the estimator $\U^\ell$ in \prettyref{eq:UGT-ell} corresponding to the $\ell$-term Taylor approximation;
\item $L\sim \Binom(\binlmt, 1/(1+t))$: This recovers the Efron-Thisted estimator \prettyref{eq:UET}, where $\binlmt$ is a tuning parameter to be chosen.
\end{itemize}
\begin{figure}[ht]
\centering
\subfigure[]{\label{fig:errors}\includegraphics[scale=0.5]{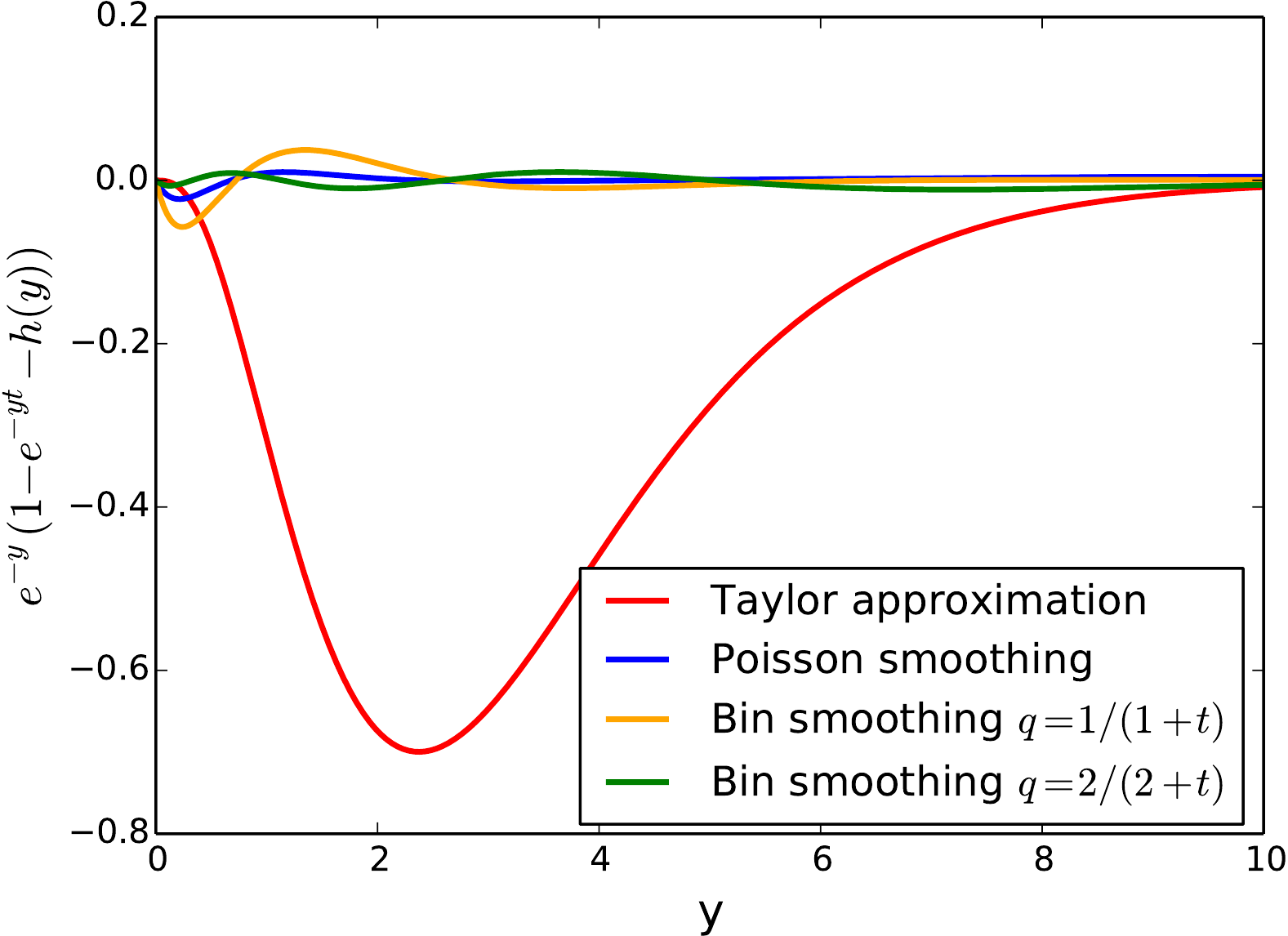}}
\subfigure[]{\label{fig:coeff}\includegraphics[scale=0.5]{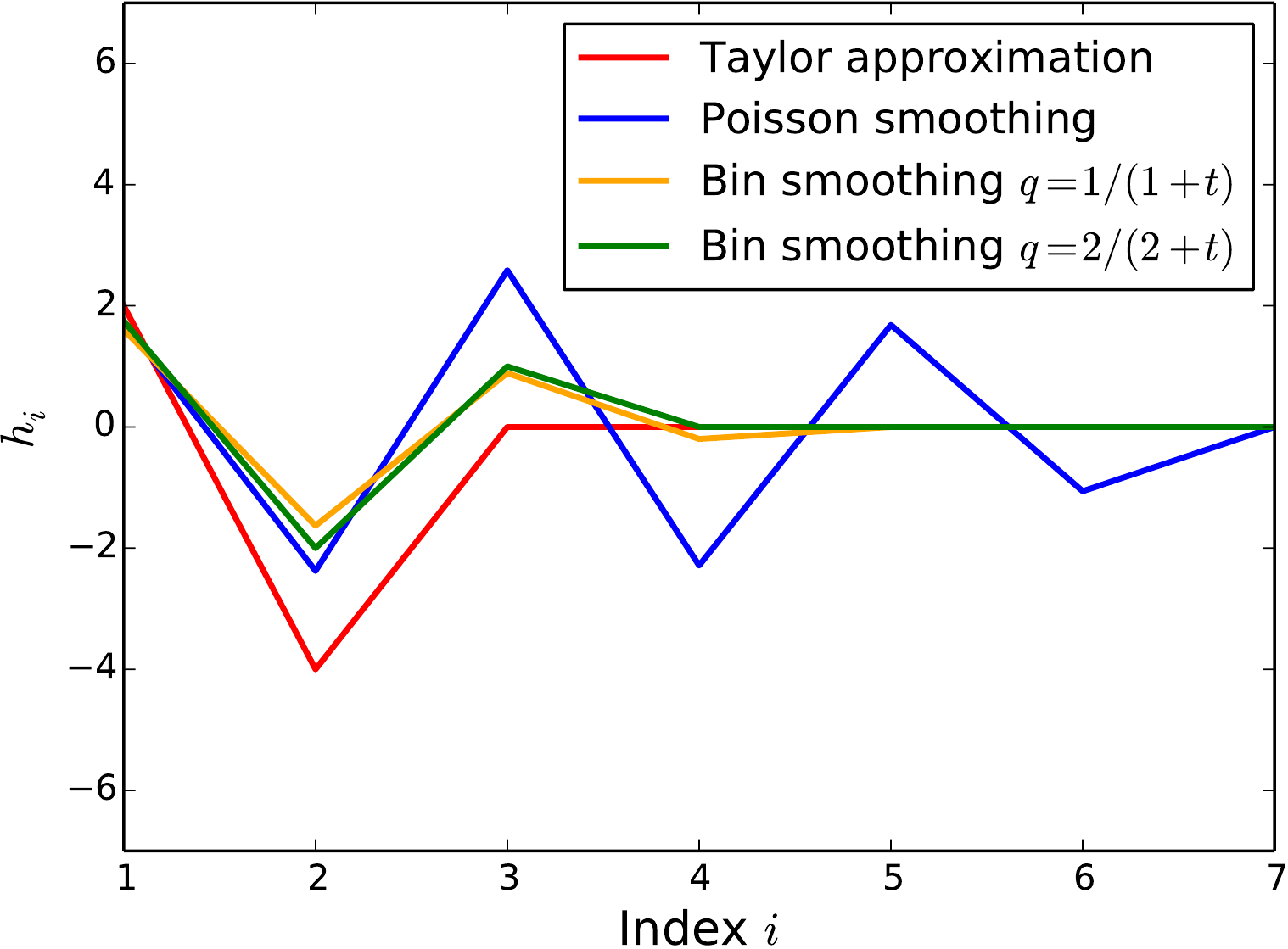}}
%\subfigure[]{\label{fig:errors}
%{\arxiv{\includegraphics[scale=0.5]{diff_est_approx-crop.pdf}}
%{\includegraphics[scale=0.5]{diff_est_approx-crop.pdf}}}}
%\subfigure[]{\label{fig:coeff}{\arxiv{\includegraphics[scale=0.5]{Coeff-crop.pdf}\label{fig:coeff}}{\includegraphics[scale=0.5]{Coeff-crop.pdf}\label{fig:coeff}}}}
\caption{Comparisons of approximations of $\hL(\cdot)$ with  $\EE[\rand] = 2$ and $t=2$. (a) $e^{-y}(1-e^{-yt} - \hL(y))$ as a function of $y$. (b) Coefficients $\hL_i$ as a function of index $i$.}
\end{figure}
We study the performance of linear
estimators corresponding to the Poisson smoothing and the Binomial smoothing. To this
end, we first systematically upper bound the bias and variance for any
probability smoothing $\rand$.  We plot the error that corresponds to each
smoothing in %Figure~\ref{fig:errors}
\arxiv{Figure~\ref{fig:errors}}{(Fig. S3a)}. Notice that the Poisson and binomial
smoothings have significantly small error compared to the Taylor series
approximation. The coefficients of the resulting estimator is plotted
in %Figure~\ref{fig:coeff}
\arxiv{Figure~\ref{fig:coeff}}{(Fig. S3b)}. It is easy so see that the maximum absolute
value of the coefficient is higher for the Taylor series approximation
compared to the Poisson or binomial smoothings.
\begin{Lemma}
For a random variable $\rand$ over $\Zplus$ and $t \geq 1$, 
\[
\Var(\Urand-\U) \leq \EE[\prev{+}] \cdot  {\EE}^2[t^{L}] + \EE[\U].
\]
\end{Lemma} 
\begin{proof}
By Lemma~\ref{lem:general_bounds}, to bound the variance it suffices to bound the highest coefficient in $\hrand$. 
\begin{equation}
\label{eq:coeff_rand}
|\hrand_i| \leq   t^i  \Pr(\rand \geq i ) = t^i \sum^\infty_{j=i}
\Pr(\rand =j) \leq   \sum^\infty_{j=i}
\Pr(\rand =j)  t^j  \leq \EE[t^\rand].
\end{equation}
The above bound together with Lemma~\ref{lem:general_bounds} yields the result.
\end{proof}
To bound the bias, we need few definitions. Let 
\begin{equation}
  g(y) \ed - \sum_{i=1}^\infty \frac{\prob{\rand \geq i}}{i!} (-y)^i.
	\label{eq:EGF}
\end{equation}
Under this definition, $\hrand(y) = g(yt)$. 
%Let  
%\[
%\Delta(y) \ed g(y) - (1-e^{-y}).
%\]
We use the following auxiliary lemma to bound the bias. 
\begin{Lemma}
	\label{lem:Delta}
	For any random variable $\rand$ over $\Zplus$, 
\begin{equation*}
	g(y) - (1-e^{-y}) = -e^{-y} \int_0^y \EE \left[{\frac{(-s)^\rand}{\rand!}}\right] e^s \diff s.
%	\label{eq:Deltay}
\end{equation*}
\end{Lemma}
\begin{proof}
Subtracting %Equation
\prettyref{eq:EGF} from the Taylor series expansion of $1-e^{-y}$
%from \prettyref{eq:EGF-ideal}
, 
	\begin{align*}
%	\Delta(y)
g(y) - (1-e^{-y})
= & ~ \sum_{i=1}^\infty \frac{\prob{\rand < i}}{i!} (-y)^i	\\
= & ~ \sum_{i=1}^\infty \sum_{j=0}^{i-1} \frac{(-y)^i}{i!} \prob{\rand = j}	\\
= & ~ \sum_{j=0}^\infty \pth{\sum_{i=j+1}^{\infty} \frac{(-y)^i}{i!} } \prob{\rand = j}.
\end{align*}
Note that $\sum_{i=j+1}^{\infty} \frac{z^i}{i!}$ can be expressed (via incomplete Gamma function) as
\[
\sum_{i=j+1}^{\infty} \frac{z^i}{i!} = \frac{e^{z}}{j!} \int_0^z \tau^j e^{-\tau} \diff \tau.
\]
Thus by Fubini's theorem,
	\begin{align*}
%	\Delta(y)
g(y) - (1-e^{-y})
= & ~ \sum_{j=0}^\infty \frac{e^{-y}}{j!} \int_0^{-y} \tau^j e^{-\tau} \diff \tau \prob{\rand = j}	\\
= & ~ e^{-y} \int_0^{-y} e^{-\tau}\diff  \tau \pth{\sum_{j=0}^\infty \frac{\tau^j}{j!}    \prob{\rand = j}}	\\
= & ~ -e^{-y} \int_0^{y} e^{s} \diff  s \pth{\sum_{j=0}^\infty \frac{(-s)^j}{j!}    \prob{\rand = j}}	\\
= & ~ -e^{-y} \int_0^y \EE \left[{\frac{(-s)^\rand}{\rand!}}\right] e^s \diff s. \qedhere
\end{align*}
\end{proof}
To bound the bias, we need one more definition. For a random variable $\rand$ over  $\Zplus$, let
\[
\DLt \ed  \max_{0 \leq s < \infty} \left  \lvert \EE \left[{\frac{(-s)^\rand}{\rand!}}\right] \right \rvert e^{-s/t},
\]
\begin{Lemma}
\label{lem:rand_bias}
For a random variable $\rand$ over $\Zplus$,
\[
|\EE[\Urand - \U]| \leq (\EE[\prev{+}] + \EE[\U])\cdot \DLt.
\]
\end{Lemma}
\begin{proof}
By Lemma~\ref{lem:Delta},
\begin{align*}
|g(y) - (1-e^{-y})| 
& \leq e^{-y} \int_0^y \left  \lvert \EE \left[{\frac{(-s)^\rand}{\rand!}}\right] \right \rvert e^s \diff s \\
& \leq \max_{s \leq y}  \left \lvert \EE \left[{\frac{(-s)^\rand}{\rand!}}\right]  \right \rvert e^{-y} \int_0^y e^s \diff s \\
&=    \max_{s \leq y} \left  \lvert \EE \left[{\frac{(-s)^\rand}{\rand!}}\right] \right \rvert (1-e^{-y})   .
\end{align*}
For a symbol $x$, 
\begin{align*}
e^{-\lambda_x} \left( \hrand(\lambda_x) - (1-e^{-\lambda_x t}) \right)
= e^{-\lambda_x} \left( g(\lambda_xt) - (1-e^{-\lambda_x t}) \right).
\end{align*}
Hence,
\begin{align*}
\lvert e^{-\lambda_x} \left( \hrand(\lambda_x) - 1-e^{-\lambda_x t} \right) \rvert
& \leq  (1-e^{-\lambda_xt}) \max_{0 \leq y \leq \infty} e^{-y} \max_{0 \leq s \leq yt} \left  \lvert \EE \left[{\frac{(-s)^\rand}{\rand!}}\right] \right \rvert \\
& \leq  (1-e^{-\lambda_xt}) \max_{0 \leq s \leq \infty} \left  \lvert \EE \left[{\frac{(-s)^\rand}{\rand!}}\right] \right \rvert e^{-s/t}.
\end{align*}
The lemma follows by summing over all the symbols and substituting
$\sum_x 1-e^{-\lambda_xt} \leq \sum_x 1-e^{-\lambda_x(t+1)} =  \EE[\prev{+}] + \EE[\U]$.
\end{proof}
The above two lemmas yield our main result.
\begin{Theorem}
\label{thm:general_bound}
For any random variable $\rand$ over $\Zplus$ and $t \geq 1$,
\[
\EE[(\Urand - \U)^2] \leq \EE[\prev{+}] \cdot {\EE}^2[t^{L}] + \EE[\U]
+  (\EE[\prev{+}] + \EE[\U] )^2  \DLt^2.
\]
\end{Theorem}
We have therefore reduced the problem of computing mean-squared loss,
to that of computing expectation of certain function of the random
variable. We now apply the above theorem for Binomial and Poisson
smoothings. Notice that the above bound is distribution dependent and can
be used to obtain stronger results for certain distributions. However,
in the rest of the paper, we concentrate on obtaining minimax
guarantees.

\subsection{Poisson smoothing}
%	\label{sec:}

\begin{Corollary}
\label{cor:poi_tail}
%
%Let $\Upoir$ be the estimator corresponding to  $\rand \sim \poi(r)$. 
For $t \geq 1$, $\rand \sim \poi(r)$ with $r=  \frac{1}{2t} \log  \left(\frac{n(t+1)^2}{t-1} \right)$,
%\[
%\Expect[(\UL - \U)^2]  \leq \Expect[\prev{+}] \cdot e^{2r(t-1)} + \Expect[\U]
%+  (\Expect[\prev{+}]  + \Expect[U])^2  \cdot  e^{-2r}.
%\]
%Thus choosing $r = \frac{1}{2t} \log  \left(\frac{n(t+1)^2}{t-1} \right)$ yields
\[
\LELnt % \leq \frac{1}{(nt)^{1/t}} \cdot \left(\frac{t(t-1)}{(t+1)^2} \right)^{\frac{1-t}{t}} + \frac{1}{nt} 
\leq \frac{c_t}{n^{1/t}},
\]
where $0 \leq c_t \leq 3$ and $\lim_{t \to\infty} c_t = 1$.
\end{Corollary}
\begin{proof}
For $\rand \sim \poi(r)$,
\begin{equation}
\label{eq:poivar}
\Expect[t^L] = e^{-r} \sum^\infty_{\fixed=0} \frac{(r t)^{\fixed}}{\fixed!} = e^{r(t-1)}.
\end{equation}
Furthermore, 
\[
\EE \left[{\frac{(-s)^\rand}{\rand!}}\right] = e^{-r} \sum_{j=0}^\infty \frac{(-s r)^j}{(j!)^2} = e^{-r} J_0(2\sqrt{sr}),
\]
where $J_0$ is the Bessel function of first order which takes values in $[-1,1]$ cf.~\cite[9.1.60]{AS64}. Therefore
\begin{equation}
\label{eq:poibias}
\DLt \leq e^{-r}.
\end{equation}
Equations~\eqref{eq:poivar} and~\eqref{eq:poibias} together with Theorem~\ref{thm:general_bound} yields
\[
\Expect[(\UL - \U)^2]  \leq \Expect[\prev{+}] \cdot e^{2r(t-1)} + \Expect[\U]
+  (\Expect[\prev{+}]  + \Expect[U])^2  \cdot  e^{-2r}.
\]
% the first part of the result.
Since $\Expect[\prev{+}]  \leq n$ and $\Expect[\U] \leq nt$,
\[
\Expect[(\UL - \U)^2]  \leq n e^{2r(t-1)} + nt +  (n+nt)^2 e^{-2r}.
\]
Choosing $r = \frac{1}{2t} \log \frac{n(t+1)^2}{t-1}$ yields
\[
\LELnt  \leq \frac{1}{(nt)^{1/t}} \cdot \left(\frac{t(t-1)}{(t+1)^2} \right)^{\frac{1-t}{t}} + \frac{1}{nt},
%\leq \frac{c_t}{n^{1/t}},
\]
and the lemma with $c_t \ed  \frac{1}{t^{1/t}} \cdot \left(\frac{t(t-1)}{(t+1)^2} \right)^{\frac{1-t}{t}} + \frac{1}{t}$.
%with $0 \leq c_t \leq 3$ and $\lim_{t \to\infty} c_t = 1$.
\end{proof}
\subsection{Binomial smoothing}
We now prove the results when $\rand \sim \Binom(\binlmt,q)$. Our analysis holds
for all $q \in [0,2/(2+t)]$ and in this range, the performance of the estimator improves as $q$ increases, and hence the NMSE bounds are strongest for $q=2/(2+t)$.
Therefore, we consider binomial smoothing for two cases: the Efron-Thisted suggested value $q = 1/(1+t)$ and the optimized value $q = 2/(2+t)$.
\begin{Corollary}
\label{cor:bin_tail}
%Let $\Ubinkq$ be the estimator corresponding to
For $t \geq 1$ and $\rand \sim \Binom(\binlmt,q)$, if $\binlmt = \left \lceil \frac{1}{2} \log_2  \frac{nt^2}{t-1} \right \rceil \text{ and } q = \frac{1}{t+1}$,
then
\[
\LELnt \leq \frac{c_t}{n^{\log_2(1+1/t)}},
\]
where $c_t$ satisfies $0 \leq c_t \leq 6$ and $\lim_{t \to\infty} c_t = 1$; if $\binlmt = \left \lceil \frac{1}{2} \log_3  \frac{nt^2}{t-1} \right \rceil \text{ and } q = \frac{2}{t+2}$,
then
\[
	\LELnt \leq \frac{c'_t}{(nt)^{\log_3(1+2/t)}},
\]
where $c'_t$ satisfies $0 \leq c'_t \leq 6$ and $\lim_{t \to\infty} c'_t = 1$.
% and $q \leq \frac{2}{t+2}$,
%\[
%\Expect[(\UL - \U)^2]  \leq \Expect[\prev{+}] \cdot (1+q(t-1))^{2\binlmt} + \Expect[\U]
%+  (\Expect[\prev{+}]+\Expect[U])^2  \cdot  (1-q)^{2\binlmt}.
%\]
%Therefore, 
\ignore{
\begin{align*}
\binlmt = \left \lceil \frac{1}{2} \log_2  \frac{nt^2}{t-1} \right \rceil \text{ and } q = \frac{1}{t+1}
\Rightarrow  & ~ \LELnt \leq \frac{c_t}{n^{\log_2(1+1/t)}},	\\
\binlmt = \left \lceil \frac{1}{2} \log_3  \frac{nt^2}{t-1} \right \rceil \text{ and } q = \frac{2}{t+2} \Rightarrow   & ~ 	\LELnt \leq \frac{c'_t}{(nt)^{\log_3(1+2/t)}},
\end{align*}
where $c_t$ satisfies $c_t \leq 6$ for $t \geq 1$ and $\lim_{t \to \infty} c_t = 1$, and 
$c'_t $ satisfies $c'_t \leq 11$ for all $t \geq 1$ and $\lim_{t \to \infty} c'_t = 1$.
}
\end{Corollary}
\begin{proof}
If $\rand\sim \Binom(\binlmt,q)$, 
\[
\Expect[t^\rand] = \sum^\binlmt_{\fixed=0} {\binlmt \choose \fixed} (tq)^{\fixed}(1-q)^{\binlmt - \fixed} = (1+q(t-1))^{\binlmt}.
\]
Furthermore, 
\[
\EE \left[{\frac{(-s)^\rand}{\rand!}}  \right]
=  \sum_{j=0}^\binlmt \frac{(-s)^j}{j!} \binom{\binlmt}{j} (q)^j (1-q)^{\binlmt-j} = (1-q)^{\binlmt}L_\binlmt\pth{\frac{qs}{1-q}},
\]
where 
\begin{equation}
L_{\binlmt}(y) = 	\sum_{j=0}^\binlmt \frac{(-y)^j}{j!}  \binom{\binlmt}{j}
	\label{eq:lag}
\end{equation}
 is the Laguerre polynomial of degree $\binlmt$.
If $\frac{tq}{2(1-q)} \leq 1$, for any $s \geq 0$,
\[
e^{-\frac{s}{t}} \left \lvert \EE \left[{\frac{(-s)^\rand}{\rand!}}  \right] \right \rvert
\leq (1-q)^{\binlmt} e^{-\frac{s}{t}} e^{\frac{qs}{2(1-q)}}
\leq (1-q)^{\binlmt},
\]
where the second inequality follows from the fact cf.~\cite[22.14.12]{AS64} that 
for all $y \geq 0$ and all $k\geq 0$,
\begin{equation}
|L_k(y)| \leq e^{y/2}.	
	\label{eq:lagbound}
\end{equation}
Hence for $q\leq 2/(t+2)$,
\[
\Expect[(\UL - \U)^2]  \leq \Expect[\prev{+}] \cdot (1+q(t-1))^{2\binlmt} + \Expect[\U]
+  (\Expect[\prev{+}] +\Expect[U])^2  \cdot  (1-q)^{2\binlmt}.
\]
Since $\Expect[\U] \leq nt$ and $\Expect[\prev{+}] \leq n$,
\begin{equation}
\label{eq:gen_bin}
\Expect[(\UL - \U)^2]  \leq n\cdot (1+q(t-1))^{2\binlmt} + nt
+   (nt+n)^2  \cdot  (1-q)^{2\binlmt}.
\end{equation}
Substituting the Efron-Thisted suggested $q = \frac{1}{t+1}$ results in
\[
\LELnt \leq \left(\frac{2^{2\binlmt}}{nt^2} + \frac{(t+1)^2}{t^2}\right)\left( \frac{t}{t+1}\right)^{2\binlmt} + \frac{1}{nt}.
\] 
Choosing $k = \left \lceil \frac{1}{2}\log_2 \frac{nt^2}{t-1} \right \rceil $
yields the first result with $c_t \ed  \left(\frac{4}{t-1} + \left(\frac{t+1}{t}\right)^2 \right)  \cdot \left(\frac{t-1}{t^2} \right)^{\log_2(1+1/t)} + \frac{1}{t}$.
For the second result, substituting $q = \frac{2}{t+2}$ in \eqref{eq:gen_bin}
results in 
\[
\LELnt \leq \left(\frac{3^{2\binlmt}}{nt^2} + \frac{(t+1)^2}{t^2}\right)\left( \frac{t}{t+2}\right)^{2\binlmt} + \frac{1}{nt}.
\] 
Choosing $k =\left \lceil  \frac{1}{2} \log_3 \frac{nt^2}{t-1} \right \rceil $ yields the result with $c'_t \ed \left(\frac{9}{t-1} + \frac{(t+1)^2}{t^2} \right) \cdot \left(\frac{t-1}{t^2} \right)^{\log_3(1+2/t)} + \frac{1}{t}$ .
\end{proof}
In terms of the exponent, the  result is strongest for $\rand \sim
\Binom(\binlmt, 2/(t+2))$. Hence, we state the following asymptotic
result, which is a direct consequence of \prettyref{cor:bin_tail}:
\begin{Corollary}
\label{cor:asym}
For  $L \sim \Binom(\binlmt,q)$, $q =
\frac{2}{t+2}$,$\binlmt = \lceil \log_3 (\frac{nt^2}{t-1})\rceil$, and any fixed $\delta$, the maximum $t$ till which $\UL$ incurs a NMSE of $\delta$ is 
\[
\lim_{n \to \infty} \frac{\max\sets{t: \LELnt<\delta}}{\log n} 
\geq \frac{2}{\log
  3\cdot \log \frac{1}{\delta}}.
\]
\end{Corollary}
\begin{proof}
By Corollary~\ref{cor:bin_tail}, if $t \to\infty$, then
\begin{equation*}
\LELnt \leq (1+o(1)) n^{- \frac{2+o(1)}{t\log 3}}.
	\label{eq:convrate}
\end{equation*}
where $o(1)=o_t(1)$ is uniform in $n$. 
Consequently, if $t = (\alpha + o(1)) \log n$ and $n\to\infty$, then
\[
\limsup_{n \to \infty} \LELnt \leq e^{-\frac{2}{\alpha\log3}}.
% e^{-1/\alpha} \left(1 + \frac{0.4}{\alpha} + \frac{0.04}{\alpha^2} \right),
\]
Thus for any fixed $\delta$, the maximum $t$ till which $\UL$ incurs a NMSE of $\delta$ is 
\[
\lim_{n \to \infty} \frac{\max\sets{t: \LELnt<\delta}}{\log n} 
\geq \frac{2}{\log
  3\cdot \log \frac{1}{\delta}}. \qedhere
\]
\end{proof}
Corollaries~\ref{cor:poi_tail} and~\ref{cor:bin_tail} imply Theorem~\ref{thm:res_main}
for the Poisson model.
\section{Extensions to other models}
\label{sec:extensions}
Our results so far have been developed for the Poisson 
model. Next we extend them to the multinomial model (fixed sample size), the Bernoulli-product model,
and the hypergeometric model (sampling without replacement)
\cite{BF93}, for which upper bounds of NMSE for general smoothing distributions that are
analogous to \prettyref{thm:general_bound} are presented in Theorem~\ref{thm:multinomial}, \ref{thm:bernoulli} and \ref{thm:hyper}, respectively.
%In particular Theorem~\ref{thm:multinomial} bounds the for the multinomial model, Theorem~\ref{thm:bernoulli} bounds the mean square error for the Bernoulli model, and Theorem~\ref{thm:hyper} bounds the mean square loss for the hypergeometric model. 
Using these results, we obtain the NMSE for Poisson and Binomial
smoothings similar to Corollaries~\ref{cor:poi_tail}
and~\ref{cor:bin_tail}.  We remark that up to multiplicative
constants, the NMSE under multinomial and Bernoulli-product model are
similar to those of Poisson model; however, the NMSE under
hypergeometric model is slightly larger.
%We first start with the multinomial model.

\subsection{The multinomial model}
%Instead of Poissonized sampling, here we consider the model with fixed sample size , where
The multinomial model corresponds to the setting described in \prettyref{sec:intro}, where upon observing $n$ i.i.d. samples, the objective is
to estimate the expected number of new symbols $\UXnm$ that would be observed
if we took $m$ more samples.
%\footnote{For notational convenience, we shall use the same notation $\UXnm$
%for all Poisson, multinomial, Bernoulli product, and hypergeometric models. }
We can write the expected
number of new symbols as
\begin{equation*}
\label{eq:binomial}
\UXnm = \sum_{x} \indic_{\mul{x} = 0} \cdot \indic_{\mul{x}' > 0}.
%
%]= \sum_{x} \indic_{\mul{x} = 0} \cdot \left(1 - (1-p_x)^{m} \right).
\end{equation*}
As before we abbreviate \[
\U \ed \UXnm
\]
and similarly $\UE\ed\UEXnm$ for any estimator $E$.
The difficulty in handling multinomial distributions is that, unlike the Poisson model, the number of
occurrences of symbols are correlated; in particular, they sum up to
$n$. This dependence renders the analysis cumbersome.  In the multinomial
setting each symbol is distributed according to $\Binom(n,p_x)$ and hence
\[
\Expect[\indic_{\mul{x} = i}] = {n \choose i} p^i_x (1-p_x)^{n-i}.
\]
As an immediate consequence,
\[
\Expect[\prev{i}] = \EE\left[\sum_{x} \indic_{\mul{x} = i} \right] =
\sum_{x} {n \choose i} p^i_x (1-p_x)^{n-i}.
\]
We now bound the bias and variance of an arbitrary linear estimator $\Uh$.
%under the multinomial model using properties of function $\hJ$. 
 We
first show that the bias
$\Expect[\Uh - \U]$ under the multinomial model is close to that
under the Poisson model, which is $\sum_{x} e^{-\lambda_x} (
h(\lambda_x) - (1-e^{-t\lambda_x}) )$ as given in \prettyref{eq:Upoisson}.
% and controlled by \prettyref{lem:bias}. 
\begin{Lemma}
\label{lem:bias_multi}
The bias of $\Uh = \sum^\infty_{i=1} \prev{i} h_i$ satisfies
\[
\left\lvert \Expect[\Uh - \U] - \sum_{x} e^{-\lambda_x} \left( h(\lambda_x) - (1-e^{-t\lambda_x}) \right) \right \rvert \leq 2\sup_i |h_i| + 2.
\]
\end{Lemma}
\begin{proof}
First we recall a result on Poisson approximation: For $X\sim \Binom(n,p)$ and $Y\sim \poi(np)$,
%Let $X\sim \Binom(n,p)$ and $Y\sim\poi(np)$. Then
\begin{equation}
|\Expect[f(X)] - \Expect[f(Y)]| \leq 2p \sup_{i} |f(i)|,
	\label{eq:poi-approx}
\end{equation}
which follows from the total variation bound $\dTV(\Binom(n,p),\poi(np)) \leq p$ \cite[Theorem 1]{BH84} and the fact that $d_{\rm TV}(\mu,\nu) = \frac{1}{2}\sup_{\|f\|_\infty\leq1} \int f d\mu-\int f d\nu$. In particular, taking $f(x) = \indic_{x=0}$ gives
\begin{equation*}
0 \leq  e^{-np} - (1-p)^n \leq 2 p.
\end{equation*}
Note that the linear estimator can be expressed as $\Uh = \sum_x h_{\mul{x}}$.
Under the multinomial model, 
\begin{equation*}
\Expect[\Uh - \U] = \sum_x \EE_{\mul{x}\sim \Binom(n,p_x)}[h_{\mul{x}}] - \sum_x (1-p_x)^n (1-(1-p_x)^m).
	\label{eq:bias-hsub}
\end{equation*}
Under the Poisson model, 
\[
\sum_{x} e^{-\lambda_x} \left( h(\lambda_x) - (1-e^{-t\lambda_x}) \right) = \sum_x \EE_{\mul{x}\sim \poi(np_x)}[h_{\mul{x}}] - \sum_x e^{-np_x} (1-e^{-mp_x}).
\]
Then
\[
\left|\sum_x \EE_{\mul{x}\sim \Binom(n,p_x)}[h_{\mul{x}}] - \sum_x \EE_{\mul{x}\sim \poi(np_x)}[h_{\mul{x}}] \right| \overset{\prettyref{eq:poi-approx}}{\leq} 2 \sup_i |h_i| \sum_x p_x = 2 \sup_i |h_i|.
\]
Furthermore,
\begin{align*}
	& ~ \sum_x (1-p_x)^n (1-(1-p_x)^m) - \sum_x e^{-np_x} (1-e^{-mp_x}) \\
\leq & ~ \sum_x e^{-n p_x} ( e^{-mp_x} - (1-p_x)^m) \overset{\prettyref{eq:poi-approx0}}{\leq} \sum_x e^{-n p_x} 2 p_x \leq 2.
\end{align*}
Similarly, $ \sum_x (1-p_x)^n (1-(1-p_x)^m) - \sum_x e^{-np_x} (1-e^{-mp_x}) \geq -2$. Assembling the above proves the lemma.
\end{proof}

The next result bounds the variance.
\begin{Lemma}
\label{lem:variance_multi}
For any linear estimator $\Uh$,
\[
\Var(\Uh - \U) \leq 8 n \max\sth{\sup_{i\geq 1} h^2_i,1} + 8m.
\]
\end{Lemma}
\begin{proof}
Recognizing that 
$\Uh - \U$ is a function of $n+m$ independent random variables, namely,  $X_1,\ldots,X_{n+m}$ drawn i.i.d. from $p$, we apply Steele's variance inequality~\cite{S86} to bound its variance.
% Similar to Equation~\eqref{eq:without_expectation}, 
Similar to \prettyref{eq:bias-hsub}, 
\[
\Uh - \U = \sum_{x}  h_{\mul{x}} + \indic_{\mul{x}=0} \indic_{\mul{x}'> 0}
\]
Changing the value of any one of the first $n$ samples changes the multiplicities of 
two symbols, and hence the value of $\Uh - \U$ can change by at most $4\max(\max_{i \geq 1} |h_i|,1)$. 
Similarly, changing any one of the last $m$ samples changes the value of $\Uh-\U$ by at most four. 
%.  If a symbol changes its multiplicity, the value of
%$\Uhxnm - \Uxnm$ can change by at most $2\max(\max_{i \geq 1} |h_i|,1)$. 
%The function can change by at most $4\max(\sup_{i\geq 1} |h_i|,1)$.  
Applying Steele's inequality gives the lemma.
\end{proof}
Lemmas~\ref{lem:bias_multi} and~\ref{lem:variance_multi} are analogous to Lemma~\ref{lem:general_bounds}. Together with \eqref{eq:coeff_rand} and Lemma~\ref{lem:rand_bias}, we obtain the main result for the multinomial model.
\begin{Theorem}
\label{thm:multinomial}
For $t \geq 1$ and any random variable $\rand$ over $\Zplus$,
\[
\Expect[(\Urand-\U)^2] \leq 8 n\, {\EE}^2[t^\rand] + 8m + \left((n(t+1) \DLt 
+ 2 \Expect[t^\rand] + 2\right)^2.
\]
\end{Theorem}
Similar to Corollaries~\ref{cor:poi_tail} and~\ref{cor:bin_tail}, one can compute the NMSE for Binomial and Poisson smoothings. We remark that up to multiplicative constants the results are identical to those for the Poisson model.

\subsection{Bernoulli-product model}
\label{sec:bernoulli}
Consider the following species assemblage model. There are $k$
distinct species and each one can be found in one of $n$ independent
sampling units. Thus every species can be present in multiple sampling
units simultaneously and each sampling unit can capture multiple
species.  For example species $x$ can be found in sampling units $1,3$
and $5$ and species $y$ can be found in units $2,3$, and $4$.  Given
the data collected from $n$ sampling units, the objective
%Among $n$ units, suppose we observe $k_{\text{obs}}$ species, our objective
is to estimate the expected number of new species that would be
observed if we placed $m$ more units.

The aforementioned problem is typically modeled as by the
\emph{Bernoulli-product model}. Since, in this model each sample only
has presence-absence data, it is often referred to as incidence
model~\cite{CCG12}.  For notational simplicity, we use the same
notation as the other three models.  In Bernoulli-product model, for a
symbol $x$, $\mul{x}$ denotes the number of sampling units in which
$x$ appears and $\prev{i}$ denotes the number of symbols that appeared
in $i$ sampling units.
%Here, for a symbol $x$,  $p_x$ represents Bernoulli
%probabilities hence are between 0 and 1, but do not necessarily sum to
%one. Each sample $X_i$ is a \emph{sample unit} that may contain multiple
%symbols, with symbol $x$ appearing once with probability $p_x$ and
%absent from the sample with probability $1-p_x$, independently of
%all other symbols.
Given a set of distinct symbols (potentially infinite), each symbol
$x$ is observed in each sampling unit independently with probability
$p_x$ and the observations from each sampling unit are independent of
each other.
%  \nbr{double check if accurate. the previous does not make sense}
%Furthermore, the probability of observing symbols are
%indepenent of each other. 
To distinguish from the multinomial and Poisson sampling models where
each sample can be only one symbol, we refer to samples here as
sampling units.  Given the results of $n$ sampling units, the goal is
to estimate the expected number of new symbols that would appear in
the next $m$ sampling units. Let $\psum = \sum_{x} p_x$.  Note that
$\psum$ is also the expected number of symbols that we observe for
each sampling unit and need not sum to $1$. For example, in the
species application, probability of catching bumble bee can be $0.5$
and honey bee be $0.7$.

This model is significantly different from the multinomial model in
two ways. Firstly, here given $n$ sampling units the number of
occurrences of symbols are independent of each other. Secondly, $\psum
\ed \sum_x p_x$ need not be $1$.  In the Bernoulli-product model, the
probability observing each symbol at a particular sample is $p_x$ and
hence in $n$ samples, the number of occurrences is distributed
$\Binom(n,p_x)$.
%Since each symbol is distributed $\Binom(n,p_x)$, 
Therefore the probability that $x$ is be observed in $i$ sampling units is
\[
\Expect[\indic_{\mul{x} = i}] = {n \choose i} p^i_x (1-p_x)^{n-i},
\]
and an immediate consequence on the number of distinct symbols that
appear $i$ sampling units is
\[
\Expect[\prev{i}] = \EE\left[\sum_{x} \indic_{\mul{x} = i} \right] =
\sum_{x} {n \choose i} p^i_x (1-p_x)^{n-i}.
\]
Furthermore, the expected total number of symbols is $n\psum$ and hence
\begin{equation*}
\sum^n_{i=1} \Expect[\prev{i}] i = n \psum.
\end{equation*}
Under the Bernoulli-product model the objective is to estimate the
number of new symbols that we observe in $m$ more sampling
units and is
\[
\UXnm = \sum_{x} \indic_{\mul{x} = 0} \cdot \indic_{\mul{x}' > 0}.
\]
As before, we abbreviate \[
\U \ed \UXnm
\]
and similarly $\UE\ed\UEXnm$ for any estimator $E$.
Since the probabilities need not add up to $1$, we redefine our definition of $\LEnt$ as
\[
\LEnt
\ed
\max \EE_p\Paren{\frac{\U-\UE}{n t \psum}}^2.
\]
Under this model, the SGT estimator satisfy similar results to that of Corollaries~\ref{cor:poi_tail} and~\ref{cor:bin_tail}, up to multiplicative constants. 
The main ingredient is to bound the bias and variance (like Lemma~\ref{lem:general_bounds}).
We note that since the marginal of $N_x$ is $\Binom(n,p_x)$ under both the multinomial and the Bernoulli-product model, the bias bound follows entirely analogously as in \prettyref{lem:bias_multi}. The proof of variance bound is very similar to that of \prettyref{lem:general_bounds} and hence is omitted.
\begin{Lemma}
\label{lem:bias_bernoulli}
The bias of the linear estimator $\Uh$ is 
\[
\left\lvert \Expect[\Uh - \U] - \sum_{x} e^{-\lambda_x} \left( h(\lambda_x) - (1-e^{-t\lambda_x}) \right) \right \rvert \leq 2 \psum \pth{\sup_i |h_i| + 1},
\]
and the variance 
\[
\Var(\Uh - \U) \leq n \psum  \cdot \left(t + \sup_{i \geq 1} {h^2_i} \right).
\]
\end{Lemma}
%\begin{proof}
%From Equation~\eqref{eq:fdefinition}, recall that $f(x,j) =
%\sum^\infty_{i=1}\indic_{j = i} h_i + \indic_{j=0} (1-(1-p_x)^{nt})$
%and $\Uhxnm - \Uxnm = \sum_{x} f(x,\mul{x})$. The bias result follows
%from applying Lemma~\ref{lem:eachbound}. The variance bound is very
%similar to that of Lemma~\ref{lem:general_bounds} and hence is omitted.
%\end{proof}
The above lemma together with \eqref{eq:coeff_rand} and Lemma~\ref{lem:rand_bias} yields  the main result for the Bernoulli-product model.
\begin{Theorem}
\label{thm:bernoulli}
For any random variable $\rand$ over $\Zplus$ and $t \geq 1$,
\[
\Expect[(\Urand-\U)^2] \leq  n \psum  \cdot \left(t +{\EE}^{2}[t^{\rand}] \right) + (n(t+1)\psum \DLt + 2 \psum(\Expect[t^{\rand}] + 1))^2.
\]
\end{Theorem}
Similar to Corollaries~\ref{cor:poi_tail} and~\ref{cor:bin_tail}, one can compute the normalized mean squared loss for Binomial and Poisson smoothings. We remark that up to multiplicative constants the results would be similar to that for the Poisson model.

\subsection{The hypergeometric model}
\label{sec:hyper}
%In this section we consider the following problem: Given an urn of $\hypN$ colored balls, how to estimate the total number of distinct colors 
The hypergeometric model considers the population estimation problem with samples drawn without replacement. 
Given $n$ samples drawn uniformly at random, without replacement from a set $\{y_1,\ldots, y_{\hypN}\}$ of $\hypN$ symbols,
the objective is to estimate the number of new symbols that would be
observed if we had access to $m$ more random samples without replacement, where $n+m \leq \hypN$.
Unlike the Poisson,
multinomial, and Bernoulli-product models we have considered so far, where the samples are
independently and identically distributed, in the hypergeometric model the samples are \emph{dependent} hence a modified analysis is needed.

Let $\hypn_x \ed \sum^\hypN_{i=1} \indic_{y_i = x}$ be the number of occurrences of 
symbol $x$ in the $\hypN$ symbols, which satisfies $\sum _x \hypn_x=\hypN$. Denote by $N_x$ the number of times $x$ appears in the $n$ samples drawn without replacements, which is 
distributed according to the hypergeometric distribution $\hyper(\hypN,\hypn_x,n)$ with the following probability mass function:\footnote{We adopt the convention that $\binom{n}{k}=0$ for all $k<0$ and $k>n$ throughout.}
\[
\Prob(\mul{x} = i) = \frac{{\hypn_x \choose i}{\hypN-\hypn_x \choose n-i}}{{\hypN \choose n}}.
%, \quad 0 \leq i \leq \min\{n,n_x\}.
\]
We also denote the joint distribution of $\{N_x\}$, which is multivariate hypergeometric, by $\hyper(\{\hypn_x\},n)$.
Consequently,
\[
\Expect[\prev{i}] = \sum_{x} \Prob(\mul{x} = i) = \sum_x \frac{{\hypn_x \choose i}{\hypN-\hypn_x \choose n-i}}{{\hypN \choose n}}.
\]
Furthermore, conditioned on $\mul{x} = 0$, $\mul{x}'$ is distributed as $\hyper(\hypN-n,\hypn_x,m)$
and hence
\begin{equation}
\label{eq:Umean-without_replacement}
\Expect[\U] = \sum_{x} \Expect[\indic_{\mul{x}=0}] \cdot  \Expect[\indic_{\mul{x}' > 0} | \indic_{\mul{x} = 0}]
= \sum_{x} \frac{{\hypN-\hypn_x \choose n}}{{\hypN \choose n}} \cdot \left( 1- \frac{{\hypN-n-\hypn_x \choose m}}{{\hypN-n \choose m}} \right).
\end{equation}
As before, we abbreviate \[
\U \ed \UXnm
\]
which we want to estimate and similarly for any estimator $\UE\ed\UEXnm$.
We now bound the variance and bias of a linear estimator $\Uh$ under the hypergeometric model.
\begin{Lemma}
\label{lem:hyper_var}
For any linear estimator $\Uh$,
\[
\Var(\Uh - \U) \leq 12 n \sup_{i} h^2_i + 6n + 3m.
\]
\end{Lemma}
\begin{proof}
We first note that for a random variable $Y$ that lies in the
interval $[a,b]$,
\[
\Var(Y) \leq \frac{(a-b)^2}{4}.
\]

For notational convenience define $h_0=0$. 
Then $\Uh =\sum_x h_{N_x}$.
Let $Z = \sum\indic_{N_x=0}$ and $Z' = \sum\indic_{N_x=N_x'=0}$ denote the number of unobserved symbols in the first $n$ samples and the total $n+m$ samples, respectively. Then $U=Z-'Z$. Since the collection of random variables 
$\indic_{\mul{x}=0}$ indexed by $x$ are negatively correlated, we have
\[
\Var\bigl(Z) \leq \sum_{x} \Var(\indic_{\mul{x}=0}\bigr) = \sum_{x} \Expect[ \indic_{\mul{x}=0}(1 - \indic_{\mul{x}=0})] \leq \sum_x \EE\qth{\indic_{\mul{x}>0}} \leq n.
\]
Analogously, $\Var(Z') \leq n+m$ and hence
\[
\Var(\Uh-\U) = \Var(\Uh-Z+Z') \leq 3 \Var(\Uh) + 3 \Var(Z') + 3  \Var(Z) \leq 3 \Var(\Uh) + 6 n + 3m.
\]
Thus it remains to show 
\begin{equation}
\Var(\Uh) \leq 4 n \sup_{i} h^2_i.	
	\label{eq:varUh}
\end{equation}
 By induction on $n$, we show that for any $n \in \naturals$, 
any set of nonnegative integers $\{\hypn_x\}$ and any function $(x,k)\mapsto f(x,k)$ with $k\in\integers_+$ satisfying $f(x,0)= 0$, 
\begin{equation}\label{eq:var-induction}
\Var\pth{\sum_x f(x,N_x) } \leq 4 n \|f\|_\infty^2,
\end{equation}
where $\{N_x\} \sim \hyper(\{\hypn_x\},n)$ and $\|f\|_\infty = \sup_{x,k} |f(x,k)|$. Then the desired Equation~\prettyref{eq:varUh} follows from \prettyref{eq:var-induction} with $f(x,k) = h_k$.

 We first prove \prettyref{eq:var-induction} for $n=1$, in which case 
 exactly one of $\mul{x}$'s is one and the rest are zero. 
 Hence, 
$|\sum_x f(x,N_x)| \leq \|f\|_\infty$ and $\Var(\sum_x f(x,N_x)) \leq \|f\|_\infty^2$.
 
Next assume the induction hypothesis holds for $n-1$.
Let $X_1$ denote the first sample and let $\tilde N_x$ denote the number of occurrences of symbol $x$ in samples $X_2,\ldots,X_n$.
Then $N_x = \tilde N_x +\indic_{X_1=x}$. Furthermore,  conditioned on $X_1= y$,  $\{\tilde N_{x}\} \sim \hyper(\{\tilde \hypn_x\}, n-1)$, where $\tilde \hypn_x=\hypn_x-\indic_{x=y}$.
By the
law of total variance, we have
\begin{equation}
\Var\pth{\sum_x f(x,N_x)}= 
\Expect \qth{V(X_1)} + \Var \pth{g(X_1)}.	
	\label{eq:totalvar}
\end{equation}
where
\[
V(y) \ed \Var\pth{\sum_x f(x,N_x) \Bigg\vert X_1=y}, \quad g(y) \ed \Expect\qth{\sum_x f(x,N_x) \Bigg\vert
  X_1=y}
\]
%The first term in \prettyref{eq:totalvar} is at most $\sup_{y}
%\Var(\sum_{x} f(\hypn_x, \mul{x}, l_x)| X_1 = y)$.
For the first term in \prettyref{eq:totalvar}, note that
\[
V(y) = \Var\pth{\sum_x f(x,\tilde N_x+\indic_{x=y}) \Bigg\vert X_1=y}  =  \Var\pth{\sum_x f_y(x,\tilde N_x) \Bigg\vert X_1=y}.
\]
where we defined $f_y(x,k) \ed f(x,k+\indic_{x=y})$. Hence, by the induction hypothesis,
$V(y) \leq 4(n-1)\|f_y\|_\infty^2 \leq 4(n-1)\|f\|_\infty^2$ and
$\Expect  \qth{ V(X_1)} \leq  4(n-1) ||f||^2_\infty$.

For the second term in \prettyref{eq:totalvar}, observe that for any $y\neq z$
\[
g(y)= \Expect[f(y,\tilde N_y+1)|X_1=y] + \Expect[f(z,\tilde N_{z})|X_1=y] + \EE\qth{\sum_{x \neq y,z} f(x, \tilde N_{x})\Bigg|
  X_1=y},
\]
and
\[
g(z)= \Expect[f(z,\tilde N_z+1)|X_1=z] + \Expect[f(y,\tilde N_{y})|X_1=z] + \EE\qth{\sum_{x \neq y,z} f(x, \tilde N_{x})\Bigg|
  X_1=z},
\]
Observe that $\{N_x\}_{x \neq y,z}$ have the same joint distribution conditioned on either $X_1=y$ or $X_1=z$ and hence
$\Expect[\sum_{x \neq y,z} f(x, \tilde N_{x})|
  X_1=y] = \Expect[\sum_{x \neq y,z} f(x, \tilde N_{x})|
  X_1=z]$.
  Therefore $|g(y) - g(z)| \leq 4 \|f\|_\infty$ for any $y\neq z$. 
  This implies that the function $g$ takes values in an interval of length at most $4 \|f\|_\infty$.
  Therefore $\Var(g(X_1)) \leq \frac{1}{4} (4 \|f\|_\infty)^2 = 4 \|f\|_\infty^2$.
  This completes the proof of \prettyref{eq:var-induction} and hence the lemma.
  \end{proof}

Let
%Let $B(h,n_x)$ be the bias corresponding to the estimator $\Uhxnm$ for
%a symbol with hypergeometric parameter $n_x$.
%\nb{What's the operational meaning of this function $B$? Bias?}
\begin{align*}
B(h,\hypn_x) \ed \sum^{\hypn_x}_{i=1} {\hypn_x \choose i} \left(\frac{n}{\hypN} \right)^i
\left(1-\frac{n}{\hypN} \right)^{\hypn_x-i} h_i - \left(1-\frac{n}{\hypN}
\right)^{\hypn_x} \left(1 - \left(1-\frac{m}{\hypN-n} \right)^{\hypn_x}\right).
\end{align*}
To bound the bias, we first prove an auxiliary result.
\begin{Lemma}
\label{lem:bin-hyper}
For any linear estimator $\Uh$,
\begin{align*}
 \left \lvert 
\Expect[\Uh - \U] -  \sum_{x}  B(h,\hypn_x)
 \right \rvert  \leq 4\max\pth{\sup_{i} |h_i|,1} + \frac{2\hypN}{\hypN-n}.
\end{align*}
%\begin{align*}
% \left \lvert 
%\Expect[\Uh - \U] - \sum_{x} \EE_{\Binom(\hypn_x,n/\hypN)}
%\left[ \sum_{i \geq 1} \indic_{\mul{x} = i} h_i - \indic_{\mul{x}=0} \left(1- \left(\frac{\hypN-n-m}{\hypN-n} \right)^{\hypn_x} \right)\right] \right \rvert  \leq 4\max(\sup_{i} |h_i|,1) + \frac{2\hypN}{\hypN-n}.
%\end{align*}
\end{Lemma}
\begin{proof}
%Recall that $\hypN_x\sim\hyper(\hypN,n_x,n)$ and $N_x'\sim\hyper(\hypN-n,\hypn_x,m)$. Denote by $\tilde{N}_x$ and $\tilde{N}'_x$ their binomial version, that is, $\tilde{N}_x\sim\Binom(\hypn_x,n/\hypN)$ and $\tilde{N}'_x\sim\Binom(\hypn_x,m/(\hypN-n))$. 
Recall that $N_x\sim\hyper(\hypN,\hypn_x,n)$. Let $\tilde{N}_x$ be a random variable distributed as $\Binom(\hypn_x,n/\hypN)$.
Since $\hyper(\hypN,\hypn_x,n)$ coincides with $\hyper(\hypN,n,\hypn_x)$, we have
\[
\dTV(\Binom(\hypn_x,n/\hypN),\hyper(\hypN,\hypn_x,n)) 
= \dTV(\Binom(\hypn_x,n/\hypN),\hyper(\hypN,n,\hypn_x)) 
\leq \frac{2\hypn_x}{\hypN},
\]
where the last inequality follows from \cite[Theorem 4]{DF80}.
 Since $d_{\rm TV}(\mu,\nu) = \frac{1}{2}\sup_{\|f\|_\infty\leq1} \int f d\mu-\int f d\nu = \sup_E \mu(E)-\nu(E)$, we have
%First we recall a result on Hypergeometric approximation:
%Let $X\sim \Binom(n,p)$ and $Y\sim\poi(np)$. Then
\begin{equation}
\left|\Expect[f(N_x)] - \Expect[f(\tilde{N}_x)]\right| \leq \frac{4\hypn_x}{\hypN} \sup_{i} |f(i)|, 
	\label{eq:hyper-approx}
\end{equation}
and
\begin{equation}
\left \lvert \frac{{\hypN-n-\hypn_x \choose m}}{{\hypN-n \choose m}} - \left(1-\frac{m}{\hypN-n}\right)^{\hypn_x} \right \rvert
\leq \dTV(\Binom(\hypn_x,m/(\hypN-n)),\hyper(\hypN-n,m,\hypn_x)) \leq \frac{2\hypn_x}{\hypN-n}.
	\label{eq:poi-approx0}
\end{equation}
Define $f_x(i) = h_i - \indic_{i=0} \left(1 - \left(1-\frac{m}{\hypN-n}\right)^{\hypn_x} \right)$. In view of \prettyref{eq:Umean-without_replacement} and the fact that $\sum \hypn_x=\hypN$, we have
\[
\left \lvert \Expect[\Uh - \U] - \sum_{x} \Expect[f_x(N_x)]\right \rvert \leq \frac{2\hypN}{\hypN-n}.
\]
Applying \prettyref{eq:hyper-approx} yields
\[
\sum_{x} \left \lvert  \Expect[f_x(\tilde{N}_x) ] 
- \EE\left[f_x(\mul{x})\right]\right \rvert \leq 4 \sup_{i} |f_x(i)| \leq
4\max\pth{\sup_{i} |h_i|,1}. 
\]
The above equation together with \eqref{eq:poi-approx0} results in the lemma since $B(h,\hypn_x)=\Expect[f_x(\tilde N_x)]$. 
\end{proof}
Note that to upper bound the bias, we need to bound $\sum_x B(h,\hypn_x)$.
It is easy to verify for the GT coefficients $\hFGT_i = -\left(-t \right)^{i}$ with $t=m/n$, $B(\hFGT,\hypn_x) = 0$. Therefore, if we choose $h=\hrand$ based on the tail of random variable $L$ with $\hrand_i = \hFGT_i \prob{\rand \geq i}$ as defined in \prettyref{eq:hrandi}, we have
\begin{align}
B(\hrand,\hypn_x) 
= & ~ \sum^{\hypn_x}_{i=1} {\hypn_x \choose i} \left(\frac{n}{\hypN} \right)^i
\left(1-\frac{n}{\hypN} \right)^{\hypn_x-i} (-t)^i \Prob(\rand < i)	\nonumber \\
= & ~ \left(1-\frac{n}{\hypN} \right)^{\hypn_x}
\sum^{\hypn_x}_{i=1} {\hypn_x \choose i} \left(- \frac{m}{\hypN-n} \right)^i \Prob(\rand < i).	\label{eq:Btail}
\end{align}

Similar to \prettyref{lem:Delta}, our strategy is to find an integral presentation of the bias. This is done in the following lemma.
\begin{Lemma}
\label{lem:hyper-equal}
For any $y \geq 0$ and any $k \in \naturals$, 
% 0 \leq y \leq 1$,
\begin{equation}
\sum^{k}_{i=1} {k \choose i} (-y)^i \Prob(\rand < i) =  - k(1-y)^{k}   \int^{y}_0 \EE \left[{k-1 \choose L} (-s)^L \right] (1-s)^{-k-1} d s.	
	\label{eq:hyper-equal}
\end{equation}
\end{Lemma}
\begin{remark}
\label{rmk:hyper-equal}
	For the special case of $y=1$, \prettyref{eq:hyper-equal} is understood in the limiting sense: Letting $\delta=1-y$ and $\beta=\frac{1-s}{\delta}$, we can rewrite the right-hand side as 
	\[- k \int^{1/\delta}_1 \EE \left[{k-1 \choose L} (\beta\delta-1)^L \right] k \beta^{-k-1} d \beta.
\] 
	For all $|\delta| \leq 1$ and hence $0 \leq 1 - \beta \delta \leq 2$,
	 we have \[
\left|\EE \left[{k-1 \choose L} (\beta\delta-1)^L \right]\right| = \Big|\EE \left[{k-1 \choose L} (\beta\delta-1)^L \indic_{L < k} \right]\Big| \leq 4^k.
\] 	By dominated convergence theorem, as $\delta\to0$, the right-hand side converges to $- \EE\qth{\binom{k-1}{L} (-1)^L}$ and coincides with the left-hand side, which can be easily obtained by applying $\binom{k}{i}=\binom{k-1}{i} + \binom{k-1}{i-1}$.
\end{remark}

\begin{proof}
Denote the left-hand side of \prettyref{eq:hyper-equal} by $F(y)$. Using $i\binom{k}{i}=k\binom{k-1}{i-1}$, we have
\begin{align}
F'(y)
= & ~ \sum^{k}_{i=1} {k \choose i} (-i)(-y)^{i-1} \Prob(\rand < i) 	= - k \sum^{k}_{i=1} {k-1 \choose i-1} (-y)^{i-1} \Prob(\rand < i) 	\nonumber \\
= & ~ - k \sum^{k}_{i=1} {k-1 \choose i-1} (-y)^{i-1} \Prob(\rand < i-1) 	- k \sum^{k}_{i=1} {k-1 \choose i-1} (-y)^{i-1} \Prob(\rand = i-1). \label{eq:Fprime}
\end{align}
The second term is simply $- k \EE \left[{k-1 \choose L} (-y)^L \right] \ed G(y)$. For the first term, since $L \geq 0$ almost surely and  $\binom{k}{i}=\binom{k-1}{i} + \binom{k-1}{i-1}$, we have
\begin{align}
k \sum^{k}_{i=1} {k-1 \choose i-1} (-y)^{i-1} \Prob(\rand < i-1)
= & ~ k \sum^{k}_{i=1} {k-1 \choose i} (-y)^{i} \Prob(\rand < i)
	\nonumber \\
= & ~ k \sum^{k}_{i=1} {k \choose i} (-y)^{i} \Prob(\rand < i) - k \sum^{k}_{i=1} {k-1 \choose i-1} (-y)^{i} \Prob(\rand < i) \nonumber	\\
= & ~ k F(y) - y F'(y). \label{eq:Fprime2}
\end{align}
Combining \prettyref{eq:Fprime} and \prettyref{eq:Fprime2} yields the following ordinary differential equation:
\[
F'(y)(1-y) + kF(y) = G(y), \quad F(0)=0, 
\]
%that is, $(F(y)(1-y)^{-k})'= G(y)(1-y)^{-n-1}$
whose solution is readily obtained as $F(y) = (1-y)^{k} \int_0^y (1-s)^{-k-1} G(s) d s$, \ie the desired Equation~\prettyref{eq:hyper-equal}.
\end{proof}
Combining \prettyref{lem:bin-hyper}--\ref{lem:hyper-equal}
%, and \prettyref{eq:Btail} 
yields the following bias bound:
\begin{Lemma}
\label{lem:hyper_bias}
For any random variable $L$ over $\Zplus$ and $t=m/n \geq 1$,
\[
|\Expect[\Urand - \U] | \leq nt \cdot \max_{0 \leq s \leq 1} \left \lvert  \EE \left[{\hypn_x-1 \choose L} (-s)^L \right] \right \rvert
+  4 \Expect[t^\rand] + \frac{2\hypN}{\hypN-n}.
\]
\end{Lemma}
\begin{proof}
Recall the coefficient bound \prettyref{eq:coeff_rand} that $\sup_i |h_i| \leq \Expect[t^\rand]$.
By Lemma~\ref{lem:bin-hyper} and the assumption that $t \geq 1$,
\[
\left \lvert 
\Expect[\Uh - \U] - \sum_{x} B(\hrand,\hypn_x) \right \rvert \leq  4\Expect[t^\rand] + \frac{2\hypN}{\hypN-n}.
\]
Thus it suffices to bound $ \sum_{x} B(\hrand,\hypn_x) $.
For every $x$, using \prettyref{eq:Btail} and applying \prettyref{lem:hyper-equal} with $y = \frac{m}{\hypN-n}$ and $k=\hypn_x$, we obtain
\begin{align*}
B(\hrand,\hypn_x)
& = - %\left(1-\frac{n}{N} \right)^{n_x} \left(1-\frac{m}{\hypN-n}\right)^{n_x} 
\left(1-\frac{n+m}{\hypN} \right)^{\hypn_x}\int^{\frac{m}{\hypN-n}}_0  \EE \left[{\hypn_x-1 \choose L} (-s)^L \right]  \hypn_x (1-s)^{-\hypn_x-1} d s.
\end{align*}
%\begin{align*}
%B(\hrand,\hypn_x)
%& = -\sum^{\hypn_x}_{i=1} {\hypn_x \choose i} \left(\frac{n}{N} \right)^i
%\left(1-\frac{n}{\hypN} \right)^{\hypn_x-i} (-t)^i \Prob(\rand < i) \\
%& = \left(1-\frac{n}{\hypN} \right)^{\hypn_x}  \sum^{\hypn_x}_{i=1} {\hypn_x \choose i} \left(\frac{-nt}{\hypN-n} \right)^i \Prob(\rand < i) \\
%& = -\left(1-\frac{n}{\hypN} \right)^{\hypn_x} \left(1-\frac{nt}{\hypN-n}\right)^{\hypn_x}  \int^{nt/(\hypN-n(t+1))}_0 \hypn_x  \EE \left[{\hypn_x-1 \choose L} (-s)^L (1+s)^{\hypn_x-L-1} \right] d s,
%\end{align*}
%where the last equality follows from Lemma~\ref{lem:hyper-equal}.
Since $0 \leq \frac{m}{\hypN-n} \leq1$, letting 
$K = \max_{0 \leq s \leq 1} \big|  \EE \big[{\hypn_x-1 \choose L} (-s )^L  \big] \big|$, 
%$K = \max_{0 \leq s \leq 1} \left \lvert  \EE \left[{\hypn_x-1 \choose L} \left(-s \right)^L  \right] \right \rvert$, 
we have
\begin{align*}
 |B(\hrand,\hypn_x)|
& \leq \left(1-\frac{n+m}{\hypN} \right)^{\hypn_x} K \int^{\frac{m}{\hypN-n}}_0 \hypn_x (1-s)^{-\hypn_x-1} d s. \\
%& \leq  \left(1-\frac{n+m}{\hypN} \right)^{\hypn_x} K \left(\left(\frac{\hypN-n}{\hypN-n-m}\right)^{\hypn_x} - 1 \right) \\
& = K
\left(\left(1-\frac{n}{\hypN} \right)^{\hypn_x} - \left(1-\frac{n+m}{\hypN} \right)^{\hypn_x}  \right) 
\leq  K
 \left(1-\frac{n}{\hypN} \right)^{\hypn_x-1} \frac{m\hypn_x}{\hypN}, \label{eq:claim}
\end{align*}
where the last inequality follows from the convexity of $x \mapsto (1-x)^{\hypn_x}$.
Summing over all symbols $x$ results in the lemma.
\end{proof}
Combining \prettyref{lem:hyper_bias} and \prettyref{lem:hyper_var} gives the following NMSE bound:
\begin{Theorem}
\label{thm:hyper}
%For any $\Zplus$-valued random variable $\rand$ and $t \geq 1$,
Under the assumption of \prettyref{lem:hyper_bias},
\[
\Expect[(\Urand - \U)^2] \leq 12(n+1){\EE}^2[t^\rand] + 6n + 3m +  \frac{12\hypN^2}{(\hypN-n)^2} + 3m^2 \max_{1 \geq\alpha > 0}  \left \lvert  \EE \left[{\hypn_x-1 \choose L} (-\alpha)^L \right] \right \rvert^2.
\]
\end{Theorem}
As before, we can choose various smoothing distribution and obtain upper bounds on the mean squared error.
\begin{Corollary}
\label{cor:hyper}
If $\rand\sim\poi(r)$ and $\hypN-n \geq m \geq n$, then
\[
\Expect[(\UL - \U)^2] \le 12(n+1)e^{2r(t-1)} +  3m^2 e^{-r} + 9m + 48.
\]
Furthermore, if $r = \frac{1}{2t-1} \cdot \log (nt^2)$,
\[
\LELnt
 \leq \frac{27}{(nt^2)^{\frac{1}{2t-1}}} + \frac{9nt + 48}{(nt)^2}.
\]
\end{Corollary}
\begin{proof}
For $\rand\sim\poi(r)$, $\Expect[t^\rand] = e^{r(t-1)}$ and 
\[
\max_{0 \leq \alpha \leq 1}  \left \lvert  \EE \left[{\hypn_x-1 \choose L} (-\alpha)^L \right] \right \rvert
= e^{-r} \max_{0 \leq \alpha \leq 1} |L_{\hypn_x-1} \left(\alpha r \right)| \leq e^{-r/2},
\]
where $L_{\hypn_x-1}$ is the Laguerre polynomial of degree $\hypn_x-1$ defined in \prettyref{eq:lag} and the last equality follows the bound \prettyref{eq:lagbound}.
Furthermore, $\hypN/(\hypN-n) = 1 + n/(\hypN-n) \leq 1 + n/m \leq 2$ and $n \leq m$, and hence the first part of the lemma.
The second part follows by substituting the value of $r$.
\end{proof}
%\hrule{\textwidth}
%\hrule height 2pt
%
%\begin{lem}
%\label{lmm:hyp}	
%	\[
%	\dTV(\hyper(\hypN,R,n), \poi(nR/\hypN)) \leq \frac{}{}
%	\]
%\end{lem}
\section{Lower bounds}
\label{sec:lower}
Under the multinomial model (i.i.d.~sampling), we lower bound the risk
$\LEnt$ for any estimator $\UE$ using the support size estimation
lower bound in~\cite{WY14b}. Since the lower bound in~\cite{WY14b}
also holds for the Poisson model, so does our lower bound.
% also holds for the Poisson model.

%The impossibility to predict the future too well follows from the minimax lower bound for estimating the support size proved in~\cite{WY14b}.
Recall that for a discrete distribution $p$, $S(p) = \sum_x \indic_{p_x > 0}$ denotes its support size. 
It is shown that given $n$ i.i.d.~samples drawn from a distribution $p$ whose minimum non-zero mass $p^+_{\min}$ is at least $1/k$, the minimax mean-square error for estimating $S(p)$ satisfies
\begin{equation}
\min_{\hat{S}} \max_{p : p^+_{\min} \geq 1/k} \Expect[(\hat{S} - S(p))^2] \geq c' k^2 \cdot \exp \left(- c \max\left(\sqrt{\frac{n\log k}{k}}, \frac{n}{k} \right) \right).	
	\label{eq:supplb}
\end{equation}
where $c,c'$ are universal positive constants with $c > 1$. We prove Theorem~\ref{thm:res_lb} under the multinomial model with $c$ being the universal constant from \prettyref{eq:supplb}. 

Suppose we have an estimator $\Ugm$ for $\Up$ that can accurately predict the number of new symbols arising in the next $m$ samples, 
we can then produce an estimator for the support size by adding the number of symbols observed, $\prev{+}$, in the current $n$ samples, namely,
\begin{equation}
\hat{S} = \Ugm +  \prev{+}.
	\label{eq:hatS}
\end{equation}
Note that $\Up = \sum_x \indic_{N_x=0} \indic_{N'_x>0}$.
When $m=\infty$, 
$U$ is the total number of unseen symbols and we have $S(p) = U + \prev{+}$. Consequently, if $\Ugm$ can foresee too far into the future (\ie for too large an $m$), then \prettyref{eq:hatS} will constitute a support size estimator that is too good to be true. 

%\ignore{This leads to the following impossibility result:
%Similarly it can be shown that $\supp + k e^{-n/k}\geq \Up + \sum_{i \geq 1} \prev{i} \geq \supp $. By choosing $k$ carefully, and using the lower bound \prettyref{eq:supplb}, we show that 
%\begin{Theorem}
%[Appendix~\ref{app:lower_bound}]
%\label{thm:lower_bound}
%Assume that $t = \frac{m}{n} \geq \max(c,1)$, where $c$ is the universal constant from \prettyref{eq:supplb}. 
%There exist universal positive constants $c_1,c_2$ such that 
%\begin{equation}
%\min_{E} \LEnt 
%= \min_{\Ugm} \max_p  \frac{1}{m^2} \EE_p \left[(\Ugm - \Up)^2\right] 
%\geq 
%c m^{ - \frac{n}{cm}}.	
%c_1 n^{ - \frac{c_2}{t}}.	
%\frac{c}{m^{c'n/m}}.	
%	\label{eq:lower_bound}
%\end{equation}
%There exist universal positive constants $\alpha,\beta$ such that 
%\begin{equation}
%%\min_{\Ugm}  \wclsss{\est}{n}{m} = 
%%\min_{\Ugm} \max_p  \frac{\EE \left[(\Up - \Ugm)^2\right]}{n^2t^2} \geq c.
%\Lsnm = \min_{\Ugm} \max_p  \frac{1}{m^2} \EE_p \left[(\Ugm - \Up)^2\right] \geq \min\sth{\alpha, \frac{4t^2}{\beta^2 \log^2 \frac{nt^2}{\beta}} \pth{\frac{\beta}{nt^2}}^{2\beta/t}}.	
%	\label{eq:lower_bound}
%\end{equation}
%\end{Theorem}
%}
Combining \prettyref{thm:res_lb} with the positive result (\prettyref{cor:poi_tail} or \ref{cor:bin_tail}) 
yields the following characterization of the minimax risk:
\begin{Corollary}
\label{cor:rate}
For all $t \geq c$, we have
\[
\inf_{\UE} \LEnt = \exp\pth{ - \Theta\pth{\max\sth{\frac{\log n}{t}, 1}}}
\]
Consequently, as $n\to\infty$, the minimax risk $\inf_{\UE} \LEnt \to 0$ if and only if $t = o(\log n)$.
\end{Corollary}

\begin{proof}[Proof of \prettyref{thm:res_lb}]
%	Suppose we have an estimator $\Ugm$, we produce an estimator for support
%as
%\[
%\hat{S} = \Ugm + \sum_{i\geq 1} \prev{i}.
%\]
%Below we abbreviate $\Up$ and $\Ugm$ as $\U$ and $\hat \U$, respectively.
Recall that $m=nt$.
%Denote the total number of symbols that appear in the $n$ samples by $S_{\rm seen} = \sum_{i\geq 1} \prev{i}$.
Let $\hat U$ be an arbitrary estimator for $U$.
For the support size estimator $\hat S = \hat{U} + \prev{+}$ defined in \prettyref{eq:hatS}, it must obey the lower bound \prettyref{eq:supplb}. Hence there exists some $p$ satisfying $p_{\min}^+ \geq 1/k$,
%either $p_x =0$ or $p_x \geq 1/k$ for all $x$, 
such that
\begin{equation}
\Expect[(S(p)- {\hat{S}})^2] \geq c' k^2 \cdot \exp \left(- c\max\left(\sqrt{\frac{n\log k}{k}}, \frac{n}{k} \right) \right).	
	\label{eq:supplb1}
\end{equation}
Let $S=S(p)$ denote the support size, which is at most $k$. 
Let $\tilde{U} \ed \EE_{X^{n+m}_{n+1}}[U]$ be the expectation of $U$ over the unseen samples $X^{n+m}_{n+1}$ conditioned on the available samples $X_1^n$. Then
$\tilde U = \sum_x \indic_{\mul{x} = 0} \left(1 - (1-p_x)^{nt} \right)$.
Since the estimator $\hat{U}$ is independent of $X^{n+m}_{n+1}$, by convexity,
\begin{equation}
\EE_{X^{n+m}_1}
[(U-\hat{U})^2] \geq 
\EE_{X^{n}_{1}}[(\EE_{X^{n+m}_{n+1}}[U-\hat{U}])^2] = \Expect[(\tilde{U}-\hat{U})^2].
\label{eq:UUtilde}
\end{equation}
Notice that with probability one,
\begin{equation}
|S - \tilde{U} - \prev{+} |\leq  Se^{-nt/k}	 \leq k e^{-nt/k},
%\Expect[(S - \U - \prev{+})^2]\leq  S^2e^{-2nt/k}	 \leq k^2 e^{-2nt/k}.
	\label{eq:Snewt}
\end{equation}
which follows from
\[
\tilde{U} +  \prev{+}
%=  \sum_{x} \indic_{\mul{x} = 0} \left(1 - e^{-ntp_x} \right) +  \indic_{\mul{x} > 0} 
=  \sum_{x: p_x > 0} \indic_{\mul{x} = 0} \left(1 - (1-p_x)^{nt} \right) +  \indic_{\mul{x} > 0}  \leq  S,
\]
and, on the other hand, 
\begin{align*}
\tilde{U} +  \prev{+}
&=  \sum_{x: p_x \geq 1/k} \indic_{\mul{x} = 0} \left(1 - (1-p_x)^{nt} \right) +  \indic_{\mul{x} > 0}  \\
& \geq  \sum_x \indic_{\mul{x} = 0} \left(1 - (1-1/k)^{nt} \right) +  \indic_{\mul{x} > 0}  \geq S(1 -  (1-1/k)^{nt}) \geq S(1-e^{-nt/k}).
\end{align*}
Expanding the left hand side of \prettyref{eq:supplb1},
\begin{align*}
\Expect[(S - {\hat{S}})^2]
& = \EE\qth{\pth{  S - \tilde{U} - \prev{+} + \tilde{U} - \hat{U} }^2} \leq  2\Expect[(S - \tilde{U} - \prev{+} )^2] + 2\Expect[ (\tilde{U} - \hat{U}))^2] \\
& \overset{\prettyref{eq:Snewt}}{\leq} 2 k^2 e^{-2nt/k}  + 2\Expect[ (\tilde{U} - \hat{U}))^2] \overset{\prettyref{eq:UUtilde}}{\leq} 2 k^2 e^{-2nt/k}  + 2\Expect[ (U - \hat{U}))^2] \\
\end{align*}
Let
\[
k = \min\sth{\frac{nt^2}{c^2 \log \frac{nt^2}{c^2}}, \frac{nt}{\log \frac{4}{c'}}},
\]
which ensures that
\begin{equation}
c' k^2 \cdot \exp \pth{- c \max\sth{ \sqrt{\frac{n\log k}{k}}, \frac{n}{k} } } \geq 4 k^2 e^{-2nt/k}.	
	\label{eq:lb1}
\end{equation}
 Then
\[
\Expect[ (\U - \hat{U})^2] \geq k^2 e^{-2nt/k},
\]
establishes the following lower bound  with 
%$\alpha \ed \log \frac{4}{c'}$ 
$\alpha \ed \frac{c'^2}{4 \log^2 (4/c')}$ 
and $\beta \ed c^2$:
\begin{equation*}
%\min_{\Ugnm}  \wclsss{\est}{n}{m} = 
%\min_{\Ugnm} \max_p  \frac{\EE \left[(\Unnp - \Ugnm)^2\right]}{n^2t^2} \geq c.
\min_{E} \LEnt
%= \min_{\Ugm} \max_p \frac{1}{m^2} \EE_p \left[(\Ugm - \Up)^2\right]
\geq \min\sth{\alpha, \frac{4t^2}{\beta^2 \log^2 \frac{nt^2}{\beta}} \pth{\frac{\beta}{nt^2}}^{2\beta/t}}.	
\end{equation*}
To verify \prettyref{eq:lb1}, since $t \geq c$ by assumption, we have
$\exp(\frac{2tn}{k} - \frac{cn}{k}) \geq \exp(\frac{nt}{k}) \geq \frac{4}{c'}$.  Similarly, 
since $k \log k \leq \frac{nt^2}{c^2}$ by definition, we have
$\frac{2nt}{k} \geq 2 c' \sqrt{\frac{n \log k}{k}}$ and hence $\exp\big(\frac{2tn}{k} - c\sqrt{\frac{n \log k}{k}}\big) \geq \exp(\frac{nt}{k}) \geq \frac{4}{c'}$, completing the proof of \prettyref{eq:lb1}.

Thus we have shown that there exist universal
positive constants $\alpha,\beta$ such that
\begin{equation*}
%\min_{\Ugnm}  \wclsss{\est}{n}{m} = 
%\min_{\Ugnm} \max_p  \frac{\EE \left[(\Unnp - \Ugnm)^2\right]}{n^2t^2} \geq c.
\min_{E} \LEnt
%= \min_{\Ugm} \max_p \frac{1}{m^2} \EE_p \left[(\Ugm - \Up)^2\right]
\geq \min\sth{\alpha, \frac{4t^2}{\beta^2 \log^2 \frac{nt^2}{\beta}} \pth{\frac{\beta}{nt^2}}^{2\beta/t}}.	
\end{equation*}
Let $y = \left( \frac{nt^2}{\beta}\right)^{2\beta/t}$, then
\[
\min_{E} \LEnt \geq \min\sth{\alpha, 16 \frac{1}{y\log^2 y}}.
\]
Since $y > 1$, $y^3 \geq y \log^2 y$ and hence for some constants $c_1, c_2 > 0$,
\[
\min_{E} \LEnt \geq \min\sth{\alpha, 16 \frac{1}{y^3}} 
\geq \min \sth{\alpha, \left( \frac{\beta}{nt^2}\right)^{6\beta/t}} 
\geq c_1 \min \sth{1,  \left( \frac{1}{n}\right)^{c_2/t}}
\geq \frac{c_1}{n^{c_2/t}}. \qedhere
\]
\end{proof}
\section{Experiments}
\label{sec:experiments}
We demonstrate the efficacy of our estimators by comparing their
performance with that of several state-of-the-art support-size estimators
currently used by ecologists: 
Chao-Lee estimator~\cite{C84, CL92},
Abundance Coverage Estimator (ACE)~\cite{C05},
and the jackknife estimator~\cite{SV84},
combined with the Shen-Chao-Lin unseen-species estimator~\cite{SCL03}.
%\subsection{Synthetic datasets}
%\label{sec:synthetic}
%In all our experiments, we choose the parameters of our
%estimators as follows.
%\begin{itemize}
%\item 
%For $\Upoir$, $r = \frac{1}{2t} \cdot \log (nt)$.
%\item 
%For $\UET$, $k = \lceil \log_2(nt)/2 \rceil$ and $q =
%  \frac{1}{1+t}$.
%\item 
%For $\Ubinkq$, $k = \lceil \log_3(nt)/2\rceil$ and $q =
%  \frac{2}{2+t}$.
%\end{itemize}
%We compare our estimators with the popular Shen-Chao-Lin
%method~\cite{SCL03} 
%that converts any support estimator %$\hat{S}$
%to an estimator for $U$.
%\[
%\Us{SC} = (\hat{S} - \prev{+}) \cdot \left(1 - \left(1 -
%\frac{\prev{1}}{n \cdot(\hat{S}-\prev{+})} \right)^m \right).
%\]
%We compare with Shen-Chao-Lin method for three popular support size
%estimators: .
We consider various natural synthetic distributions and established datasets.
Starting with the former, \arxiv{Figure~\ref{fig:synthetic}}{(Fig.\ 2)} shows the \emph{species discovery curve}, 
the prediction of $U$ as a function of $t$ of several
predictors for various distributions. 
%We compare the performance of the above estimators on data generated
%from the following six distributions supported on $k=10^6$ symbols: a)
%uniform distribution over $k$, b) step distribution with half the
%symbols with probability $\propto 1/2k$ and half the symbols with
%probability $3/2k$, c) Zipf distribution with parameter $1/2$, d) Zipf
%distribution with parameter $1$, e) a distribution generated according
%to Dirichlet($1$) prior and f) a distribution generated according to
%Dirichlet($1/2$) prior.
%All our experiments are averaged over $100$ trials.
% This experiment is repeated for $100$ trials and the results are
% recorded in in Figure~\ref{fig:synthetic}.  The black line is the
% average number of new symbols that appear in the next $nt$ samples.
The true value is shown in black, and the other estimators are
color coded, with the solid line representing their mean estimate,
and the shaded area corresponding to one standard deviation.
Note that the Chao-Lee and ACE estimators are designed specifically
for uniform distributions, hence in \arxiv{Figure~\ref{fig:uniform}}{(Fig.\ 2a)} they coincide with the
true value, but for all other distributions, our proposed smoothed Good-Toulmin
estimators outperform the existing ones.
%\subsection{Real datasets}
%\label{sec:real_data}
\begin{figure}
\centering
\subfigure[Uniform]{\label{fig:uniform}\arxiv{\includegraphics[width=60mm]{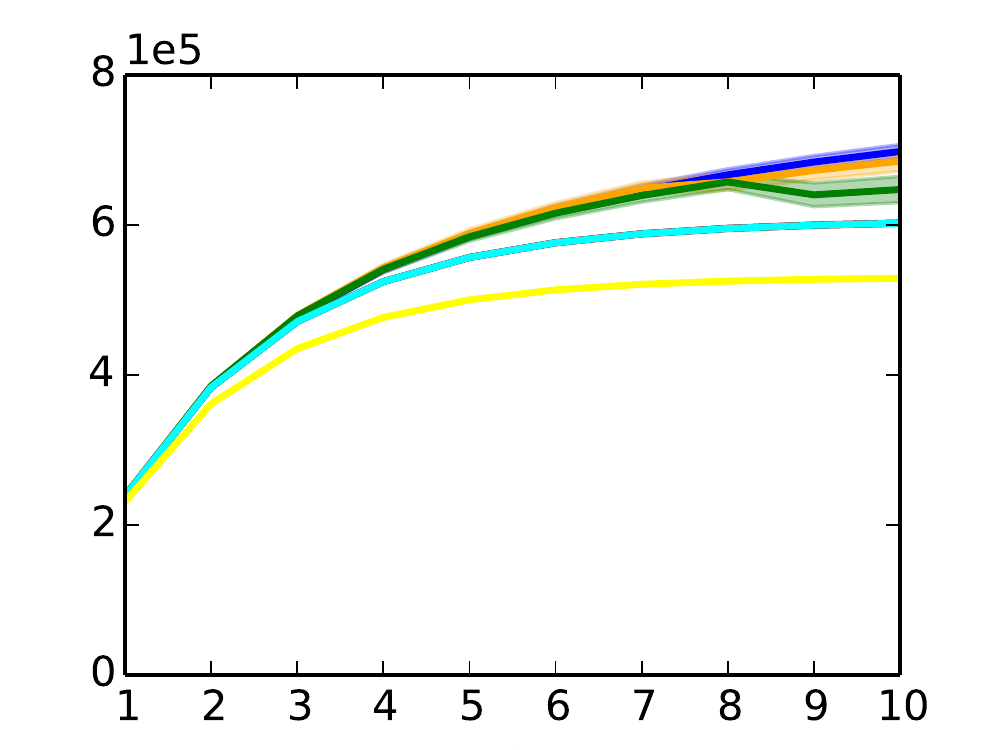}}{\includegraphics[width=60mm]{syn_uniform.pdf}}}
\subfigure[Two steps: $\frac{1}{2k} \times \frac{k}{2} \cup \frac{3}{2k} \times \frac{k}{2}$  ]{\label{fig:step}\arxiv{{\includegraphics[width=60mm]{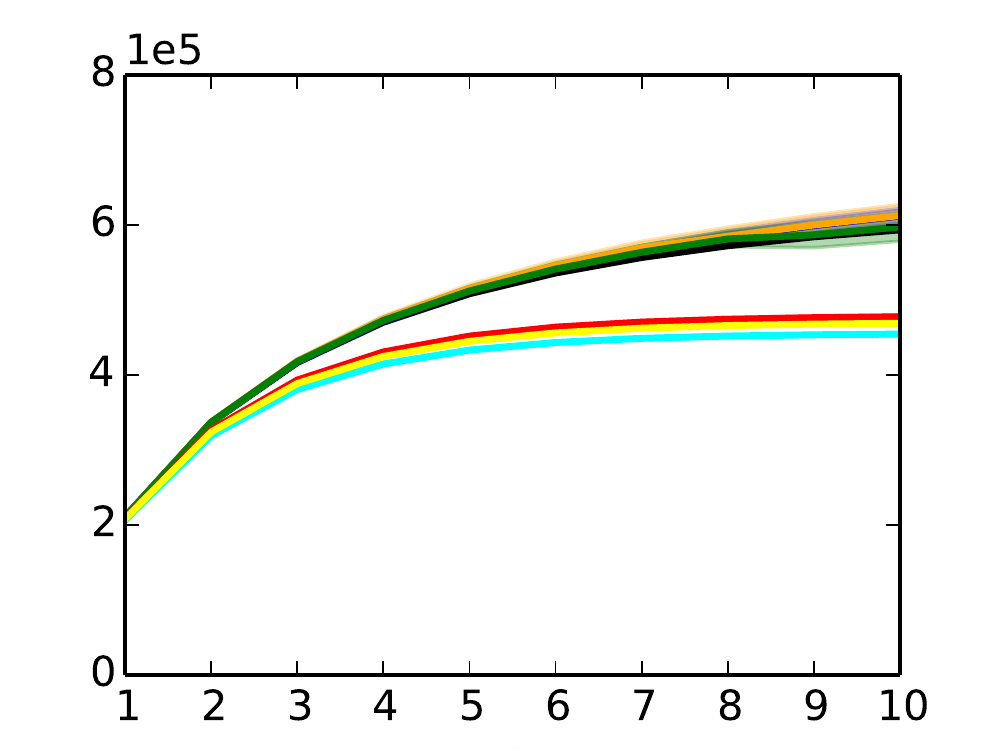}}}{\includegraphics[width=60mm]{syn_step.pdf}}}
\subfigure[Zipf-$1$: $p_i \propto \frac{1}{i}$]{\label{fig:zipf1}\arxiv{\includegraphics[width=60mm]{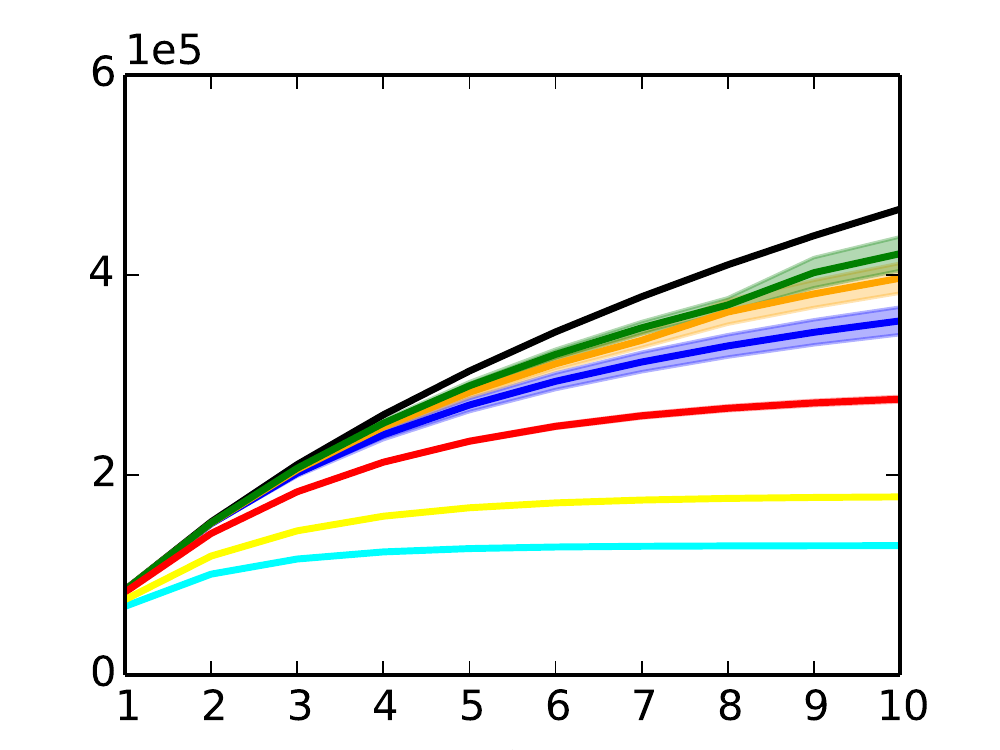}}{\includegraphics[width=60mm]{syn_zipf1.pdf}}}
\subfigure[Zipf-$1.5$: $p_i \propto \frac{1}{i^{1.5}}$]{\label{fig:zipf32}\arxiv{\includegraphics[width=60mm]{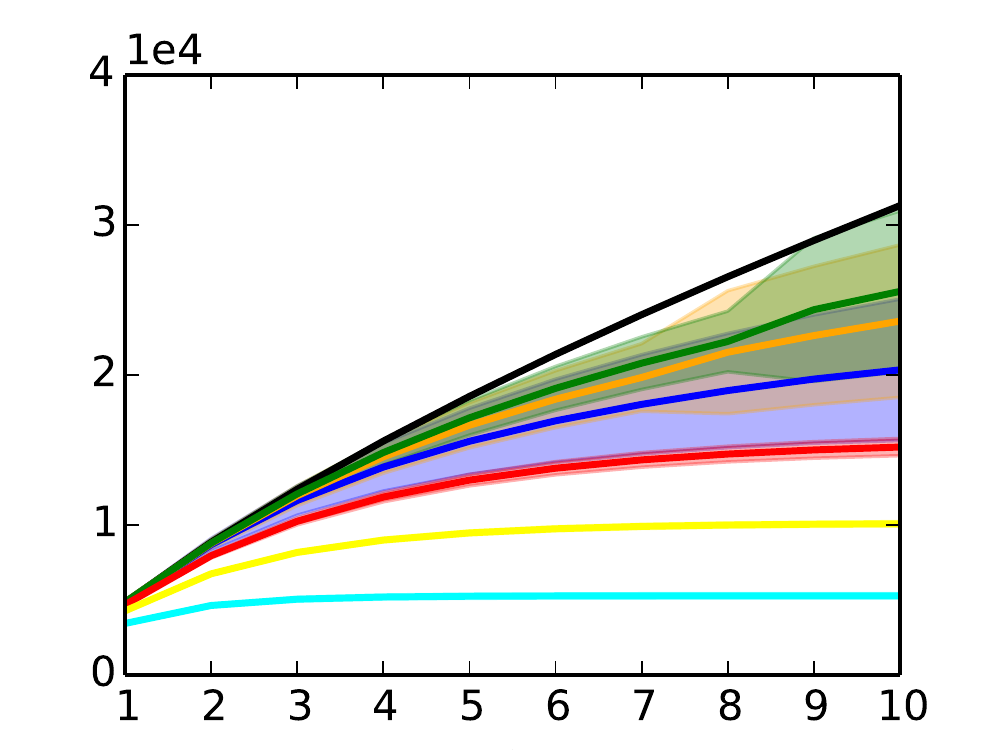}}{\includegraphics[width=60mm]{syn_zipf32.pdf}}}
\subfigure[Dirichlet-$1$ prior]{\label{fig:dir1}\arxiv{\includegraphics[width=60mm]{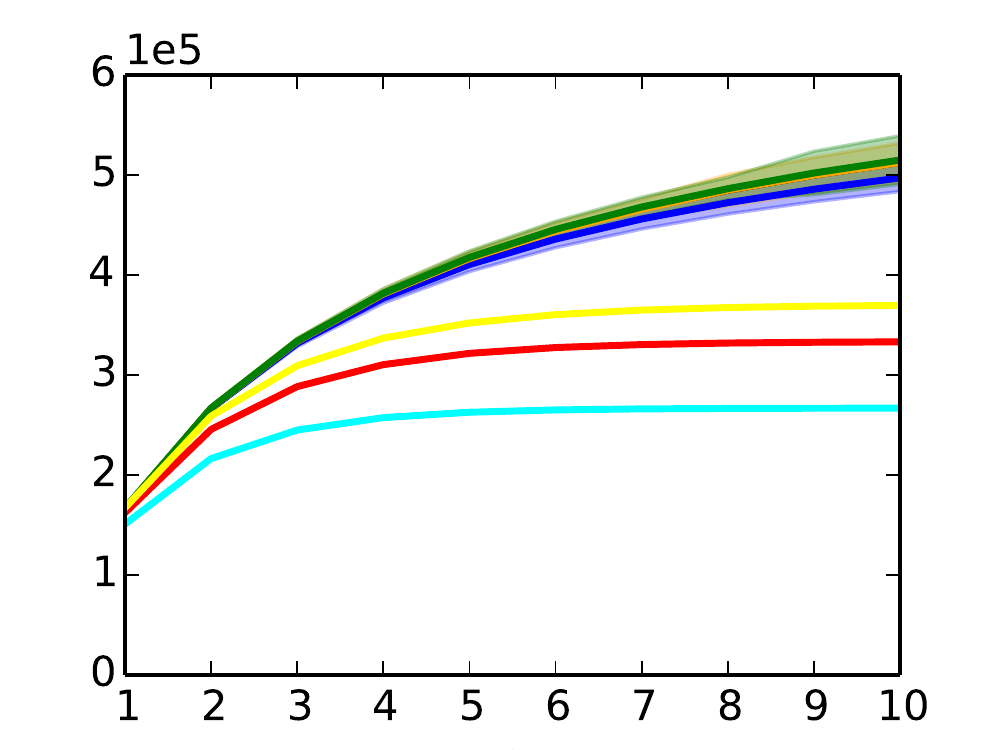}}{\includegraphics[width=60mm]{syn_dir1.pdf}}}
\subfigure[Dirichlet-$1/2$ prior]{\label{fig:dir12}\arxiv{\includegraphics[width=60mm]{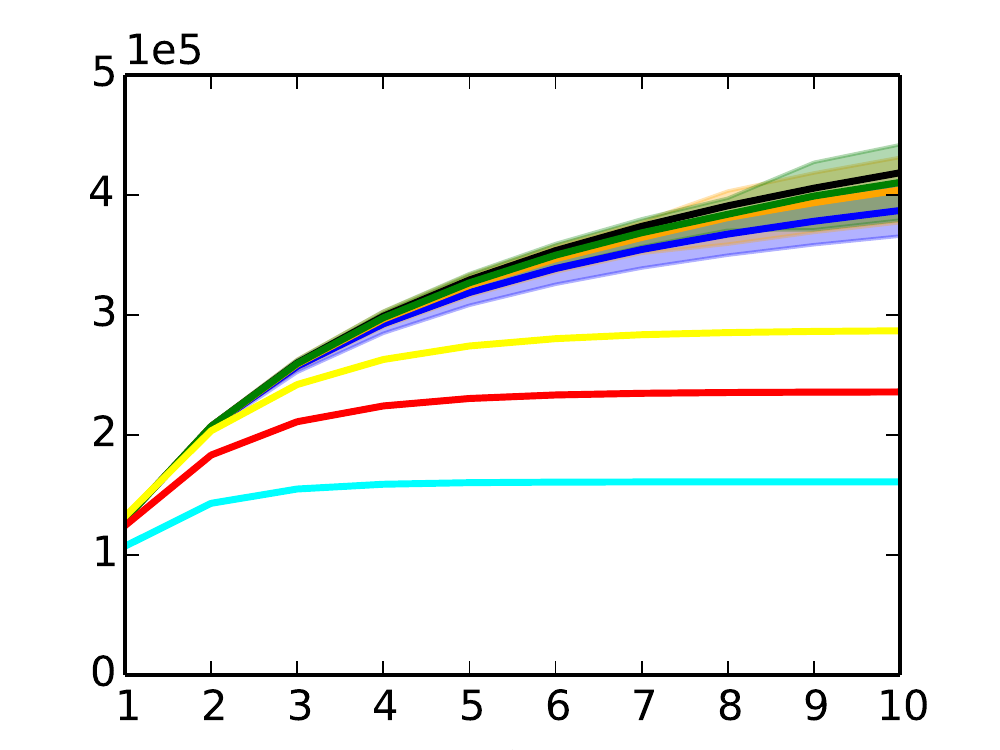}}{\includegraphics[width=60mm]{syn_dir12.pdf}}}
\begin{tabular}{| c                                                                                                                       | l l | l l |}
\hline
True value  & Previous & &  Proposed & \\ \hline 
 {\color{black} \rule[0.08cm]{0.05\textwidth}{1pt}}    & Chao-Lee  & {\color{myred} \rule[0.08cm]{0.05\textwidth}{1pt}} & Poisson smoothing  & {\color{myblue}\rule[0.08cm]{0.05\textwidth}{1pt}}  \\
       & ACE   & {\color{mycyan} \rule[0.08cm]{0.05\textwidth}{1pt}} & Binomial smoothing $q=1/(t+1)$ &  {\color{myorange} \rule[0.08cm]{0.05\textwidth}{1pt}}\\ 
       & Jackknife  & {\color{myyellow} \rule[0.08cm]{0.05\textwidth}{1pt}}  & Binomial smoothing $q = 2/(t+2)$ & {\color{mygreen} \rule[0.08cm]{0.05\textwidth}{1pt}} \\
\hline
\end{tabular}
\caption{Comparisons of the estimated number of unseen species 
as a function of $t$. All experiments have distribution support size $10^6$,
$n=5 \cdot 10^5$, and are averaged over $100$ iterations.}
\label{fig:synthetic}
\end{figure}
Of the proposed estimators, the binomial-smoothing estimator 
with parameter $q = \frac{2}{2+t}$ has a stronger theoretical
guarantee and performs slightly better than the others.
Hence when considering real data we plot only its performance
and compare it with the other state-of-the art estimators.
%Chao-Lee, ACE and
%Jackknife estimators described in \prettyref{sec:synthetic}.
%obtained from the Shen-Chao-Lin approach.  
We test the estimators on three real datasets taken from various
scientific applications where the samples size $n$ ranges from few
hundreds to a million. For all these date sets, our estimator 
outperforms the existing procedures.% \cite{CCG12}.

\begin{figure}
\centering
\subfigure[Hamlet random]{\label{fig:hamlet_random}\arxiv{\includegraphics[scale=0.35]{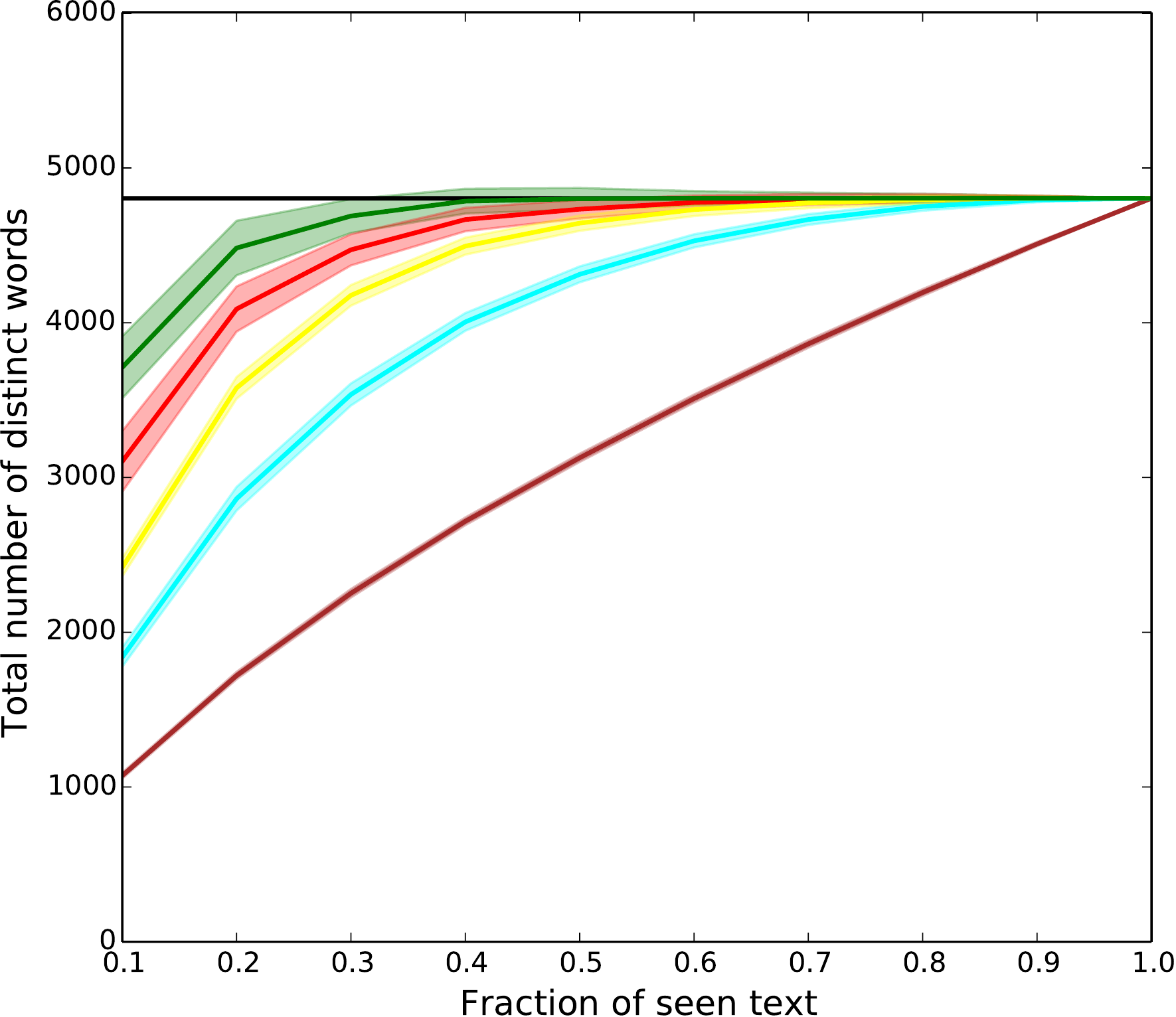}}{\includegraphics[scale=0.35]{Hamlet_data-crop.pdf}}}
\subfigure[Hamlet consecutive]{\label{fig:hamlet_consecutive}\arxiv{\includegraphics[scale=0.35]{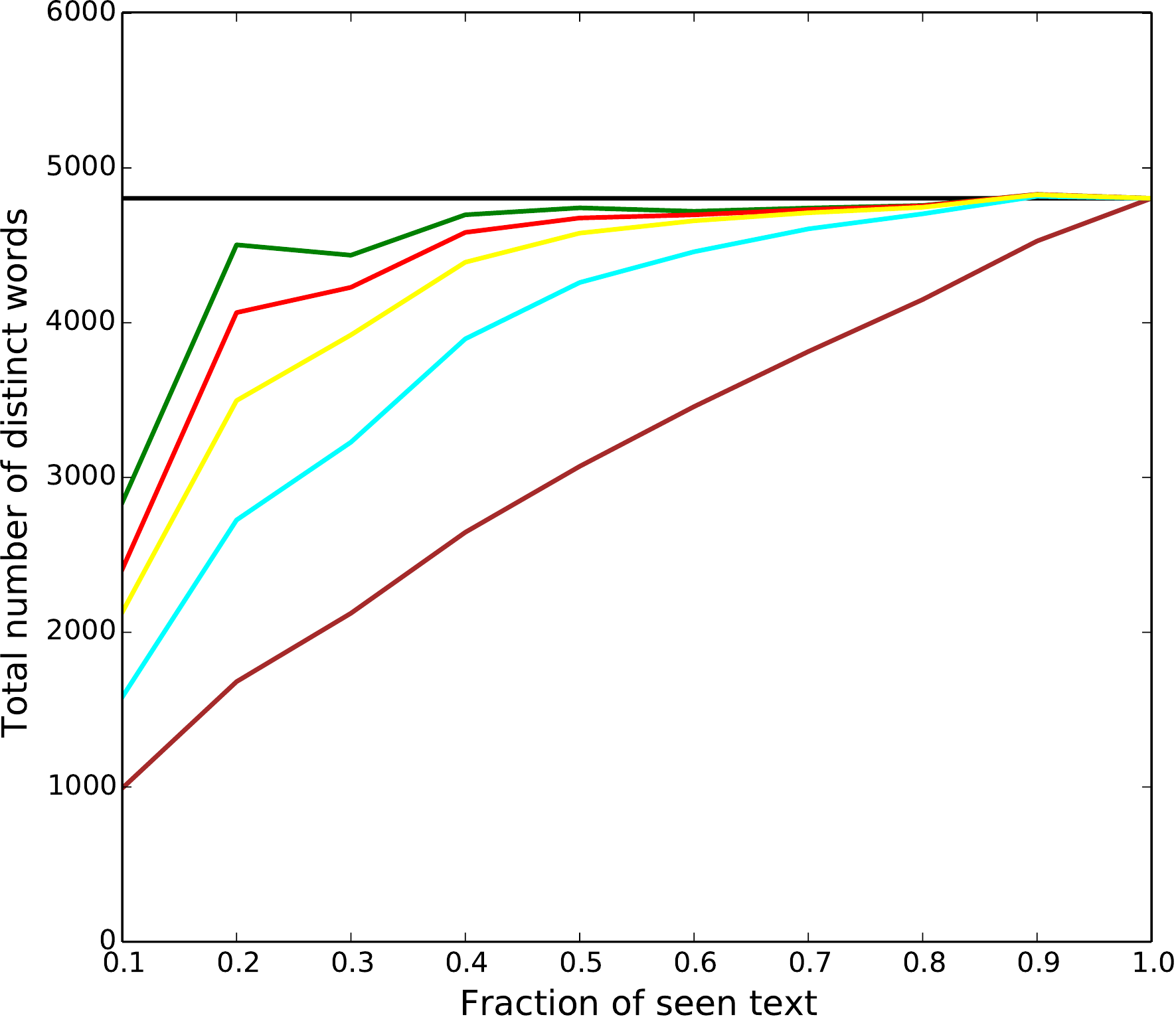}}{\includegraphics[scale=0.35]{Hamlet_consecutive-crop.pdf}}}
\subfigure[SLOTUs]{\label{fig:haplo}\arxiv{\includegraphics[scale=0.35]{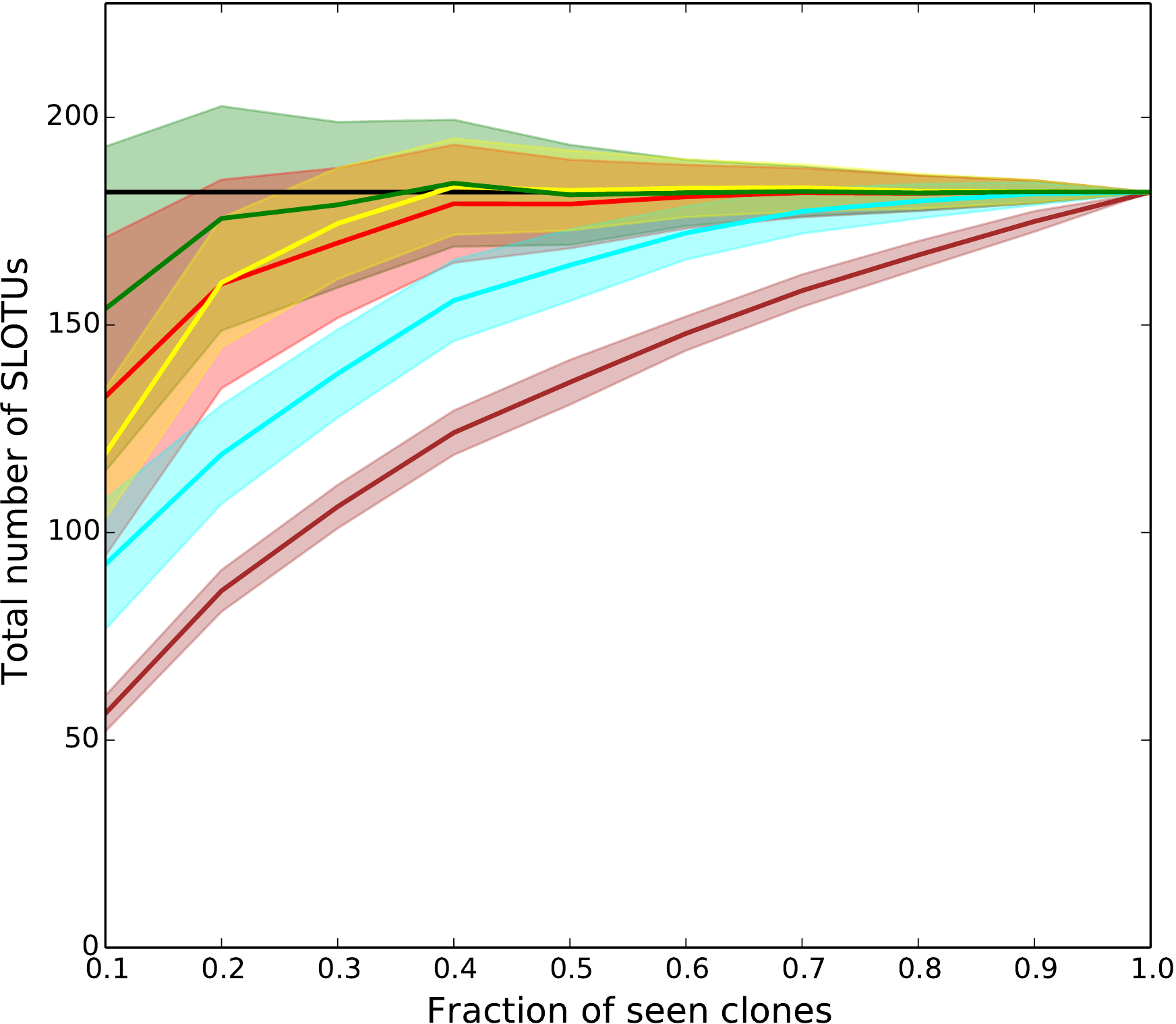}}{\includegraphics[scale=0.35]{Haplo_data-crop.pdf}}}
\subfigure[Last names]{\label{fig:name}\arxiv{\includegraphics[scale=0.35]{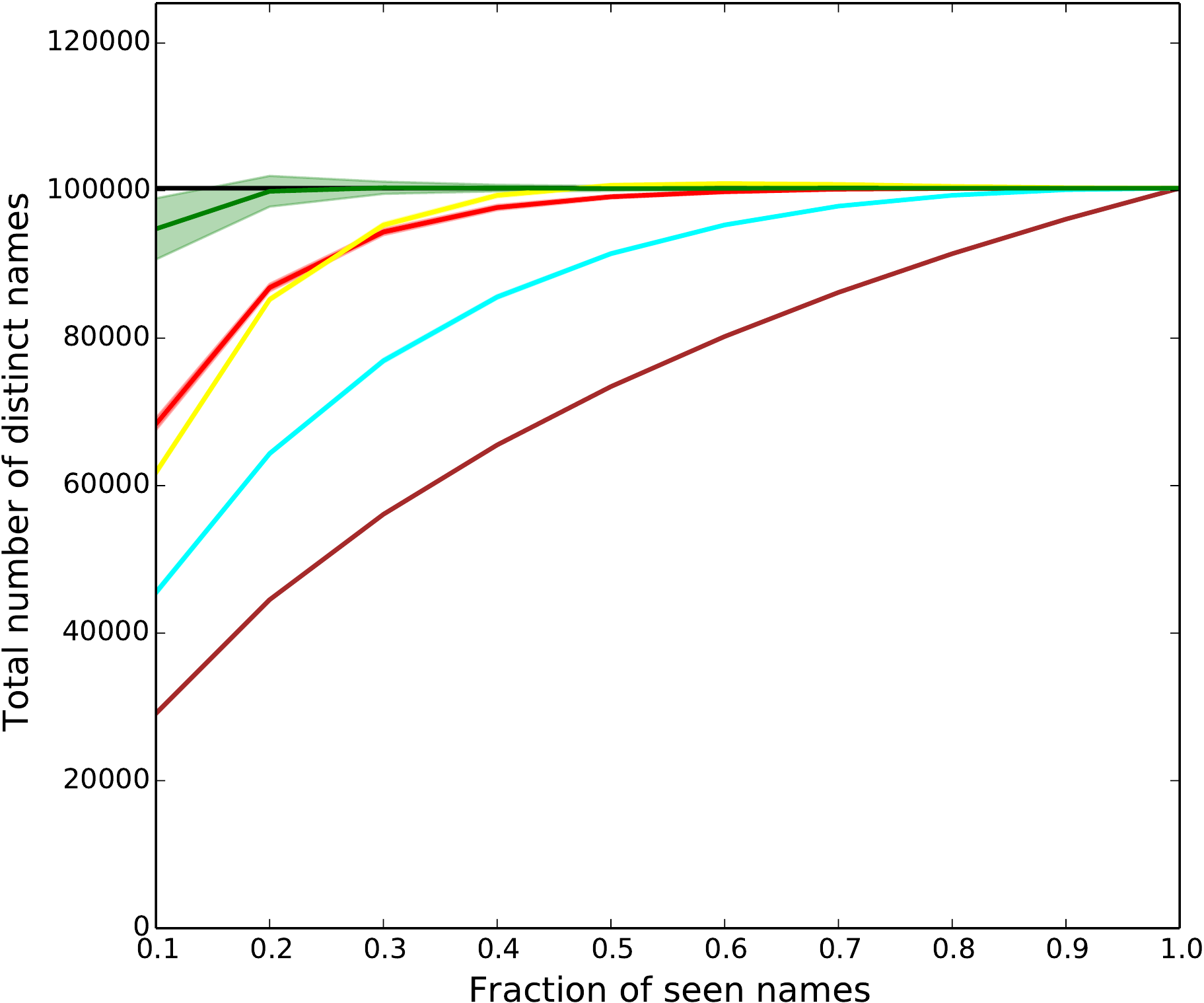}}{\includegraphics[scale=0.35]{Name_data-crop.pdf}}}
\begin{tabular}{| c                                                                                                                       | l l | l l |}
\hline
True value  & Previous & &  Proposed & \\ \hline 
 {\color{black} \rule[0.08cm]{0.05\textwidth}{1pt}}   
 & Chao-Lee  
& {\color{myred} \rule[0.08cm]{0.05\textwidth}{1pt}} 
& Binomial smoothing $q = 2/(t+2)$ & {\color{mygreen} \rule[0.08cm]{0.05\textwidth}{1pt}} \\
       & ACE   & {\color{mycyan} \rule[0.08cm]{0.05\textwidth}{1pt}}   & & \\ 
       & Jackknife  & {\color{myyellow} \rule[0.08cm]{0.05\textwidth}{1pt}}  &  & \\
& Empirical & {\color{mybrown} \rule[0.08cm]{0.05\textwidth}{1pt}}   & & \\ 
\hline
\end{tabular}
\caption{Estimates for number of: $(a)$ distinct words in Hamlet with
random sampling $(b)$ distinct words in Hamlet with consecutive
sampling $(c)$ SLOTUs on human skin $(d)$ last names.}
\label{fig:hamlet}
\end{figure}
%Based on the work by~\cite{SP99}, Shen-Chao-Lin~\cite{SCL03} proposed
%a technique to convert any support size estimator to an estimator for $U$
%and justified it for uniform distributions. This methodology together
%with popular support size estimators such as Chao-Lee~\cite{C84, CL92},
%abundance coverage estimator~\cite{C05}, and jackknife estimators are frequently
%used in ecology \cite{CCG12}.  As we demonstrate in
%our experiments, while this approach works well for uniform
%distributions, for other distributions such as Zipf distributions, it incurs a large bias.  

\arxiv{Figure~\ref{fig:hamlet_random}}{(Fig.\ 3a)} shows the first real-data experiment, predicting
vocabulary size based on partial text. 
Shakespeare's play \emph{Hamlet} consists of $\ntotal=31999$ words,
of which $4804$ are distinct.
We randomly select $n$ of the $\ntotal$ words without replacement,
predict the number of unseen words in $\ntotal-n$ new ones,
and add it to those observed.
The results shown are averaged over $100$ trials.
Observe that the new estimator outperforms existing ones and
that as little as $20\%$ of the data 
%which corresponds to $t=4$, 
already yields an accurate estimate of the total number of
distinct words. 
\arxiv{Figure~\ref{fig:hamlet_consecutive}}{(Fig.\ 3b)} repeats the experiment but instead of random sampling,
uses the first $n$ consecutive words, with similar conclusions.

\arxiv{Figure~\ref{fig:haplo}}{(Fig.\ 3c)} estimates the number of bacterial species on the human skin.
\cite{GCB07} considered forearm skin biota of six subjects.
They identified $\ntotal=1221$ clones consisting
of $182$ different species-level operational taxonomic units (SLOTUs).
%diversity of the skin biota from the forearms of six healthy subjects
%by using broad-range small subunit rRNA genes PCR-based sequencing of
%randomly selected $n_{\text{total}} = 1221$ clones. Of these $1221$
%clones, they identified $182$ species-level operational taxonomic
%units (SLOTUs). 
As before, we select $n$ out of the $\ntotal$ clones without replacement and
predict the number of distinct SLOTUs found.
Again the estimates are more accurate than those of existing
estimators and are reasonably accurate already with $20\%$ of the data.
% the our estimates are reasonably accurate.
%  The results averaged over $100$ trials 
%are in Figure~\ref{fig:haplo}.
%Notice that with about $20\%$ of the data, which corresponds to $t =
%4$, we are able to predict the total number of SLOTUs in the corpora
%reasonably accurately.

Finally, \arxiv{Figure~\ref{fig:name}}{(Fig.\ 3d)} considers the 2000 United States
 Census~\cite{Census00}, which lists all U.S. last names
corresponding to at least 100 individuals.
With these many repetitions, even just a small fraction of the data
will cover all names, hence we 
%Notice that in this
%data set, since every symbol already appears $100$ times, if we
%randomly observe $10\%$ of the data we would have already seen most of
%the names. Hence, to make the problem interesting, 
first subsampled the data $\ntotal=10^6$ and obtained a list of 
100328 distinct last names. 
As before we estimate for this number
using $n$ randomly chosen names, again with similar conclusions.
%For this dataset, our estimator
%significantly outperform its other estimators.
\section*{Acknowledgments}
\label{sec:acknowledgements}
This work was partially completed while the authors were visiting the
Simons Institute for the Theory of Computing at UC Berkeley, whose
support is gratefully acknowledged.  We thank Dimitris Achlioptas,
David
Tse, Chi-Hong Tseng, and Jinye Zhang for helpful discussions and comments.
\appendix
\section{Proof of \prettyref{lmm:UGT-ell}}
	\label{app:ell}
\begin{proof}
To rigorously prove an impossibility result for the truncated GT estimator, we demonstrate a particular 
distribution under which the bias is large.  Consider the uniform
distribution over $n/({\fixed}+1)$ symbols, where $\ell$ is a non-zero even integer. By \prettyref{lem:general_bounds}, for
this distribution the bias is
\begin{align*}
\Expect[U-U^\ell]
& =\sum_x e^{-\lambda_x} (1-e^{-\lambda_xt} - h(\lambda_x))\\
& = \frac{n}{{\fixed}+1} e^{-({\fixed}+1)} \left(1-e^{-({\fixed}+1)t} + \sum^{\fixed}_{i=1} \frac{(-({\fixed}+1)t)^i}{i!} \right) \\
& \geq \frac{n}{{\fixed}+1} e^{-({\fixed}+1)} \left(\sum^{\fixed}_{i=1} \frac{(-({\fixed}+1)t)^i}{i!} \right) \\
& \stackrel{(a)}{\geq} \frac{n}{{\fixed}+1} e^{-({\fixed}+1)} \left( \frac{(({\fixed}+1)t)^\fixed}{\fixed!} -  \frac{(({\fixed}+1)t)^{\fixed-1}}{(\fixed-1)!} \right) \\
& \geq \frac{n}{({\fixed}+1)} e^{-({\fixed+1})}\frac{(({\fixed}+1)t)^\fixed}{\fixed!} \cdot \frac{(t-1)}{t}  \\
& \geq \frac{n}{3( \fixed+1)^{3/2}} t^{\fixed} \frac{(t-1)}{t} \geq \frac{n}{3 \cdot 2^{3/2}} \frac{t^{\ell}}{\ell^{3/2}} \frac{(t-1)}{t} ,
%& \geq \min \left(\frac{n}{6}, \frac{n}{6( \fixed+1)^{3/2}} t^{\fixed} \right) \\
%\stackrel{(b)}{\gg} n,
\end{align*}
where $(a)$ follows from the fact that $\frac{(-({\fixed}+1)t)^i}{i!}$
for $i =1,\ldots,\ell$ is an alternating series with increasing 
magnitude of terms.
%For $\ell = 0$, the above bound is $n/6$. For $\ell \geq 1$,
Hence % for any even value $\ell > 0 $,
%\[
%\frac{t^{\ell}}{(\ell+1)^{3/2}} \geq \frac{1}{2^{3/2}} 
%\frac{t^{\ell}}{\ell^{3/2}}, % \geq \left(\frac{e \log t}{3} \right)^{3/2},
%\]
%%where the last inequality follows from the fact that $e^{y}/y$ is minimized at $y = 1$.
%and
\[
\Expect[U-U^\ell]
\geq %\frac{n}{3}  \left(\frac{e \log t}{3} \right)^{3/2} \frac{(t-1)}{t}.
\frac{n}{3\cdot 2^{3/2}}   \frac{(t-1)}{t} \min_{\ell\in \{2,4,\ldots\}} 
\frac{t^{\ell}}{\ell^{3/2}}.
\]
For $t \geq 2$, the above minimum occurs at $\ell=2$ and hence
$ \min_{\ell\in \{2,4,\ldots\}}
\frac{t^{\ell}}{\ell^{3/2}} \geq \frac{(t-1)^{3/2}}{2^{3/2}}$.
For $1<t < 2$, using the fact that $ e^{y}\geq e y $ for $y>0$ and $\log t \geq (t-1)\log 2$ for $1<t<2$, we have $
 \min_{\ell\in \{2,4,\ldots\}}
\frac{t^{\ell}}{\ell^{3/2}} \geq (\frac{2e \log t}{3} )^{3/2} \geq 
(\frac{2 e \log 2 (t-1)}{3} )^{3/2}$.
Thus for any even value of $\ell >0$,
\[
\Expect[U-U^\ell]
\geq %\frac{n}{3}  \left(\frac{e \log t}{3} \right)^{3/2} \frac{(t-1)}{t}.
%\frac{n}{24}   \frac{(t-1)}{t} (t-1)^{3/2}
 \frac{n (t-1)^{5/2}}{6.05 t}.
\]
 A similar argument holds for odd values of
$\ell$ and $\ell = 0$, showing that $|\Expect[U-U^\ell]| \gtrsim  \frac{n (t-1)^{5/2}}{t}$ and hence the desired NMSE bound.
%Thus the problem with Taylor series approximation
%% or if we exactly match the function $h$ for derivatives values $0$,
%is that for $\lambda_x \approx n/({\fixed}+1)$, the bias is at least a
%constant times $n(t-1)^{5/2}/t$ and hence for any $n$ and $t > 1$, the NMSE is at
%least a constant.
\end{proof}

%\section{Introduction}
%\label{sec:intro}
%\input{introduction}
%\section{Approach and results}
%\input{results}
%\item Details of experiments
%Finally, in 
%\input{definitions}
%\input{prelim}
%\input{upper}
%\input{extensions}
%\input{multinomial}
%\input{bernoulli}
%\input{hyper}
%\input{lower}
%\section{Experiments}
%\input{experiments}
%\input{acknowledgements}
%\input{app-ell}

%\bibliographystyle{alpha}
%\bibliography{strings,justoneref}

\begin{thebibliography}{JVHW15}

\bibitem[AS64]{AS64}
M.~Abramowitz and I.~A. Stegun.
\newblock {\em {Handbook of mathematical functions with formulas, graphs, and
  mathematical tables}}.
\newblock Wiley-Interscience, New York, NY, 1964.

\bibitem[BF93]{BF93}
John Bunge and M~Fitzpatrick.
\newblock Estimating the number of species: a review.
\newblock {\em Journal of the American Statistical Association},
  88(421):364--373, 1993.

\bibitem[BH84]{BH84}
A.~D. Barbour and Peter Hall.
\newblock On the rate of poisson convergence.
\newblock {\em Mathematical Proceedings of the Cambridge Philosophical
  Society}, 95:473--480, 5 1984.

\bibitem[Bur14]{Census00}
United States~Census Bureau.
\newblock Frequently occurring surnames from the {C}ensus 2000, 2014.

\bibitem[CCG{\etalchar{+}}12]{CCG12}
Robert~K Colwell, Anne Chao, Nicholas~J Gotelli, Shang-Yi Lin, Chang~Xuan Mao,
  Robin~L Chazdon, and John~T Longino.
\newblock Models and estimators linking individual-based and sample-based
  rarefaction, extrapolation and comparison of assemblages.
\newblock {\em Journal of Plant Ecology}, 5(1):3--21, 2012.

\bibitem[Cha84]{C84}
Anne Chao.
\newblock Nonparametric estimation of the number of classes in a population.
\newblock {\em Scandinavian Journal of statistics}, pages 265--270, 1984.

\bibitem[Cha05]{C05}
Anne Chao.
\newblock Species estimation and applications.
\newblock {\em Encyclopedia of statistical sciences}, 2005.

\bibitem[CL92]{CL92}
Anne Chao and Shen-Ming Lee.
\newblock Estimating the number of classes via sample coverage.
\newblock {\em Journal of the American statistical Association},
  87(417):210--217, 1992.

\bibitem[CL11]{CL11}
T.T. Cai and M.~G. Low.
\newblock Testing composite hypotheses, {Hermite} polynomials and optimal
  estimation of a nonsmooth functional.
\newblock {\em The Annals of Statistics}, 39(2):1012--1041, 2011.

\bibitem[DF80]{DF80}
P.~Diaconis and D.~Freedman.
\newblock Finite exchangeable sequences.
\newblock {\em Ann. Probab.}, 8(4):745--764, 08 1980.

\bibitem[ET76]{ET76}
B.~Efron and R.~Thisted.
\newblock Estimating the number of unseen species: How many words did
  shakespeare know?
\newblock {\em Biometrika}, 63(3):435--447, 1976.

\bibitem[FCW43]{FCW43}
Ronald~Aylmer Fisher, A~Steven Corbet, and Carrington~B Williams.
\newblock The relation between the number of species and the number of
  individuals in a random sample of an animal population.
\newblock {\em The Journal of Animal Ecology}, pages 42--58, 1943.

\bibitem[FH07]{FH07}
Dinei Florencio and Cormac Herley.
\newblock A large-scale study of web password habits.
\newblock In {\em Proceedings of the 16th international conference on World
  Wide Web}, pages 657--666. ACM, 2007.

\bibitem[Goo53]{G53}
Irving~John Good.
\newblock The population frequencies of species and the estimation of
  population parameters.
\newblock {\em Biometrika}, 40(3-4):237--264, 1953.

\bibitem[GT56]{GT56}
I.J. Good and G.H. Toulmin.
\newblock The number of new species, and the increase in population coverage,
  when a sample is increased.
\newblock {\em Biometrika}, 43(1-2):45--63, 1956.

\bibitem[GTPB07]{GCB07}
Zhan Gao, Chi-hong Tseng, Zhiheng Pei, and Martin~J Blaser.
\newblock Molecular analysis of human forearm superficial skin bacterial biota.
\newblock {\em Proceedings of the National Academy of Sciences},
  104(8):2927--2932, 2007.

\bibitem[HHRB01]{HHJTB01}
Jennifer~B Hughes, Jessica~J Hellmann, Taylor~H Ricketts, and Brendan~JM
  Bohannan.
\newblock Counting the uncountable: statistical approaches to estimating
  microbial diversity.
\newblock {\em Applied and environmental microbiology}, 67(10):4399--4406,
  2001.

\bibitem[HNSS95]{HNSS95}
Peter~J Haas, Jeffrey~F Naughton, S~Seshadri, and Lynne Stokes.
\newblock Sampling-based estimation of the number of distinct values of an
  attribute.
\newblock In {\em VLDB}, volume~95, pages 311--322, 1995.

\bibitem[ILLL09]{ICL09}
Iuliana Ionita-Laza, Christoph Lange, and Nan~M Laird.
\newblock Estimating the number of unseen variants in the human genome.
\newblock {\em Proceedings of the National Academy of Sciences},
  106(13):5008--5013, 2009.

\bibitem[JVHW15]{JVHW15}
Jiantao Jiao, Kartik Venkat, Yanjun Han, and Tsachy Weissman.
\newblock Minimax estimation of functionals of discrete distributions.
\newblock {\em IEEE Transactions on Information Theory}, 61(5):2835--2885,
  2015.

\bibitem[KLR99]{KPD99}
Ian Kroes, Paul~W Lepp, and David~A Relman.
\newblock Bacterial diversity within the human subgingival crevice.
\newblock {\em Proceedings of the National Academy of Sciences},
  96(25):14547--14552, 1999.

\bibitem[Kol86]{K86}
Gina Kolata.
\newblock Shakespeare's new poem: An ode to statistics.
\newblock {\em Science (New York, NY)}, 231(4736):335, 1986.

\bibitem[LNS99]{LNS99}
Oleg Lepski, Arkady Nemirovski, and Vladimir Spokoiny.
\newblock On estimation of the {$L_r$} norm of a regression function.
\newblock {\em Probability Theory and Related Fields}, 113(2):221--253, 1999.

\bibitem[MS00]{MS00}
David~A. McAllester and Robert~E. Schapire.
\newblock On the convergence rate of good-turing estimators.
\newblock In {\em In the Proc. of the Conference on Learning Theory}, pages
  1--6, 2000.

\bibitem[OS15]{OS15}
Alon Orlitsky and Ananda~Theertha Suresh.
\newblock Competitive distribution estimation: Why is {G}ood-{T}uring good.
\newblock In {\em Advances in Neural Information Processing Systems}, pages
  2134--2142, 2015.

\bibitem[PBG{\etalchar{+}}01]{PBG01}
Bruce~J Paster, Susan~K Boches, Jamie~L Galvin, Rebecca~E Ericson, Carol~N Lau,
  Valerie~A Levanos, Ashish Sahasrabudhe, and Floyd~E Dewhirst.
\newblock Bacterial diversity in human subgingival plaque.
\newblock {\em Journal of bacteriology}, 183(12):3770--3783, 2001.

\bibitem[RCS{\etalchar{+}}09]{R09}
Harlan~S Robins, Paulo~V Campregher, Santosh~K Srivastava, Abigail Wacher,
  Cameron~J Turtle, Orsalem Kahsai, Stanley~R Riddell, Edus~H Warren, and
  Christopher~S Carlson.
\newblock Comprehensive assessment of {T}-cell receptor $\beta$-chain diversity
  in $\alpha$$\beta$ t cells.
\newblock {\em Blood}, 114(19):4099--4107, 2009.

\bibitem[Rob68]{R68}
Herbert~E Robbins.
\newblock Estimating the total probability of the unobserved outcomes of an
  experiment.
\newblock {\em The Annals of Mathematical Statistics}, 39(1):256--257, 1968.

\bibitem[RRSS09]{RRSS09}
Sofya Raskhodnikova, Dana Ron, Amir Shpilka, and Adam Smith.
\newblock Strong lower bounds for approximating distribution support size and
  the distinct elements problem.
\newblock {\em SIAM Journal on Computing}, 39(3):813--842, 2009.

\bibitem[SCL03]{SCL03}
Tsung-Jen Shen, Anne Chao, and Chih-Feng Lin.
\newblock Predicting the number of new species in further taxonomic sampling.
\newblock {\em Ecology}, 84(3):798--804, 2003.

\bibitem[sta92]{stanfordF}
Stanford statistics department brochure, 1992.
\newblock
  \url{https://statistics.stanford.edu/sites/default/files/1992_StanfordStatisticsBrochure.pdf}.

\bibitem[Ste86]{S86}
J.~Michael Steele.
\newblock An {E}fron-{S}tein inequality for nonsymmetric statistics.
\newblock {\em Ann. Statist.}, 14(2):753--758, 06 1986.

\bibitem[SvB84]{SV84}
Eric~P Smith and Gerald van Belle.
\newblock Nonparametric estimation of species richness.
\newblock {\em Biometrics}, pages 119--129, 1984.

\bibitem[TE87]{TE87}
Ronald Thisted and Bradley Efron.
\newblock Did shakespeare write a newly-discovered poem?
\newblock {\em Biometrika}, 74(3):445--455, 1987.

\bibitem[VV11]{VV11}
Gregory Valiant and Paul Valiant.
\newblock Estimating the unseen: an $n/\log (n)$-sample estimator for entropy
  and support size, shown optimal via new {CLT}s.
\newblock In {\em Proceedings of the 43rd annual ACM symposium on Theory of
  computing}, pages 685--694, 2011.

\bibitem[VV13]{VV13}
Paul Valiant and Gregory Valiant.
\newblock Estimating the unseen: Improved estimators for entropy and other
  properties.
\newblock In {\em Advances in Neural Information Processing Systems}, pages
  2157--2165, 2013.

\bibitem[VV15]{VV15}
Gregory Valiant and Paul Valiant.
\newblock Instance optimal learning.
\newblock {\em arXiv preprint arXiv:1504.05321}, 2015.

\bibitem[WY15a]{WY14b}
Yihong Wu and Pengkun Yang.
\newblock Chebyshev polynomials, moment matching, and optimal estimation of the
  unseen.
\newblock {\em preprint arxiv:1504.01227}, Apr. 2015.

\bibitem[WY15b]{WY14}
Yihong Wu and Pengkun Yang.
\newblock Minimax rates of entropy estimation on large alphabets via best
  polynomial approximation.
\newblock {\em to appear in IEEE Transactions on Information Theory,
  arxiv:1407.0381}, Jul 2015.

\end{thebibliography}

\newcommand{\etalchar}[1]{$^{#1}$}

\end{document}